\documentclass[reqno,a4paper,11pt,titlepage]{amsart}
\usepackage[foot]{amsaddr}
\usepackage{amssymb}
\usepackage{amsthm}
\usepackage{amsmath}
\usepackage[utf8]{inputenc}
\usepackage[margin=1.0in]{geometry}
\usepackage{pst-node}
\usepackage{verbatim}
\usepackage{tikz-cd} 
\newtheorem{theorem}{Theorem}[section]
\newtheorem{lemma}[theorem]{Lemma}
\newtheorem{proposition}[theorem]{Proposition}
\newtheorem{corollary}[theorem]{Corollary}

\newtheorem{definition}[theorem]{Definition}

\theoremstyle{remark}
\newtheorem{remark}[theorem]{Remark}

\usepackage[T1]{fontenc}
\catcode`@=11
\def\underarrow#1{\mathop{\vtop{\m@th\ialign{##\crcr
$\hfil\displaystyle{#1}\hfil$\crcr
\noalign{\kern3pt\nointerlineskip}
\hfil$\uparrow$\hfil\crcr\noalign{\kern3pt}}}}\limits}
\catcode`@=12

\usepackage[english]{babel}
\usepackage[utf8]{inputenc}

\usepackage{amsfonts}
\def\lim{\mathop{\rm lim}\nolimits}
\def\rank{\mathop{\rm rank}\nolimits}
\def\colim{\mathop{\rm colim}\nolimits}
\def\Spec{\mathop{\rm Spec}}
\def\Hom{\mathop{\rm Hom}\nolimits}

\def\Tors{\mathop{\rm Tors}\nolimits}

\def\Sh{\mathop{\textit{Sh}}\nolimits}

\def\B+{{B^+_{{\rm dR}}}}
\def\BdR{{B_{{\rm dR}}}}

\newcommand\cInd{\mathop{\mbox{$c$-$\mathrm{Ind}$}}}

\newcommand{\Syl}{\mathrm{Syl}}
\newcommand{\G}{\mathrm{G}}
\newcommand{\T}{\mathrm{T}}
\newcommand{\rP}{\mathrm{P}}
\newcommand{\rH}{\mathrm{H}}
\newcommand{\rN}{\mathrm{N}}

\newcommand{\M}{\mathrm{M}}

\newcommand{\U}{\mathrm{U}}

\newcommand{\Z}{\mathbb{Z}}

\newcommand{\R}{\mathbb{R}}
\newcommand{\C}{\mathbb{C}}
\newcommand{\bL}{\mathbb{L}}

\newcommand{\Wchi}{\mathcal{W}(\chi)}
 \newcommand{\Fs}{\mathcal{F}_{\mathcal{W}(\Id)}}

\newcommand{\Res}{\mathrm{Res}}
\newcommand{\Irr}{\mathrm{Irr}}

\newcommand{\ind}{\mathrm{ind}}

\newcommand{\GL}{\mathrm{GL}}

\newcommand{\id}{\mathrm{id}}
\newcommand{\Id}{\mathrm{Id}}
\newcommand{\End}{\mathrm{End}}

\newcommand{\Gm}{{\mathbb{G}_m}}

\newcommand{\E}{\mathcal{E}}
\newcommand{\OO}{\mathcal{O}}
\newcommand{\mc}{\mathcal}

\newcommand{\Q}{\mathbb{Q}}

\newcommand{\F}{\mathbb{F}}

\newcommand{\bb}{\mathbb}

\newcommand{\Ig}{\mathrm{Ig}}

\newcommand{\Lie}{\mathrm{Lie}}

\newcommand{\redu}{\mathrm{red}}

\newcommand{\dom}{\mathrm{dom}}
\newcommand{\Std}{\mathrm{Std}}

\newcommand{\ov}{\overline}

\newcommand{\Cent}{\mathrm{Cent}}
\newcommand{\A}{\mathbb{A}}
\newcommand{\Gal}{\mathrm{Gal}}

\newcommand{\GSp}{\mathrm{GSp}}

\newcommand{\I}{\mathrm{I}}

\newcommand{\Mult}{\mathrm{Mult}^{\mathrm{sc}}}
\newcommand{\N}{\mathbb{N}}

\newcommand{\std}{\mathrm{std}}
\newcommand{\Ind}{\mathrm{Ind}}

\newcommand{\Sht}{\mathrm{Sht}}
\newcommand{\Spd}{\mathrm{Spd}}

\newcommand{\opp}{\mathrm{opp}}
\newcommand{\Act}{\mathrm{Act}}
\newcommand{\Bun}{\mathrm{Bun}}
\newcommand{\BunG}{\mathrm{Bun_{\G}}}
\newcommand{\Bunn}{\mathrm{Bun_{n}}}
\newcommand{\Rep}{\mathrm{Rep}}

\newcommand{\Aut}{\mathrm{Aut}}
\newcommand{\Dc}{\mathrm{D}}
\newcommand{\Dlis}{\mathrm{D}_{\mathrm{lis}}}

\newcommand{\Perf}{\mathrm{Perf}}
\newcommand{\Coh}{\mathrm{Coh}}
\newcommand{\Nilp}{\mathrm{Nilp}}

\newcommand{\IndPerf}{\mathrm{IndPerf}}

\newcommand{\ad}{\mathrm{ad}}

\newcommand{\rB}{\mathrm{B}}

\newcommand{\rZ}{\mathrm{Z}}

\mathchardef\mhyphen="2D

\def\ol{\overline}
\def\ra{\rightarrow}
\def\Div{\mathrm{Div}}

\usepackage{hyperref}

\title{Categorical local Langlands and torsion classes of some Shimura varieties}
\author{Kieu Hieu Nguyen}
\email{kieu-hieu.nguyen@uvsq.fr}
\address{University of Versailles Saint-Quentin, France}
\thanks{
\small{MSC class:	11S37}
}
\thanks{
\small{
The author acknowledges support by the European Union ERC in form of Consolidator Grants : RELANTRA, project number 101044930. Views and opinions expressed are however those of the author only and do not necessarily reflect those of the European Union or the European Research Council. Neither the European Union nor the granting authority can be held responsible for them.
 }
}

\usepackage[english]{babel}
\usepackage[utf8]{inputenc}
\begin{document}

\begin{abstract}
We study the cohomology of various local Shimura varieties for $\GL_n$. This provides an explicit description of the spectral action constructed by Fargues-Scholze in certain cases and allows us to prove some strongly generic part of the categorical local Langlands conjecture for $\GL_n$ with torsion coefficients. As applications, we are able to prove an analogue of the Harris-Viehmann conjecture and deduce new vanishing results for the cohomology of Shimura varieties of type $A$ in the torsion coefficient setting.
           
\end{abstract}

\maketitle
\tableofcontents
\section{Introduction}

Let $ \ell \neq p$ be primes and fix an algebraic closure $\overline{\Q}_p$ of $\Q_p$ with associated Weil group $W_{\Q_p}$. In \cite{FS}, Fargues and Scholze formulated the categorical local Langlands program for a reductive group $\G$ defined over $\Q_p$. The aim of this program is to relate sheaves on the stack of Langlands parameters, "the spectral side", to sheaves on $\Bun_{\G}$, "the geometric side". On the spectral side, one considers $\Perf^{\mathrm{qc}}( [Z^1(W_{\Q_p}, \widehat{\G})_{\Lambda}/\widehat{\G}] )$ (resp. $\Dc^{b,\mathrm{qc}}_{\mathrm{coh}}( [Z^1(W_{\Q_p}, \widehat{\G})_{\Lambda}/\widehat{\G}]) $), the derived category of perfect complexes on the stack of $L$-parameters $ [Z^1(W_{\Q_p}, \widehat{\G})_{\Lambda}/\widehat{\G}] $ with quasi-compact support (resp. bounded derived category of sheaves with coherent cohomology with quasi-compact support), as in \cite{DH,Zhu1} and \cite[Section~VIII.I]{FS}. On the geometric side, one considers $\Dlis(\Bun_{\G},\Lambda)$, the category of lisse-\'etale sheaves and $\Dlis(\Bun_{\G},\Lambda)^{\omega}$ the sub-category of compact objects, as defined in \cite[Section~VII.7]{FS}, where $ \Lambda \in \{ \ov \F_{\ell}, \ov \Z_{\ell}, \ov \Q_{\ell} \}$ and $\ell$ is not a bad prime, namely the order of $\pi_0 Z(\G)$ is invertible in $\Lambda$. Fargues and Scholze constructed an action (called the spectral action) of the category $\Perf^{\mathrm{qc}}( [Z^1(W_{\Q_p}, \widehat{\G})_{\Lambda}/\widehat{\G}] )$ on the category $\Dlis(\Bun_{\G}, \Lambda)$ and expect this action relates these categories in a precise way. Fix a Borel subgroup $\rB \subset \G$ with unipotent radical $\U$ and a generic character $\psi : \U(\Q_p) \longrightarrow \Lambda^{\times}$. Let $\mathcal{W}_{\varphi}$ be the Whittaker sheaf, which is the sheaf concentrated on the stratum $\Bun_{\G}^1$ corresponding to the representation $\cInd^{\G}_{\U} \psi $ of $\G(\Q_p)$. We define the functor
\begin{align*}
  \Ind\Perf^{\mathrm{qc}}( [Z^1(W_{\Q_p}, \widehat{\G})_{\Lambda}/\widehat{\G}]) &\longrightarrow \Dlis(\Bun_{\G}, \Lambda)  \\
  \M &\longmapsto \M \star \mathcal{W}_{\psi}
\end{align*}
as the colimit-preserving extension of the spectral action on $\mathcal{W}_{\psi}$ (where $\star$ denotes the spectral action). Then the categorical local Langlands conjecture predicts that the corresponding right adjoint of the above functor is fully faithful when restricted to the compact objects, and induces an equivalence of ($\Perf([Z^1(W_{\Q_p}, \widehat{\G})_{\Lambda}/\widehat{\G}])$-linear small stable) $\infty$-categories 
\[
\Dlis(\Bun_{\G}, \Lambda)^{\omega} \cong \Ind\Perf^{\mathrm{b,qc}}_{\Coh, \Nilp}( [Z^1(W_{\Q_p}, \widehat{\G})_{\Lambda}/\widehat{\G}]).
\]

\subsection{Main results} \textbf{} \\

In \cite{GLn}, we proved some part of this conjecture for $\G = \GL_n$ and $\Lambda = \ov \Q_{\ell}$. The goal of this current manuscript is to extend these results to the case $\Lambda = \ov \F_{\ell}$ and deduce some applications to the cohomology of local and global Shimura varieties (we treat the case $\G = \Res_{F/\Q_p}(\GL_{n, F})$ in the main text but let us only consider the case $F = \Q_p$ in the introduction).

We now describe in more details the main results. On the spectral side, let $\phi = \phi_1 \oplus \dotsc \oplus \phi_r$ be an $L$-parameter of $\G$ such that for $1 \leq i \leq r$, $\phi_i$ is an irreducible representation of $W_{\Q_p}$ and for $1 \leq i \neq j \leq r$, there does not exist an unramified character $\chi$ of $W_{\Q_p}$ such that $\phi_i \simeq \phi_j \otimes \chi$. Let $[C_{\phi}]$ be the connected component containing $\phi$ in $[Z^1(W_{\Q_p}, \widehat{\G})_{\Lambda}/\widehat{\G} ]$. Then we have (proposition \ref{itm : simple connected components})
    \[
    [C_{\phi}] \simeq [(\mathbb{G}^r_{m} \times \mu_{\Lambda}) / \mathbb{G}^r_m],
    \]
    the quotient stack of $(\mathbb{G}^r_{m} \times \mu_{ \Lambda})$ by the trivial action of $\mathbb{G}^r_{m}$ and $\mu_{\Lambda} = \displaystyle \prod_{i = 1}^r \Spec \overline{\F}_{\ell}[\Syl_i]$ where $\Syl_i$ is the $\ell$-Sylow group of $ \F^{\times}_{p^{f_i}} $ and the $f_i$'s are uniquely determined by $\phi$. We note that the centralizer $S_{\phi}$ of $\phi$ is isomorphic to $\mathbb{G}^r_{m}$. In particular, the category $\Ind\Perf^{\mathrm{b,qc}}_{\Coh}([C_{\phi}])$ is a direct summand of $\Ind\Perf^{\mathrm{b,qc}}_{\Coh, \Nilp}( [Z^1(W_{\Q_p}, \widehat{\G})_{\Lambda}/\widehat{\G}])$ and
    \[
    \Ind\Perf^{\mathrm{b,qc}}_{\Coh}([C_{\phi}]) \simeq \bigoplus_{\chi \in \Irr(\mathbb{G}^r_m)} \mathrm{D}^{\mathrm{b}}_{\mathrm{\Coh}}(A_{\phi}\text{-Mod}),
    \]
where $A_{\phi}$ is the ring of global sections of $[C_{\phi}]$ and $\mathrm{D}^{\mathrm{b}}_{\mathrm{\Coh}}(A_{\phi}\text{-Mod})$ is the derived category of bounded complex of $A_{\phi}$-modules with coherent cohomology.
    
    On the automorphic side, we let $\Dlis^{[C_{\phi}]}(\Bun_{\G}, \Lambda)^{\omega} \subset \Dlis(\Bun_{\G},\Lambda)^{\omega}$ be the full sub-category generated by the objects on which the excursion operator corresponding to the function that is $1$ on $[C_{\phi}]$ and $0$ elsewhere acts via the identity. In particular, the Schur-irreducible objects in this subcategory all have Fargues-Scholze parameter given by an $L$-parameter in $[C_{\phi}]$. 
    
    Let $\chi \in \Irr(S_{\phi})$ be a character, then one can construct an element $b \in B(\G)$ and an irreducible representation $\pi_b$ in $\Rep_{\Lambda}(\G_b(\Q_p))$ whose Fargues-Scholze parameter is given by $\phi$. Let $R(\phi)$ be the set of such pairs $(b, \pi_b)$, then the previous construction gives a bijection between $\Irr(S_{\phi})$ and the set $R(\phi)$, (proposition \ref{itm : vanishing}). For each pair $(b_{\chi}, \pi_{\chi})$, one can consider the block $\Rep_{\Lambda}(\mathfrak{s}_{\phi, \chi})$ of $\Rep_{\Lambda}(\G_{b_{\chi}}(\Q_p))$ containing $\pi_{\chi}$ and show that this block is actually equivalent to the category of $A_{\phi}$-modules (propositions \ref{itm : description of blocks}, \ref{itm : structure of a block}). In addition, one can consider the derived category $\mathrm{D}(\Rep_{\Lambda}(\mathfrak{s}_{\phi, \chi}))^{\omega}$ consisting of bounded complexes of finitely generated cohomology groups as a full sub-category of $\Dlis^{[C_{\phi}]}(\Bun_{\G}, \Lambda)^{\omega}$ via the renormalized push forward functor
    \[
    i^{\mathrm{ren}}_{b_{\chi} !} := i_{b_{\chi}!}(\delta^{-1/2}_{b_{\chi}} \otimes - )[- d_{\chi}] :\Dlis (\Bun^{b_{\chi}}_{\G}, \Lambda) \longrightarrow \Dlis(\Bun_{\G}, \Lambda),
    \]
    where $d_{\chi} = \langle 2\rho, \nu_{b_{\chi}} \rangle $. Our first main theorem is the following (see theorem \ref{itm : spectral action - general} for the precise statement):

    \begin{theorem} \phantomsection \label{itm : main theorem in the intro}
        \begin{enumerate}
            \item[(1)] We have a decomposition of categories
            \[
            \Dlis^{[C_{\phi}]}(\Bun_{\G}, \Lambda)^{\omega} \simeq \bigoplus_{\chi \in \Irr(S_{\phi})} i^{\mathrm{ren}}_{b_{\chi} !} \big( \mathrm{D}(\Rep_{\Lambda}(\mathfrak{s}_{\phi, \chi}))^{\omega} \big),
            \]
            \item[(2)] The spectral action acting on the Whittaker sheaf $\mathcal{W}_{\varphi}$ gives an equivalence 
            \[
            \Ind\Perf^{\mathrm{b,qc}}_{\Coh}([C_{\phi}]) \simeq \Dlis^{[C_{\phi}]}(\Bun_{\G}, \Lambda)^{\omega}.
            \]
        \end{enumerate}
    \end{theorem}

\subsection{Some ideas of the proof} \textbf{} \\

The strategy of the proof of our first main results is essentially the same as that of the $\ell$-adic case $\Lambda = \ov \Q_{\ell}$ in \cite{GLn}. For each $\chi \in \Irr(S_{\phi})$, we denote by $\mathcal{F}_{\chi}$ the sheaf $i^{\mathrm{ren}}_{b_{\chi} !}(\pi_{b_{\chi}}) \in \Dlis^{[C_{\phi}]}(\Bun_{\G})^{\omega}$. Note also that there is a vector bundle $C_{\chi}$ on $[C_{\phi}]$ associated to $\chi$. The first step in the proof is to show that (theorem \ref{itm : main theorem I})
\begin{equation} \phantomsection \label{main equation}
C_{\chi} \star \mathcal{F}_{\Id} \simeq \mathcal{F}_{\chi}, 
\end{equation}
and then extend this computation to an equivalence of categories by the knowledge of the structures of $\Rep_{\Lambda}(\mathfrak{s}_{\phi, \chi})$. 

Let us give more details on the proof of the equation (\ref{main equation}). Note that $\Irr(S_{\phi})$ is a group isomorphic to $\Z^r$ and it is enough to compute $C_{\chi} \star \mathcal{F}$ for $\chi = (d_1, \dotsc, d_r)$ with positive entries. For each $1 \leq i \leq r$, denote by $\chi_i$ the element in $\Z^r$ with $1$ in the $i^{\textrm{th}}$-position and $0$ elsewhere. We use induction on $r$ and then induction on $|\chi| := d_1 + \dotsc + d_r$. For the induction step, the arguments in \cite{GLn} still work without too much change. For this, we only need to show that
\[
C_{\chi_i} \star \mathcal{F}_{\chi} \simeq \mathcal{F}_{\chi \otimes \chi_i},
\]
for $1 \leq i \leq r$. In order to show this equality, we need to prove that
\begin{equation} \phantomsection \label{itm : stalk}
i^*_{b_{\chi \otimes \chi_i}}C_{\chi_i} \star \mathcal{F}_{\chi} \simeq i^*_{b_{\chi \otimes \chi_i}} \mathcal{F}_{b_{\chi \otimes \chi_i}},    
\end{equation} 
and show that when $b \neq b_{\chi \otimes \chi_i} \in B(\G)$, we have
\begin{equation} \phantomsection \label{itm : vanishing of stalks}
i^*_{b}C_{\chi_i} \star \mathcal{F}_{\chi} \simeq 0.   
\end{equation}

In order to compute the spectral action of $C_{\chi}$, let $\mu_{\Std} = (1, 0^{(n-1)})$ be the simplest minuscule cocharacter of $\GL_n$ and denote by $\T_{\mu_{\Std}}$ the Hecke operator on $\Dlis^{[C_{\phi}]}(\Bun_{\G}, \Lambda)$ associated to the highest weight representation $V$ of $\GL_n$ of weight $\mu_{\Std}$. Then, on one hand we can relate the action of $\T_{\mu_{\Std}}$ with the cohomology of local Shtukas spaces
\[
i^*_b\T_{\mu_{\Std}}(\mathcal{F}_{\chi}) \simeq R\Gamma_{c}(\G,b_{\chi},b,\mu)[\pi_{b_{\chi}}],
\]
up to some shifts and twists (for precise formulas, see lemma \ref{shimhecke}). On the other hand, we can also relate $\T_{\mu}$ with the $C_{\chi_i}$'s by the formula (proposition \ref{itm : fundamental decomposition of Hecke operator})
\[
\T_{\mu_{\Std}}(-) \simeq \bigoplus^r_{i = 1} C_{\chi_i} \star (-) \boxtimes \phi_i,
\]
and thus one can compute $C_{\chi_i} \star \mathcal{F}_{\chi} $ via the cohomology of various local Shimura varieties. By combining vanishing results of local Shimura varieties (proposition \ref{itm : vanishing}) with an analogue of Boyer's trick \cite[section $5$]{GLn} and a careful study of modifications of vector bundles on the Fargues-Fontaine curve, we can prove the equalities $(\ref{itm : stalk}), (\ref{itm : vanishing of stalks})$ as in \cite[section $6$]{GLn}. 

In the $\ell$-adic case, the initial step (the case $r = 1$) can be deduced quite easily from the works \cite{Far04}, \cite{Boy09}, \cite{SWS}, \cite{Han}. However these results are not available in the $\ell$-modular setting. For this reason, we will use global method to treat this case. Namely, we relate the cohomology of basic local Shimura varieties with some global Shimura varieties with $\ov \F_{\ell}$-coefficient and then use torsion vanishing technique to relate the latter with the cohomology of global Shimura varieties with $\ov \Q_{\ell}$-coefficient.  

We now elaborate on details of this argument. In this case we have $\Irr(S_{\phi}) \simeq \Z$ and let $\chi$ be a character of $S_{\phi}$. By using vanishing results for local Shimura varieties (proposition \ref{itm : vanishing}), one can show equality (\ref{main equation}) up to some shift (section \ref{itm : up to some shift}).

By using Mantovan's formula \cite[Theorem 8.5.7]{DHKM24}, \cite[Theorem 1.12]{HL} for some well-chosen Shimura variety $\Sh_{K^p}$, we know that the complex $R\Gamma_c(\Sh_{K^p}, \ov \F_{\ell})$ has a filtration as a complex of $\G(\Q_p)$-representations with graded pieces isomorphic to $ i_1^* \T_{\mu}(i_{b !} \mathcal{I} )[-d](-d/2) $ where $d = \langle 2\rho, \mu \rangle$ is the dimension of the Shimura variety $\Sh_{K^p}$, where $b$ varies in the finite set $B(\G, \mu)$ and where $\mathcal{I}$ is some sheaf related to Igusa varieties. Here $b_{\chi}$ is the unique basic element in $B(\G, \mu)$. 

We denote by $R\Gamma_c(\Sh_{K^p}, \ov \F_{\ell})^{[\phi]}$ the projection of $R\Gamma_c(\Sh_{K^p}, \ov \F_{\ell})$ to the block corresponding to $\phi$ in the derived category $\mathrm{D}(\Rep_{\ov \F_{\ell}}(\G(\Q_p)))$. By local vanishing property again, the contribution of a non-basic $b$ to $R\Gamma_c(\Sh_{K^p}, \ov \F_{\ell})^{[\phi]}$ is trivial and by the affineness of Igusa varieties, we see that $ i_1^* \T_{\mu}(i_{b_{\chi} !} \mathcal{I} )^{[\phi]}[-d](-d/2) $ is concentrated in $1$ degree. Thus $R\Gamma_c(\Sh_{K^p}, \ov \F_{\ell})^{[\phi]}$ is concentrated in $1$ degree and torsion free. Hence we can lift this cohomology group to characteristic $0$ coefficient where we know how to compute the cohomology. This allows us to deduce the correct shift in equality (\ref{main equation}) and conclude.

\subsection{Some applications} \textbf{} \\

\textit{Hecke eigensheaves and the Harris-Viehmann conjecture.} \textbf{} \\

Our main theorem \ref{itm : main theorem in the intro} has some interesting applications. First of all, since the irreducible representations of $S_{\phi}$ are all of dimension $1$, we deduce the following description of the regular representation $ \displaystyle V_{\text{reg}} = \bigoplus_{\chi \in \Irr(S_{\phi})} \chi $. Thus the sheaf $ \displaystyle \mathcal{G}_{\phi} := \bigoplus_{\chi \in \Irr(S_{\phi})} \mathcal{F}_{\chi} $ is non trivial and  we can compute $\displaystyle \T_{V} (\mathcal{G}_{\phi})$ for every algebraic representation $V$ of $\GL_n$ and then prove the following result (theorem \ref{itm : Hecke eigensheaf}).
\begin{theorem}
The sheaf $ \displaystyle \mathcal{G}_{\phi} $ is a non trivial Hecke eigensheaf corresponding to the $L$-parameter $\phi$.
\end{theorem} 

We can also apply our main theorem to deduce a description of the cohomology of non-basic local Shimura varieties. Thus we can deduce a version of the Harris-Viehmann conjecture as in the $\ell$-adic case. For simplicity, let us state the result up to some shifts and twists and only in the minuscule case where $\mu = (1^{(a)}, 0^{(n-a)})$ for some integer $1 \le a \le n$ (see theorem \ref{itm : Harris-Viehmann conjecture} for a more general version of the result).
\begin{theorem}
With the above notations, if $\mu$ is minuscule then we have the following identity (up to some shifts and twists) of $ \GL_n(\Q_p) \times W_{\Q_p} $-representations
\[
R\Gamma_{c}(\GL_n, b,\mu) [\pi_{b}] \simeq \Ind_{\rP}^{\GL_n} R\Gamma_{c}(\M, b_{\M},\mu_{\M})  [\pi_{b}].
\]
where $ b_{\M} $ is the reduction of $b$ to $\M$ and $\mu_{\M} = \mu_1 \times \dotsc \times \mu_k$ and where $\mu_j = (1^{(\deg\E(\lambda_j))}, 0^{(\rank\E(\lambda_j) - \deg\E(\lambda_j))})$.
\end{theorem}

\textit{Torsion vanishing of the cohomology of some Shimura varieties.} \textbf{} \\

Torsion vanishing result of the cohomology of Shimura varieties are studied extensively these years. Pioneering results in this direction are obtained in \cite{CS, CS1} for generic part of the cohomology. Then similar results are proved in \cite{Ko1, HL, dSS, DHKM24} by using ideas coming from Fargues-Scholze machinery. The papers \cite{Boy19}, \cite[Theorem B, Theorem 7.5.2]{CT} contain results beyond the generic case as well. In a very recent paper \cite{Yang-Zhu}, the authors announced a torsion vanishing result for the generic part of the cohomology of some Shimura varieties and their method is more versatile. With our main theorem, we can obtain some torsion vanishing results for generic part (section \ref{itm : generic part}) in the style of \cite{Ko1, HL} and also for non-generic part (section \ref{itm : non-generic part}) in the direction of \cite{Boy19, CT}. We can also compute the supercuspidal part of the cohomology with torsion coefficients of local Shimura varieties associated to unitary groups (section \ref{itm : local Shimura variety of unitary group}).


\subsection*{Acknowledgments}

I would like to thank Alexander Bertoloni-Meli, Pascal Boyer, Linus Hamann, Naoki Imai, Teruhisa Koshikawa, Thomas Lanard, Nhat-Hoang Le, Thanh-Dat Pham-Ngo, Vincent Sécherre, Mingjia Zhang for very helpful discussions and feedback.  \\

\section*{Notation}
We use the following notation: 
\begin{itemize}
\item  $\ell$ and $p$ are distinct odd prime numbers.
\item $ \breve \Q_p $ is the completion of the maximal unramified extension of $\Q_p$ with Frobenius $ \sigma $. Let $\overline{\Q}_p$ be an algebraic closure of $\Q_p$ and denote $\Gamma := \Gal(\overline{\Q}_p / \Q_p)$.
\item $ \G $ is a connected reductive group over $\Q_p$. Let $\mathrm{H}$ be a quasi-split inner form of $\G$ and fix an inner twisting $ \G_{\breve \Q_p} \overset{\sim}{\longrightarrow} \mathrm{H}_{\breve \Q_p} $. 
\item  $ A \subseteq \T \subseteq B $ where $ A $ is a maximal split torus, $ \T = Z_H(A) $ is the centralizer of $ A $ in $ \mathrm{H} $ and $ B $ is a Borel subgroup in $\mathrm{H}$. Let $ \U $ be its unipotent radical.
\item $ (X^*(\T), \Phi, X_*(\T), \Phi^{\vee}) $ is the absolute root datum of $\G$ with positive roots $ \Phi^+ $ and simple roots $ \Delta $ with respect to the choice of $B$.
\item $\rho$ is the half sum of the positive roots.
\item  Further, $ (X^*(A), \Phi_0, X_*(A), \Phi^{\vee}_0) $ denotes the relative root datum, with positive roots $ \Phi^+_0 $ and simple roots $ \Delta_0 $.
\item On $ X_*(A)_{\Q} $ resp.~$ X_*(\T)_{\Q} $ we consider the partial order given by $ \nu \leq \nu' $ if $ \nu' - \nu $ is a non-negative rational linear combination of positive resp.~relative coroots. 
\item Let $\rP$ be a parabolic subgroup of $\G$, then $\Ind_{\rP}^{\G}$, resp. $\ind_{\rP}^{\G}$, denotes the normalized, resp. un-normalized parabolic induction.
\item Let $ C | \overline{\Q}_p $ be an algebraically closed complete field. Let $C^{\circ}$ resp. $C^{\flat,\circ}$ be the subring of power-bounded elements of $C$ resp. $C^{\flat}$ and let $\xi$ be a generator of the kernel of the surjective map $W(C^{\flat,\circ})\rightarrow C^{\circ}$. Let $\B+:=\B+(C)$ be the $\xi$-adic completion of $W(C^{\flat,\circ})[1/p]$ and $\BdR=\BdR(C)=\B+[\xi^{-1}]$. Then $\B+\cong C  [[\xi]]$ and $\BdR\cong C((\xi))$.
\item Let $F$ be a finite extension of $\Q_p$. We denote by $X_F$ the schematic Fargues-Fontaine curve over $C^{\flat}$. The untilt $C$ of $C^{\flat}$ corresponds to a point $ \infty \in |X_F|$ with residue field $C$ and $\widehat{\mathcal{O}}_{X, \infty}\cong \B+$. The subscript $F$ will be omitted when there is no confusion.  
\item For $ \lambda \in \Q $, we denote the stable vector bundle on $X_F$ whose slope is $\lambda$ by $\OO_F(\lambda)$. If $\lambda = 0$, we simply write $\OO_F$ for $\OO_F(0)$. The subscript $F$ will be omitted when there is no confusion. 
\item Let $B(\G)$ be the set of $\G(\breve \Q_p)$-$\sigma$-conjugacy classes of elements of $\G(\breve \Q_p)$. For each elements $[b] \in B(\G)$, we denote by $\G_b$ the $\sigma$-centralizer of $b$.  By work of Kottwitz, elements $[b]$ are classified by their Kottwitz point $\kappa_{\G}(b)\in \pi_1(\G)_{\Gamma}$ and their Newton point $\nu_b \in X_*(A)_{\mathbb Q,\dom}$.
\item For a $\G$-bundle $\E$ on $X$ let $\nu_{\E} \in X_*(A)_{\mathbb Q,\dom}$ be the corresponding Newton polygon.
\item $\Bun_{\G}$ is the stack of $\G$-bundles on the Fargues-Fontaine curve.
\end{itemize} 
\section{Generalities on $\Bun_{\G}$}
\subsection{Combinatoric description of $\Bun_{\G}$}
Let $\G$ be a connected reductive group defined over a finite extension of $\Q_p$. We let $\Perf$ denote the category of perfectoid spaces over $\ol{\mathbb{F}}_{p}$. We write $\ast := \Spd(\ol{\mathbb{F}}_{p})$ for the natural base. The key object of study is the moduli pre-stack $\BunG$ sending $S \in \Perf$ to the groupoid of $\G$-bundles on the relative Fargues--Fontaine curve $X_{S}$. The following theorem gives a geometric description of $\BunG$.
\begin{theorem} \cite[Theorem III.0.2]{FS}
    The pre-stack $\Bun_{\G}$ satisfies the following properties:
    \begin{enumerate}
        \item The prestack $\Bun_{\G}$ is a small $v$-stack.
        \item The points $ | \Bun_{\G} | $ are naturally in bijection with Kottwitz’ set $B(\G)$ of $\G$-isocrystals.
        \item The map 
        \[
        \nu :  | \BunG | \longrightarrow B(\G) \longrightarrow  (X_{*}(\T)_{\Q, \mathrm{dom}})^{\Gamma}
        \]
        is semi-continuous and
        \[
        \kappa : | \BunG | \longrightarrow B(\G) \longrightarrow \pi_1(\G_{\ol{\Q}_p})_{\Gamma}
        \]
        is locally constant. Moreover, the map $| \BunG | \longrightarrow B(\G)$ is a homeomorphism when $B(\G)$ is equipped
with the order topology \cite{Vie, Han2}.
        \item The semistable locus $\Bun^{ss}_{\G} \subset \BunG $ is open, and given by
        \[
        \Bun^{ss}_{\G} \simeq \coprod_{b \in B(\G)_{\mathrm{basic}}} [\bullet / \underline{\G_b(\Q_p)}].
        \]
        \item For any $b \in B(\G)$, the corresponding subfunctor
        \[
        i^b : \Bun^b_{\G} \subset \BunG
        \]
        is locally closed, and isomorphic to $[\bullet / \widetilde{\G}_b ]$  where $\widetilde{\G}_b$ is a $v$-sheaf of groups such that $\widetilde{\G}_b \longrightarrow \ast $  is representable in locally spatial diamonds with $ \pi_0\widetilde{\G}_b = \G_b(\Q_p) $. The connected component $\widetilde{\G}_b^{0} \subset \widetilde{\G}_b$ of the identity is cohomologically smooth of dimension $\langle 2\rho, \nu_b \rangle$. 
    \end{enumerate}
\end{theorem}
For $\G = \Res_{F/\Q_p} \GL_{n, F}$ where $F$ is an unramified extension of $\Q_p$ of degree $d$, we have $X_*(\T) \simeq \displaystyle \prod_{i \in \Z / d \Z} \Z^n $. The Weil group $W_{\Q_p}$ acting on $X_*(\T)$ factors through the relative Weil group $W_{F}$ and the element $\overline{1}$ which is the image of $1$ in $ \Z / d \Z  \simeq \Gal(F / \Q_p) \simeq W_{\Q_p} / W_{F} $ acts by sending $ (x_i)_{i=0}^{d-1} \in \displaystyle \prod_{i \in \Z / d \Z} \Z^n $ to $ (x_{i-1})_{i=0}^{d-1} \in \displaystyle \prod_{i \in \Z / d \Z} \Z^n $ where we denote $x_{j+kd} := x_{j}$ for $0 \leq j < d$ and for every integer number $k$.

The positive roots of $\Res_{F/\Q_p} \GL_{n, F}$ (corresponding to the Borel subgroup given by the upper triangular matrices) are 
\[
 \Phi^+ = \{ e_{k, i} - e_{k', i} \ | \ k, k' \in \{ 1, 2, \dotsc, n \}, \ i \in \{ 0, 1, \dotsc, d-1 \}, \ k < k' \}, 
\]
and the simple roots are
\[
 \Delta = \{ e_{k, i} - e_{k+1, i} \ | \ k \in \{ 1, 2, \dotsc, n-1 \}, \ i \in \{ 0, 1, \dotsc, d-1 \} \},
\]
and therefore we can identify $(X_{*}(\T)_{\Q, \mathrm{dom}})^{\Gamma}$ with the set $ \{ (x_1, \dotsc, x_n) \ | \ x_i \in \Q, \ \ x_i \geq x_{i + 1}  \}$.

We have $B(\Res_{F/\Q_p} \GL_{n, F}) \simeq B(\GL_{n, F}) $ by Shapiro's lemma. An element $b$ of $B(\GL_{n, F})$ is given by a sequence
\[
\lambda_1 = \dfrac{r_1}{s_1} > \dotsc > \lambda_a = \dfrac{r_a}{s_a},
\]
where $(r_i, s_i) = 1$ and with the multiplicities $ m_1s_1, \dotsc, m_as_a $ such that $ \displaystyle \sum_i m_is_i = n $. Then the corresponding 
Newton point $\nu_b$ in $(X_{*}(\T)_{\Q, \mathrm{dom}})^{\Gamma}$ is
\[
\nu_b = \big(\underbrace{\tfrac{\lambda_1}{d}, \dotsc,\tfrac{\lambda_1}{d}}_{m_1s_1 \ \text{times}}, \dotsc, \underbrace{\tfrac{\lambda_a}{d} , \dotsc, \tfrac{\lambda_a}{d}}_{m_as_a \ \text{times}} \big) \in \Q^n,
\]
and we have $ \langle 2\rho, \nu_b \rangle = \displaystyle \sum_{i < j} m_im_js_is_j(\lambda_i - \lambda_j) $. The associated group $\G_b$ is isomorphic to 
\[
\G_b = \Res_{F / \Q_p}(\GL_{m_1}(D_{\lambda_1})) \times \dotsc \times \Res_{F / \Q_p}(\GL_{m_a}(D_{\lambda_a})),
\]
where $D_{\lambda}$ is the division algebra over $F$ with invariant $\lambda$. Moreover, $\pi_1(\G_{\ol{\Q}_p})_{\Gamma} = X_*(\T)/(\text{coroot lattice}) $ is naturally isomorphic to $\Z$, and in this case $\kappa(b)$ is given by $ \displaystyle \sum_i m_ir_i $. Remark that if $\E_b$ is the $\G$-bundle corresponding to $b$ then $\nu_{\E_b} = (-\nu_b)_{\text{dom}}$.
\subsection{Sheaves on $\Bun_{\G}$}

Let $\Lambda$ be in $\{ \overline{\Q}_{\ell}, \overline{\F}_{\ell} \}$.  We will be interested in the derived category $\Dlis(\Bun_{\G},\ol{\mathbb{Q}}_{\ell})$ of lisse-\'etale $\overline{\mathbb{Q}}_{\ell}$-sheaves \cite[Section~VII.6.]{FS} and denote by $\Dlis(\Bunn,\ol{\mathbb{Q}}_{\ell})^{\omega}$ for the triangulated sub-category of compact objects in $\Dlis(\Bun_{\G},\ol{\mathbb{Q}}_{\ell})$. When $\ell$ is very good as in \cite[Page 33]{FS}, we will also be interested in $\Dlis(\Bun_{\G},\ol{\mathbb{F}}_{\ell})$, the derived category of étale $\overline{\F}_{\ell}$-sheaves on $\Bun_{\G}$ and its triangulated sub-category of compact objects $\Dlis(\Bun_{\G},\ol{\mathbb{F}}_{\ell})^{\omega}$. 


The key point is that objects in these categories are manifestly related to smooth admissible representations of $\G(\mathbb{Q}_{p})$. The strata of $\BunG$ admit a natural map $\Bun_{\G}^{b} \ra [\ast/\G_{b}(\mathbb{Q}_{p})]$ to the classifying stack defined by the $\mathbb{Q}_{p}$-points of the $\sigma$-centralizer $\G_{b}$, and, by \cite[Proposition~VII.7.1]{FS}, pullback along this map induces an equivalence
\[ \Dc(\G_{b}(\mathbb{Q}_{p}), \Lambda) \simeq \Dlis([\ast/\G_{b}(\mathbb{Q}_{p})], \Lambda) \xrightarrow{\simeq} \Dlis(\Bun^{b}_{\G}, \Lambda), \]
where $\Dc(\G_{b}(\mathbb{Q}_{p}), \Lambda)$ is the unbounded derived category of smooth $\Lambda$-representations of  $\G_{b}(\mathbb{Q}_{p})$.

Recall that via excision triangles, there is an infinite semi-orthogonal decomposition of $\Dlis(\BunG, \Lambda)$ into the various $\Dlis(\Bun^b_{\G}, \Lambda)$ for $b \in B(\G)$ \cite[Chap. VII]{FS}. 

We recall the definition of Hecke operators on $\Dlis(\BunG, \Lambda)$. For a finite index set $I$, we let $\Rep_{\Lambda}(\phantom{}^{L}\G^{I}(\Lambda))$ denote the category of algebraic representations of $I$-copies of the Langlands dual group, and we let $\Div^{I}$ be the product of $I$-copies of the diamond $\Div^{1} = \Spd(\Breve{\mathbb{Q}}_{p})/\mathrm{Frob}^{\mathbb{Z}}$. The diamond $\Div^{1}$ parametrizes, for $S \in \Perf$, characteristic $0$ untilts of $S$, which in particular give rise to Cartier divisors in $X_{S}$. We then have the Hecke stack
\[ \begin{tikzcd}
& & \arrow[dl, swap, "h^{\leftarrow}"] \mathrm{Hck} \arrow[dr,"h^{\rightarrow} \times supp"] & & \\
& \Bun_{\G} & & \Bun_{\G} \times \Div^{I}  & 
\end{tikzcd} \]
defined as the functor parametrizing, for $S \in \Perf$ together with a map $S \rightarrow \Div^{I}$ corresponding to characteristic $0$ untilts $S_{i}^{\sharp}$ defining Cartier divisors in $X_{S}$ for $i \in I$, a pair of $G$-torsors $\mathcal{E}_{1}$, $\mathcal{E}_{2}$ together with an isomorphism 
\[ \beta:\mathcal{E}_{1}|_{X_{S} \setminus \bigcup_{i \in I} S_{i}^{\sharp}} \xrightarrow{\simeq} \mathcal{E}_{2}|_{X_{S} \setminus \bigcup_{i \in I} S_{i}^{\sharp}},\]
where $h^{\leftarrow}((\mathcal{E}_{1},\mathcal{E}_{2},i,(S_{i}^{\sharp})_{i \in I})) = \mathcal{E}_{1}$ and $h^{\rightarrow} \times supp((\mathcal{E}_{1},\mathcal{E}_{2},\beta,(S_{i}^{\sharp})_{i \in I})) = (\mathcal{E}_{2},(S_{i}^{\sharp})_{i \in I})$. For each element $W \in \Rep_{\Lambda}(^{L}\G^{I}(\Lambda))$, the geometric Satake correspondence of Fargues--Scholze \cite[Chapter~VI]{FS} furnishes a solid $\Lambda$-sheaf $\mathcal{S}_{W}$ on $\mathrm{Hck}$. This allows us to define Hecke operators.
\begin{definition}
For each $W \in \Rep_{\Lambda}(\phantom{}^{L}\G^{I}(\Lambda))$, we define the Hecke operator
\[ \T_{W}: \Dlis(\Bun_{\G},\Lambda) \rightarrow \Dlis(\Bun_{\G} \times X^{I},\Lambda) \]
\[ A \mapsto R(h^{\rightarrow} \times supp)_{\natural}(h^{\leftarrow *}(A) \otimes^{\mathbb{L}} \mathcal{S}_{W}),\]
where $\mathcal{S}_{W}$ is a solid $\Lambda$-sheaf and the functor $R(h^{\rightarrow} \times supp)_{\natural}$ is the natural push-forward (i.e the left adjoint to the restriction functor in the category of solid $\Lambda$-sheaves \cite[Proposition~VII.3.1]{FS}). 
\end{definition}
It follows by 
\cite[Theorem~I.7.2, Proposition~IX.2.1, Corollary~IX.2.3]{FS} that, if $E$ denotes the reflex field of $W$, this induces a functor 
\[ \Dlis(\Bun_{\G},\Lambda) \rightarrow \Dlis(\Bun_{\G},\Lambda)^{BW_{E}^{I}} \]
which we will also denote by $\T_{W}$. The Hecke operators are natural in $I$ and $W$ and compatible with exterior tensor products.

Let $b, b'$ be elements in $ B(\G)$. Given a geometric dominant cocharacter $\mu$ of $\G$ with reflex field $E$, we call the quadruple $(\G,b,b',\mu)$ a local shtuka datum. Attached to it, we define the shtuka space
\[ \Sht(\G,b,b',\mu) \longrightarrow \Spd(\Breve{E}), \]
as in \cite{SW} where $\Breve{E} := \Breve{\mathbb{Q}}_{p}E$, to be the space parametrizing modifications $\mathcal{E}_{b} \longrightarrow \mathcal{E}_{b'}$ of $\G$-bundles on the Fargues--Fontaine curve $X$ with type bounded by $\mu$.
\begin{remark}
    When $b=1$, we also use the notation $\Sht(\G, b', \mu)$ for $\Sht(\G, 1, b', \mu)$. We note that our definition of $\Sht(\G, b', \mu)$ coincides with $\Sht(\G, b', -\mu)$ in the notation of \cite{SW}, where $ - \mu$ is the dominant inverse of $\mu$. This convention limits the appearance of duals when studying the cohomology of these spaces. 
\end{remark}

This has commuting actions of $\G_{b}(\mathbb{Q}_{p})$ and $\G_{b'}(\mathbb{Q}_{p})$ coming from automorphisms of $\mathcal{E}_{b}$ and $\mathcal{E}_{b'}$, respectively. Moreover, we have a Weil descent datum of $\Sht(\G, b, b', \mu)$ coming from the equlities $ b b^{\sigma} (b^{-1})^{\sigma} = b $ and $ (b')^{-1} b' (b')^{\sigma} = (b')^{\sigma} $. 

We define the tower 
\[ \Sht(\G,b,b',\mu)_{K} := \Sht(\G,b,b',\mu)/\underline{K} \longrightarrow \Spd(\Breve{E}) \] of locally spatial diamonds \cite[Theorem~23.1.4]{SW}
for varying open compact subgroups $K \subset \G_{b'}(\mathbb{Q}_{p})$. We can define the cohomology $R\Gamma_c(\Sht(\G,b,b',\mu)_{K}, \mathcal{S}_{\mu})$ as in \cite[section 3]{Nao}. More precisely, there is a natural map 


\[ \mathsf{p}: \Sht(\G,b,b',\mu)_{\infty} \longrightarrow \mathrm{Hck} \] 
mapping to the locus of modifications with type bounded by $\mu$. Attached to the geometric cocharacter $\mu$, consider the highest weight representation $V_{\mu} \in \Rep_{\Lambda}(\phantom{}^{L}\G(\Lambda))$. The associated $\Lambda$-sheaf $\mathcal{S}_{\mu}$ on $\mathrm{Hck}$ (by the geometric Satake isomorphism) considered above will be supported on this locus, and we abusively denote $\mathcal{S}_{\mu}$ for the pullback of this sheaf along $\mathsf{p}$. Since $\mathsf{p}$ factors through the quotient of $\Sht(\G,b,b',\mu)_{\infty}$ by the simultaneous group action of $\G_{b}(\mathbb{Q}_{p}) \times \G_{b'}(\mathbb{Q}_{p})$, this sheaf will be equivariant with respect to these actions. This allows us to define the complex
\[ R\Gamma_{c}(\G,b,b',\mu) := \colim_{K \rightarrow \{1\}} R\Gamma_{c}(\Sht(\G,b,b',\mu)_{K,\mathbb{C}_{p}},\mathcal{S}_{\mu}) \]
which will be a complex of smooth admissible $\G_{b}(\mathbb{Q}_{p}) \times \G_{b'}(\mathbb{Q}_{p}) \times W_{E_{\mu}}$-modules, where $\Sht(\G,b,b',\mu)_{K,\mathbb{C}_{p}}$ is the base change of $\Sht(\G,b,b',\mu)_{K}$ to $\mathbb{C}_{p}$. Remark that if $\mu$ is minuscule then $ \mathcal{S}_{\mu} \simeq \Lambda[d_{\mu}](\tfrac{d_{\mu}}{2})$ where $d_{\mu} = \langle 2 \rho, \mu \rangle$. For $\pi_{b}$ an irreducible smooth representation of $\G_b(\Q_p)$ with coefficient in $\Lambda$, this allows us to define the following complexes 
\begin{equation}{\label{eq: RGammaflat}}
    R\Gamma^{\flat}_{c}(\G,b,b',\mu)[\pi_{b}] := R\mathcal{H}om_{\G_{b}(\mathbb{Q}_{p})}(R\Gamma_{c}(\G,b,b',\mu),\pi_{b})
\end{equation}
and
\begin{equation}{\label{eq: RGamma}}
    R\Gamma_{c}(\G,b,b',\mu)[\pi_{b}] := R\Gamma_{c}(\G,b,b',\mu) \otimes^{\mathbb{L}}_{\mathcal{H}(\G_{b}(\Q_p))} \pi_{b}
\end{equation}
where $\mathcal{H}(\G_{b}(\Q_p)) := C_c^{\infty}(\G_b(\Q_p), \Lambda)$ is the smooth Hecke algebra. Analogously, for $\pi_{b'}$ an irreducible smooth $\Lambda$-representation of $\G_{b'}(\Q_p)$, we can define $R\Gamma_{c}(\G,b,b',\mu)[\pi_{b'}]$ and $R\Gamma^{\flat}_{c}(\G,b,b',\mu)[\pi_{b'}]$. 
 
Recall that for each $b \in B(\GL_n)$, we can define a character $\kappa_b : \G_b(\Q_p) \longrightarrow \Lambda^{\times}$ and we have
\[
R\Gamma_c(\widetilde{\G}_b^{0}, \Lambda) = \kappa_b[-2d_b],
\]
as in \cite{GI}, lemma $4.18$ and as in \cite{Ham1}, page $88$, before lemma $11.1$. Remark that by \cite[Corollary 1.9(2)]{HI}, we have $\kappa_b = \delta^{-1}_b$ where $\delta_b$ is the modulus character associated to $b$ which we recall in \S \ref{itm : combinatoric description of Hecke operators}. By the same arguments as in \cite[Lemma 2.10]{GLn}, we can show the following result relating the above complexes to Hecke operators on $\Bun_{\G}$.

\begin{lemma}{\label{shimhecke}}{\cite[Section~IX.3]{FS}}
Given a local shtuka datum $(\G,b,b',\mu)$ as above and $\pi_{b'}$ (resp. $\pi_{b}$) a finitely generated smooth $\Lambda$-representation of $\G_{b'}(\mathbb{Q}_{p})$ (resp. $\G_{b}(\mathbb{Q}_{p})$), we can consider the associated sheaves $\pi_{b} \in \Dc(\G_{b}(\mathbb{Q}_{p}), \Lambda) \simeq \Dlis(\Bun_{\G}^{b})$ and $\pi_{b'} \in \Dlis(\Bun_{\G}^{b'}) \simeq \Dc(\G_{b'}(\mathbb{Q}_{p}), \Lambda)$ on the HN-strata $i_{b}: \Bun_{\G}^{b} \hookrightarrow \Bun_{\G}$ and $i_{b'}: \Bun_{\G}^{b'} \hookrightarrow \Bun_{\G}$. There then exists isomorphisms
\[ 
R\Gamma_{c}(\G,b,b',\mu)[\pi_{b} \otimes \kappa_b^{-1}][2d_{b}] \simeq  i_{b'}^{*}\T_{\mu}i_{b!}(\pi_{b}),
\]
\[
R\Gamma^{\flat}_{c}(\G,b,b',\mu)[\pi_{b}][-2d_{b'}] \simeq  i_{b'}^{!}\T_{\mu}Ri_{b*}(\pi_{b})\otimes\kappa_b^{-1},
\]
of complexes of $\G_{b'}(\mathbb{Q}_{p}) \times W_{E_{\mu}}$-modules and isomorphisms
\[ 
R\Gamma_{c}(\G,b,b',\mu)[\pi_{b'}\otimes \kappa_{b'}^{-1}][2d_{b'}] \simeq i_{b}^{*}\T_{-\mu}i_{b'!}(\pi_{b'}), \]
\[
R\Gamma^{\flat}_{c}(\G,b,b',\mu)[\pi_{b'}][-2d_{b}] \simeq i_{b}^{!}\T_{-\mu}Ri_{b'*}(\pi_{b'}) \otimes \kappa_{b'}^{-1}, 
\]
of complexes of $\G_{b}(\mathbb{Q}_{p}) \times W_{E_{\mu}}$-modules, where $-\mu$ is the dominant cocharacter conjugate to the inverse of $\mu$ and where $d_b = \langle 2\rho, \nu_b \rangle$ and $d_{b'} = \langle 2\rho, \nu_{b'} \rangle$. 
\end{lemma}

We can show that the complexes defined in equations $(\ref{eq: RGammaflat})$ and $(\ref{eq: RGamma})$ will be valued in smooth admissible representations of finite length by using the same arguments as in \cite[Theorem~I.7.2]{FS} and \cite[Pages~322, 323]{FS} . More precisely, the Hecke operators $ \T_{\mu} $ and the functors $i_{b!}, i^*_{b'}$ preserve compact and ULA objects for all $b, b'$. Then lemma $\ref{shimhecke}$ above implies that the cohomology groups of $R\Gamma_{c}(\G,b,b',\mu)[\pi_{b}]$ are compact and ULA in the category of smooth $\G_{b'}(\Q_p)$-representations with $\Lambda$-coefficients. Now, by pulling through Verdier duality as in \cite[Pages~323]{FS}, we deduce the same conclusion for $R\Gamma^{\flat}_{c}(\G,b,b',\mu)[\pi_{b}]$.

\section{Stacks of $L$-parameters}
 We fix a nonarchimedean local field $E$ with residue field $\F_q$ of characteristic $p \neq 2$ and an odd prime $\ell \neq p$. Let $\G$ be a reductive group over $E$ and then we have the dual group $\widehat{\G}/ \Z_{\ell}$, which we endow with its usual algebraic action of the Weil group $W_{E}$. 

\subsection{Some connected components of the stack of $L$-parameters} 
 
 In this paragraph we recall the definition of the stack of $L$-parameters and recollect some of its geometric properties. Let $A$ be any $\Z_{\ell}$-algebra, regarded as a condensed $\Z_{\ell}$-algebra. 
 
\begin{definition}
    An $L$-parameter for $\G$ with coefficients in $A$ is a section
    \[
    \phi : W_{E} \longrightarrow \widehat{G}(A) \rtimes W_{E}
    \]
    of the natural map of condensed groups $ \widehat{G}(A) \rtimes W_{E} \longrightarrow W_{E} $. Equivalently, an $L$-parameter for $\G$ with coefficients in $A$ is a condensed $1$-cocycle $\phi : W_{E} \longrightarrow \widehat{\G}(A)$ for the given $W_{E}$ action. More concretely, if $A$ is endowed with the topology given by writing $A = \textrm{colim}_{A' \subset A} A'$ where $A'$ is finitely generated $\Z_{\ell}$-module with its $\ell$-adic topology then an $L$-parameter with values in $A$ is a $1$-cocycle $ \phi : W_{E} \longrightarrow \widehat{\G}(A)$ such that if $\widehat{\G} \hookrightarrow \GL_n $ the associated map $ \phi : W_{E} \longrightarrow \GL_n(A) $ is continuous.
\end{definition}

With this definition of $L$-parameter over any $\Z_{\ell}$-algebra $A$, we can define a moduli space, whose $A$-points are the continuous $1$-cocycles $ \phi : W_{E} \longrightarrow \widehat{\G}(A) $ with respect to the natural action of $W_{E}$ on $\widehat{\G}(A)$. This defines the scheme considered in \cite{FS, DH, Zhu1} and we denote this scheme by $ Z^1(W_{E}, \widehat{\G})_{\Z_{\ell}} $, over $\Z_{\ell}$.

By a discretization process, Fargues and Scholze show \cite[ Theorem I.9.1 ]{FS} that $ Z^1(W_E, \widehat{\G})_{\Z_{\ell}} $ can be written as a union of open and closed affine subschemes $ Z^1(W_{E}/\rP, \widehat{\G})_{\Z_{\ell}} $ as $\rP$ runs through subgroups of the wild inertia $\rP_E$ of $W_{E}$, where $ Z^1(W_{E}/\rP, \widehat{\G})_{\Z_{\ell}} $ parametrizes condensed $1$-cocycles that are trivial on $\rP$. We also denote its base change to $\overline{\Q}_{\ell}$ or to $\overline{\F}_{\ell}$ by $ Z^1(W_{E}/\rP, \widehat{\G})_{\overline{\Q}_{\ell}} $ and $ Z^1(W_{E}/\rP, \widehat{\G})_{\overline{\F}_{\ell}} $, respectively. This allows us to consider the Artin stack quotient $[Z^1(W_{E}, \widehat{\G})_{\Z_{\ell}}/\widehat{\G} ]$, where $\widehat{\G}$ acts via conjugation. We then consider the base change to $\overline{\Q}_{\ell}$ or to $\overline{\F}_{\ell}$, denoted by $[Z^1(W_{E}, \widehat{\G})_{\overline{\Q}_{\ell}}/\widehat{\G} ]$ or $[Z^1(W_{E}, \widehat{\G})_{\overline{\F}_{\ell}}/\widehat{\G} ]$ respectively and refer to it as the stack of Langlands parameters, as well as the categories $\Perf([Z^1(W_{E}, \widehat{\G})_{\overline{\Q}_{\ell}}/\widehat{\G}])$ and $\Perf([Z^1(W_{E}, \widehat{\G})_{\overline{\F}_{\ell}}/\widehat{\G}])$ of perfect complexes of sheaves on these spaces.

From now on, we suppose $\G = \Res_{F/\Q_p}(\GL_{n, F})$ where $F$ is an unramified extension of $\Q_p$ of degree $d$.  We fix a lift $\sigma$ of the arithmetic Frobenius of $F$ to $W_{F}/\rP_{F}$ and $\tau$ a topological generator of $\I_{F}/\rP_{F}$. Let $\phi$ be an $L$-parameter of $\G$ (over $\Q_p$), by Shapiro's lemma it corresponds to a representation $\Tilde{\phi} : W_{F} \longrightarrow \widehat{\GL_n}(\Lambda)$. We study the connected components of $[Z^1(W_{\Q_p}, \widehat{\G})_{\Lambda}/\widehat{\G} ]$ and describe the components containing an $L$-parameter $\phi$ such that $ \Tilde{\phi}_{| \I_F}$ is multiplicity free (when considered as an $\I_F$-representation), where $\Lambda \in \{ \overline{\Q}_{\ell}, \overline{\F}_{\ell} \}$.  

\begin{proposition} \phantomsection \label{itm : simple connected components}
    Let $\phi$ be an $L$-parameter of $\G$ such that $\Tilde{\phi} \simeq \Tilde{\phi}_1 \oplus \dotsc \oplus \Tilde{\phi}_r$ where for $1 \leq i \leq r$, $\Tilde{\phi}_i$ is an irreducible representation of $W_F$ and for $1 \leq i \neq j \leq r$, there does not exist an unramified character $\chi$ of $W_F$ such that $\Tilde{\phi}_i \simeq \Tilde{\phi}_j \otimes \chi$. Let $[C_{\phi}]$ be the connected component containing $\phi$ in $[Z^1(W_{\Q_p}, \widehat{\G})_{\Lambda}/\widehat{\G} ]$. For $1 \leq i \leq r$, we denote by $f_i$ the number of irreducible factors of $\Tilde{\phi}_{i| \I_F}$. Then we have
    \[
    [C_{\phi}] \simeq [(\mathbb{G}^r_{m} \times \mu_{\Lambda}) / \mathbb{G}^r_m],
    \]
    the quotient stack of $(\mathbb{G}^r_{m} \times \mu_{ \Lambda})$ by the trivial action of $\mathbb{G}^r_{m}$ and $\mu_{ \Lambda} = \Spec \overline{\Q}_{\ell}$ if $\Lambda = \overline{\Q}_{\ell}$ and $\mu_{\Lambda} = \displaystyle \prod_{i = 1}^r \Spec \overline{\F}_{\ell}[\Syl_i]$ where $\Syl_i$ is the $\ell$-Sylow group of $ \F^{\times}_{p^{f_i}} $ if $\Lambda = \ov \F_{\ell}$.
\end{proposition}
\begin{proof}
By the compatibility of the formalism of $\Bun_{\G}$, of the stacks of $L$-parameters and of the Fargues-Scholze correspondence with the Weil restriction \cite[Section IX.6.3]{FS}, it is enough to consider the group $ \GL_{n, F}$ over $F$. More precisely, $[C_{\phi}]$ is isomorphic to the connected component $[C_{\Tilde{\phi}}]$ containing $\Tilde{\phi}$ of $ [Z^1(W_{F}, \widehat{\GL_{n, F}})_{\Lambda}/\widehat{\GL_{n, F}} ] $.

 Let $(W/\rP_{F})^0$ be the subgroup of $W_{F}/\rP_{F}$ generated by a pro-generator $\tau$ of $\I_F / \rP_F$, a lift $\sigma$ of the arithmetic Frobenius to $W_F / \rP_F$ and $\rP_{F}$, regarded as discrete group. Let $W^0_F$ be the inverse image of $(W/\rP_{F})^0$ in $W_{F}$. We can define $(\I_F/\rP_F)^0$ and $\I^0_F$ in a similar way by considering only $\tau $ and $\rP_{F}$. The functor that sends $R$ to $Z^1(W^0_F, \widehat{\GL_{n, F}}(R))$ is represented by an affine scheme denoted by $Z^1(W^0_{F}, \widehat{\GL_{n, F}})_{\Lambda}$. We choose a decreasing sequence $(\rP^e)_{e \in \N}$ of normal open subgroups of the wild inertia $\rP_{F}$ whose intersection is $\{ 1 \}$. We know that $ Z^1(W_{F}, \widehat{\GL_{n, F}})_{\Lambda} $ can be written as a union of open and closed affine sub-schemes $ Z^1(W^0_{F}/\rP^e, \widehat{\GL_{n, F}})_{\Lambda} $ for $e \in \N$. The conclusion of the proposition in the case $ \Lambda = \overline{\Q}_{\ell} $ is proved in \cite[Proposition 3.6]{GLn}, so we suppose $\Lambda = \overline{\F}_{\ell}$ and omit the subscript $\Lambda$ from now on.

We denote by $\I^{\ell}_{F}$ the maximal closed subgroup of $\I_{F}$ with prime-to-$\ell$ pro-order and denote by $\tau_{\ell}$ the pro-generator of $\I_F / \I^{\ell}_F$. If $V$ is a vector space together with a morphism $ f : \I^{\ell}_{F} \longrightarrow \Aut(V) $ then we denote by $V^{\tau_{\ell}}$ the representation of $\I^{\ell}_{F}$ whose underlying vector space is $V$ and where $ i \in \I^{\ell}_{F}$ acts by $ f (\tau_{\ell} i \tau_{\ell}^{-1} ) $. Denote by $V_i$ the underlying vector space of $\widetilde{\phi}_i$ and $ \displaystyle V = \bigoplus_{i = 1}^k V_i$. Since the action of $\tau_{\ell}$ gives an isomorphism between the $\I^{\ell}_{F}$-representations $ V $ and $ V^{\tau_{\ell}} $, by the same argument as in \cite[Lemma 3.6, page 22]{GLn}, we see that for each $1 \leq i \leq m $, the restriction $ \Tilde{\phi}_{i | \I^{\ell}_F}$ is multiplicity free. Moreover, if $\Tilde{\phi}_{| \I^{\ell}_{F}}$ is not a multiplicity free $\I^{\ell}_F$-representation then the same argument as in \cite[Lemma 3.7]{GLn} implies that for some $1 \leq i \neq j \leq m $ there exists a character $\chi$ of $ \I_{F} $ that is trivial on an open normal subgroup $\I^{\ell , e}_{F}$ containing $\I^{\ell}_{F}$ such that $\Tilde{\phi}_i \simeq \chi \otimes \Tilde{\phi}_j$. However, $\I_F / \I^{\ell, e}_F$ is an $\ell$-group then every continuous character $\chi : \I_F / \I^{\ell , e}_F \longrightarrow \overline{\F}^{\times}_{\ell}$ is trivial. It contradicts the assumption that $\Tilde{\phi}_{| \I_{F}}$ is a multiplicity-free $\I_F$-representation. Hence $\Tilde{\phi}_{| \I^{\ell}_F}$ is a multiplicity-free $\I^{\ell}_F$-representation and we can suppose that as a $\I^{\ell}_F$-representation, $\Tilde{\phi} = \Tilde{\varphi_1} \oplus \dotsc \oplus \Tilde{\varphi_k}$ for some $k \in \N$ and where the $\Tilde{\varphi_i}$'s are pairwise distinct irreducible representations of $\I^{\ell}_F$.

For each $e \in \N$, by a reduction to tame parameters argument \cite[Theorem 3.1, equation 4.2, section 4.2]{DH}, there exists a finite set $\Phi_e^{\ell}$ consisting of $1$-cocycle $ \varphi : \I^{\ell}_{F} / \rP^e \longrightarrow \widehat{\GL_{n, F}}(\overline{\F}_{\ell})$ such that we have a decomposition
\[
[Z^1(W^0_{F}/\rP^e, \widehat{\GL_{n, F}})/\widehat{\GL_{n, F}}] = \coprod_{\varphi \in \Phi_e^{\ell}} [Z^1(W^0_{F}, \widehat{\GL_{n, F}})_{\varphi}/C(\varphi)]
\]
where $C(\varphi)$ is the centralizer of $\varphi$ and $Z^1(W^0_{F}, \widehat{\GL_{n, F}})_{\varphi}$ is the (non empty) sub-scheme of $Z^1(W^0_{F}, \widehat{\GL_{n, F}})$ parametrizing the $1$-cocycles that can be extended to a $1$-cocycles $ \varphi' : W_F \longrightarrow \widehat{\GL_{n, F}} $ such that $\varphi'_{| \ \I^{\ell}_{F} / \rP^e} \simeq \varphi$.

Thus, there exist $e \in \N$ and $\varphi \in \Phi_e^{\ell}$ such that $[C_{\Tilde{\phi}}]$ is a connected component of $[Z^1(W^0_{F}, \widehat{\GL_{n, F}})_{\varphi}/C(\varphi)]$. Since $\Tilde{\phi}_{| W_F^0}$ is semi-simple, by \cite[Corollary 4.16]{DH} its unique continuous extension to $W_F$ is $\Tilde{\phi}$. Hence, we have $\Tilde{\phi}_{| \I^{\ell}_{F}} \simeq \varphi$, in particular $C(\varphi) \simeq \mathbb{G}_m^k$. Moreover, we have an injection $\mathbb{G}_m^s \simeq C(\Tilde{\phi}_{| \I_{F}}) \hookrightarrow C(\varphi)$ where $s = \displaystyle \sum_{i = 1}^r f_i$. 

Let $\varphi' : W_F \longrightarrow \widehat{GL_{n, F}}(\overline{\F}_{\ell})$ be a representation extending $\varphi \simeq \Tilde{\phi}_{| \I^{\ell}_{F}} $. By the same argument $V \simeq V^{\tau_{\ell}} $ as $\I_F^{\ell}$-representations, we deduce that 
the semi-simplification of $\varphi'$ as $\I_{F}$-representation is isomorphic to $\Tilde{\phi}_{| \I_{F}}$ and then, by an analogue of \cite[Proposition 3.4]{GLn}, we deduce that $\varphi'$ is semi-simple. 

The conjugation action of $W^0_{F}$ on $\GL_{n, F}(\overline{\F}_{\ell})$ by $\Tilde{\phi}$ stabilizes $C(\varphi)$ and the restricted action on this group factors over $W^0_{F} / \rP_{F}$. Denoting by $\mathrm{Ad}_{\Tilde{\phi}}$ this action and denoting by $Z^1_{\mathrm{Ad}_{\Tilde{\phi}}}(W^0_{F}/\rP_{F}, C(\varphi) )_{\ell}$ the affine scheme of $1$-cocycles $f : W^0_{F}/\rP_{F} \longrightarrow C(\varphi) $ such that $f(\tau) \in C(\Tilde{\phi}_{| \I_F}) $. Note that for any $w \in W^0_{F}$ and any element $\phi' \in Z^1(W^0_{F}, \GL_{n, F})_{\varphi}$ we can write $ \phi'(w) = \eta(w) \Tilde{\phi}(w) $ and by Schur's lemma, we can deduce that $\eta(w)$ belongs to $C(\varphi)$. Moreover $\phi'$ and $\Tilde{\varphi}$ can be extended to $W_F$ and its restriction to $\I_F^{\ell}$ are isomorphic, we deduce that $\eta(\tau^{\ell^m})$ is an isomorphism of $\Tilde{\phi}_{| I_{F}}$ for some $m$. As the base field is of characteristic positive $\ell$, we deduce that $\eta(\tau)$ is already an isomorphism of $\Tilde{\phi}_{| \I_{F}}$ and then $\eta(\tau) \in C(\Tilde{\phi}_{| \I_{F}})$. Hence, as in \cite[page 20]{DH}, the map $ \eta \longmapsto \eta \cdot \Tilde{\phi} $ sets up an isomorphism of $\overline{\F}_{\ell}$-schemes 
\[
Z^1_{\mathrm{Ad}_{\Tilde{\phi}}}(W^0_{F}/\rP_{F}, C(\varphi) )_{\ell} \xrightarrow{ \ \simeq \ }   Z^1(W^0_{F}, \GL_{n, F})_{\varphi}.
\]

Thus the scheme $Z^1_{\mathrm{Ad}_{\Tilde{\phi}}}(W^0_{F}/\rP_{F}, C(\varphi) )_{\ell}$ is isomorphic to the closed sub-scheme of $C(\Tilde{\phi}) \times C(\varphi)$ cut out by the equations deduced from the equality $ \sigma \tau \sigma^{-1} = \tau^q $. 

We suppose first that $\Tilde{\phi}$ is irreducible. If $W$ is a vector space together with a morphism $ f : \I_F \longrightarrow \Aut(W) $ then we denote by $W^{\sigma}$ the representation of $\I_F$ whose underlying vector space is $V$ and where $ i \in \I_F$ acts by $ f (\sigma i \sigma^{-1} ) $. Thus, as in \cite[Proposition 3.6]{GLn} we see that $\Tilde{\phi}_{| \I_F} \simeq V \oplus V^{\sigma} \oplus \dotsc \oplus V^{\sigma^{s-1}}$ where $V$ is an irreducible representation of $\I_F$ and $s$ is the smallest natural number such that $V \simeq V^{\sigma^s}$. Similarly, we have a decomposition $V_{| \I^{\ell}_F} \simeq W \oplus W^{\tau} \oplus \dotsc \oplus W^{\tau^{s'-1}}$ where $W$ is an irreducible representation of $\I^{\ell}_F$ and $s'$ is the smallest natural number such that $W \simeq W^{\tau^{s'}}$. In particular we have $ss' = k$. By writing the matrices of $\tau$ and $\sigma$ in the basic given by basics of $W$'s and $V$'s we can describe the action of $W_F$ on $C(\varphi)$ and also the action of $C(\varphi)$ on $Z^1_{\mathrm{Ad}_{\Tilde{\phi}}}(W_{F}/\I^{\ell}_{F}, C(\varphi) )$ coming from the conjugation action of $C(\varphi)$ on $Z^1(W_{F}, \GL_{n, F})_{\varphi}$ as in \cite[Proposition 3.6]{GLn}. Thus we have
\[
[Z^1(W^0_{F}, \widehat{\GL_{n, F}})_{\varphi}/C(\varphi)] \simeq [\bb G_m \times \mu_{q^{s}-1} / \bb G_m],
\]
the quotient of $\bb G_m \times \mu_{q^{s}-1}$ by the trivial action of $\bb G_m$ and where $\mu_{q^{s}-1}$ is the group scheme of $(q^s-1)$-roots of unity. Hence the connected component containing $\Tilde{\phi}$ is given by
\[
[\bb G_m \times \Spec \overline{\F}_{\ell}[\Syl] / \bb G_m]
\]
where $\Syl$ is the $\ell$-Sylow group of $\overline{\F}^{\times}_{q^s}$. In general when $\Tilde{\phi}$ is multiplicity free, we see that the connected component of $\Tilde{\phi}$ is given by
\[
     [(\mathbb{G}^r_m \times \mu) / \mathbb{G}^r_m].
\]
\end{proof}

\subsection{Perfect complexes on the stack of $L$-parameters}

In this subsection we talk about the spectral action defined in \cite{FS}. Recall that we suppose $\G = \Res_{F/\Q_p}(\GL_{n, F})$ where $F$ is an unramified extension of $\Q_p$ of degree $d$ and $\Lambda \in \{\overline{\Q}_{\ell}, \overline{\F}_{\ell} \}$. Denote by $\Perf([Z^1(W_{\Q_p}, \widehat{\G})_{\Lambda}/ \widehat{\G}])$ the derived category of perfect complexes on $[Z^1(W_{\Q_p}, \widehat{\G})_{\Lambda}/ \widehat{\G}]$ where $\widehat{\G}$ is the reductive group defined over $\Lambda$ dual of $\G$. We write $\Perf([Z^1(W_{\Q_p}, \widehat{\G})_{\Lambda}/\widehat{\G}])^{BW_{F}^{\I}}$ for the derived category of objects with a continuous $W_{F}^{\I}$ action for a finite index set $\I$. 

We will be interested in the derived category $\Dlis(\Bun_{\G}, \Lambda)$ and $\Dlis(\Bun_{\G}, \Lambda)^{\omega}$.


By \cite[Corollary~X.1.3]{FS}, there exists a $\Lambda$-linear action 
\[ \Perf([Z^1(W_{\Q_p}, \widehat{\G})_{\Lambda}/\widehat{\G}])^{BW_{F}^{I}} \ra \mathrm{End}(\Dlis(\Bun_{\G}, \Lambda)^{\omega})^{BW_{F}^{\I}} \]
\[ C \mapsto \{A \mapsto C \star A\}\]
which, extending by colimits, gives rise to an action 
\[ \IndPerf([Z^1(W_{\Q_p}, \widehat{\G})_{\Lambda}/\widehat{\G}])^{BW_{F}^{I}} \ra \mathrm{End}(\Dlis(\Bun_{\G},\Lambda)^{\omega})^{BW_{F}^{\I}} \]
where $\IndPerf([Z^1(W_{\Q_p}, \widehat{\G})_{\Lambda}/\widehat{\G}])$ is the triangulated category of Ind-Perfect complexes, and this action is uniquely characterized by some complicated properties. For our purposes, we will need the following:
\begin{enumerate}
    \item For $V = \boxtimes_{i \in I} V_{i}\in \Rep_{\Lambda}(\phantom{}^{L}\G(\Lambda)^{I})$, there is an attached vector bundle $C_{V} \in \Perf([Z^1(W_{\Q_p}, \widehat{\G})_{\Lambda}/\widehat{\G}])^{BW_{F}^{I}}$ whose evaluation at a $\Lambda$-point of $[Z^1(W_{\Q_p}, \widehat{\G})_{\Lambda}/\widehat{\G}]$ corresponding to a (not necessarily semi-simple) $L$-parameter $\phi: W_{\mathbb{Q}_{p}} \ra \phantom{}^{L}\G(\Lambda)$ is the vector space $V$ with $W_{F}^{I}$-action given by $\boxtimes_{i \in I} V_{i} \circ \phi$. Then the Hecke operator $\T_{V}$ defined in the previous sections, is given by the endomorphism
    \[ C_{V} \star (-): \Dlis(\Bun_{\G},\Lambda) \ra \Dlis(\Bun_{\G}, \Lambda)^{BW_{F}^{I}}. \] 
    by compatibility between Hecke operators and spectral action \cite[Theorem X.1.1]{FS}.
    \item The action is symmetric monoidal in the sense that given $C_{1},C_{2} \in \IndPerf([Z^1(W_{\Q_p}, \widehat{\G})_{\Lambda}/\widehat{\G}])$, we have a natural equivalence of endofunctors: 
    \[ (C_{1} \otimes^{\mathbb{L}} C_{2}) \star (-) \simeq C_{1} \star (C_{2} \star (-)). \]
\end{enumerate}

We remark that following \cite{FS}, we can replace $BW_{F}$ by $BW_{\Q_p}$ in the above paragraphs as our group $\G$ is defined over $\Q_p$. However, for our purpose, we only need to consider $BW_{F}$.

Let $\phi = \phi_1 \oplus \dotsc \oplus \phi_r$ be an $L$-parameter of $\G$ satisfying the conditions in proposition \ref{itm : simple connected components}. We denote by $[C_{\phi}]$ the connected component containing $\phi$ and we know that $[C_{\phi}]$ consists of the $L$-parameters of the form $\phi_1 \otimes \xi_1 \oplus \dotsc \oplus \phi_r \otimes \xi_r$ where the $\xi_i$'s are unramified characters. This connected component gives rise to a direct summand
\[
\Perf([C_{\phi}]) \hookrightarrow \Perf([Z^1(W_{\Q_p}, \widehat{\G})/\widehat{\G}]).
\]

Therefore the spectral action gives rise to a corresponding direct summand 
\[
\Dlis^{[C_{\phi}]}(\Bun_{\G},\Lambda)^{\omega} \subset \Dlis(\Bun_{\G},\Lambda)^{\omega},
\]
explicitly given as those objects on which the excursion operator corresponding to the function that
is $1$ on $[C_{\phi}]$ and $0$ elsewhere acts via the identity. In particular, the Schur-irreducible objects in this subcategory all have Fargues-Scholze parameter given by an $L$-parameter in $[C_{\phi}]$.

Now we consider the vector bundles on the connected component $[C_{\phi}]$. By proposition $\ref{itm : simple connected components}$ we have $[C_{\phi}] \simeq [(\mathbb{G}^r_{m} \times \mu_{\Lambda}) / \mathbb{G}^r_m] $ where $ \bb G^r_m $ acts trivially. Let $\Irr(S_{\phi})$ be the set of irreducible algebraic representations of $S_{\phi}$. For each $\chi \in \Irr(S_{\phi}) = \Irr(\bb G^r_m)$ we have an associated vector bundle $C_{\chi}$ on $[(\mathbb{G}^r_{m} \times \mu_{\Lambda}) / \mathbb{G}^r_m]$. More precisely the trivial line bundle on $\bb G_m^r$ together with the action of $\bb G_m^r$ defined by $\chi$ at every fibers gives rise to a $\bb G_m^r$-equivariant vector bundle on $\bb G_m^r$ and this vector bundle corresponds to $C_{\chi}$ on $[(\mathbb{G}^r_{m} \times \mu_{\Lambda}) / \mathbb{G}^r_m]$. Since $\bb G_m^r$ is affine, we see that $C_{\chi}$ is projective. Therefore it is clear that for $\chi, \chi' \in \Irr(S_{\phi})$ we have $ C_{\chi \otimes \chi'} \simeq C_{\chi} \otimes C_{\chi'} \simeq C_{\chi} \otimes^{\bb L} C_{\chi'} $. In particular, $ C_{\chi} \otimes C_{\chi^{-1}} \simeq C_{\chi^{-1}} \otimes C_{\chi} \simeq C_{\Id} $ is the identity functor of $\Dlis^{[C_{\phi}]}(\Bun_{\G}, \Lambda)^{\omega}$. Hence for an arbitrary $\chi \in \Irr(S_{\phi})$, the functor $ C_{\chi} \star (-) $ is an auto-equivalence of $\Dlis^{[C_{\phi}]}(\Bun_{\G}, \Lambda)^{\omega}$.

Let $V \in \Rep_{\Lambda}(\widehat{\G}(\Lambda))$, by the arguments in pages $339$--$340$ and theorem $X.1.1$ in \cite{FS}, we can express the action of the Hecke operator $\T_V$ on $\Dlis^{[C_{\phi}]}(\Bun_{\G}, \Lambda)^{\omega}$ in terms of the spectral action of the $C_{\chi}$'s for $\chi \in \Irr(S_{\phi})$. Recall that the action of $\T_V$ is given by the vector bundle $C_V$ defined above. By \cite[lemma 3.8]{Ham}, it is enough to consider the restriction of $C_V$ to $[(\mathbb{G}^r_{m} \times \mu_{\Lambda}) / \mathbb{G}^r_m]$ and for simplicity we also use $C_V$ to denote the restriction of $ C_V $ to $[(\mathbb{G}^r_{m} \times \mu_{\Lambda}) / \mathbb{G}^r_m]$.

The restriction of $V$ to $S_{\phi}$ admits a commuting $W_{F}$-action given by $\phi$. This defines a monoidal functor
\[
\Rep_{\Lambda}(\widehat{\G}(\Lambda)) \longrightarrow \Rep_{\Lambda}(S_{\phi})^{BW_{F}}.
\]

Note that as a $S_{\phi} \times W_{F}$-representation, we can decompose $V$ into a direct sum
\[
V \simeq \bigoplus_{\chi \in \Irr(S_{\phi})} \chi \boxtimes \sigma_{\chi},
\]
where $\sigma_{\chi}$ is the $W_F$-representation $\Hom_{S_{\phi}}(\chi, V)$.  If we forget the $W_F$-action, we see that
\[
\T_V = \bigoplus_{\chi \in \Irr(S_{\phi})} C^{\dim \sigma_{\chi}}_{\chi}.
\]

However, $\T_V$ has an action of $W_F$ and we want to understand this action. Recall that $W_F$ acts on $C_V$ and this action gives rise to the action of $W_F$ on $\T_V$. By the concrete description of $W_F$ on the fibers of $C_V$ and remark that $ \phi'_{| \I_F} \simeq \phi_{| \I_F} $ for every $\phi'$ in $[C_{\phi}]$, we deduce that the action of $\I_F$ on $\T_V$ can be expressed as $ \displaystyle \T_V = \bigoplus_{\chi \in \Irr(S_{\phi})} C_{\chi} \boxtimes \sigma_{\chi | \I_F} $. In other words, the following diagram of monoidal functors commutes
\begin{center}
     \begin{tikzpicture}
     \draw (0,0) node [above] {} node{$\Rep_{\Lambda}(\widehat{\G}(\Lambda))$};
     \draw (7,0) node [above] {} node{$\End_{\Lambda}( \Dlis^{[C_{\phi}]}(\Bun_{\G}, \Lambda)^{\omega} )^{B\I_{F}}$};
     \draw (0,-2.5) node [above] {} node{$\Rep_{\Lambda}(S_{\phi})^{B\I_F}$};
     \draw [->] (0,-0.5) -- (0,-2);
     \draw [->] (1.3,0) -- (3.9,0);
     \draw [->] (1.3,-2) -- (5,-0.5);
     \end{tikzpicture}
 \end{center}

By the same arguments as in \cite[Proposition 3.8]{GLn}, if we apply the above diagram to a Schur-irreducible sheaf $\mathcal{F}$ on $\Bun_{\G}$ whose $L$-parameter is given by $\phi$ then we can get precise information of the action of the Frobenius.

\begin{proposition} \phantomsection \label{itm : fundamental decomposition of Hecke operator}
Let $V \in \Rep_{\Lambda}(\widehat{\G}(\Lambda))$ and suppose that as a $S_{\phi} \times W_F$-representation, we can decompose $V$ into a direct sum
\[
V \simeq \bigoplus_{\chi \in \Irr(S_{\phi})} \chi \boxtimes \sigma_{\chi},
\]
where $\sigma_{\chi}$ is the $W_F$-representation $\Hom_{S_{\phi}}(\chi, V)$. Then for $\mathcal{F}$ a Schur-irreducible sheaf on $\Bun_{\G}$ whose $L$-parameter is given by $\phi$, we have 
\[
\T_V(\mathcal{F}) = \bigoplus_{\chi \in \Irr(S_{\phi})} C_{\chi} \star \mathcal{F} \boxtimes \sigma_{\chi}.
\]   
as object in $\Dlis^{[C_{\phi}]}((\Bun_{\G}, \Lambda)^{\omega})^{BW_F}$.
\end{proposition}
\begin{proof}
 Note that the category $\Rep_{\Lambda}(\widehat{\G}(\Lambda))$ is semisimple if $\Lambda = \overline{\Q}_{\ell}$ but it is no longer semisimple if $\Lambda = \overline{\F}_{\ell}$. However we can still apply the arguments in \cite[Proposition 3.8]{GLn} to this situation. 
 
 First of all, we can show that the conclusion of the proposition holds when $V \simeq \std$ is the standard representation. Then any irreducible algebraic representation $W$ of $\G(\Lambda)$ is a subquotient of $ \std^m \otimes (\std^{\vee})^{m'} $ for some $m, m' \in \N$. However $W_{| S_{\phi}}$ is a direct summand of $ \std^m \otimes (\std^{\vee})^{m'}_{| S_{\phi}} $ as $S_{\phi} \simeq \mathbb{G}^r_m $-representations. Then the result for $W$ follows from the monoidal property of the Hecke operators.   
\end{proof}
 
\section{Modular representations of inner forms of $\GL_n$}
\subsection{Generalities on modular representations of inner forms of $\GL_n$} \textbf{}

Let $F$ be a finite extension of $\Q_p$ whose residue field is the finite field of $ q = p^d$ elements for some $d \in \N$. In this section, we recall some structure and classification results of the category of smooth representations with coefficients in $\Lambda \in \{\overline{\Q}_{\ell}, \overline{\F}_{\ell} \}$ of $\GL_n(F)$ and its inner forms, following the works \cite{ Tad90, Vig96, Vig98, MS14, SS16, DHKM24}.  

Let $D$ be some division algebra over $F$ of reduced degree $r$ and for each $m \in \N$, let $\G_m$ be the group $\GL_m(D)$, in particular $\G_m$ is an inner form of $\GL_{mr}$ over $F$. We denote by $\Rep_{\Lambda}(\G_m)$ the category of smooth representations of $\G_m$ with coefficients in $\Lambda$ and by $\Irr_{\Lambda}(\G_m)$ the set of isomorphism classes of irreducible representations of $\G_m$. An irreducible representation of $\G_m$ is cuspidal if it is not isomorphic to any quotient (or equivalently any sub-representation) of a parabolic induction of a representation of a proper Levi subgroup. It is supercuspidal if it is not isomorphic to any subquotient of a parabolic induction of an irreducible representation of a proper Levi subgroup. We remark that in the case $\Lambda = \overline{\F}_{\ell}$ and $\ell$ is not a banal prime, there exists irreducible representations that are cuspidal but are not supercuspidal. 

In \cite{MS14}, Minguez et Sécherre obtained two classifications of the smooth irreducible representations of $\G_m$: one by supercuspidal multisegments, classifying representations with a given supercuspidal support, and one by aperiodic multisegments, classifying representations with a given cuspidal support. For our goals, it is enough to recall the classification by supercuspidal multisegments.  


Let $\mathrm{N}_m$ be the reduced norm of $M_m(D)$ over $F$. We denote by $| \cdot |_{F}$ the normalised absolute value of $F$, in particular if $\omega$ is a uniformiser of $F$ then we have $|\omega|_F = q^{-1}$. Since $q$ is invertible in $\Lambda$, we have a character
\begin{align*}
    \nu : \G_m& \longrightarrow \Lambda \\
    g &\longmapsto | \mathrm{N}_m(g) |_F.
\end{align*}

Let $\rho$ be an irreducible cuspidal representation of $\G_m$ and let $\chi$ be an unramified character of $\G_m$. By \cite[Proposition 4.37]{MS}, there exists a unique positive integer number $s(\rho)$ such that the normalised parabolic induction $\Ind_{\rP}^{\G_{2m}}( \rho \times \chi\rho)$ is reducible if and only if $\chi = \nu^{s(\rho)}$ or $\chi = \nu^{-s(\rho)}$ and where $\rP$ is a parabolic subgroup of $\G_{2m}$ whose Levi factor is $\G_m \times \G_m$. By \cite[Proposition 3.13]{MS}, the number $s(\rho)$ depends only on the inertial class of $\rho$. We denote the character $\nu^{s(\rho)}$ by $\nu_{\rho}$.

\begin{definition}
\begin{enumerate}
    \item A segment is a finite sequence of the form 
    \[
    [a, b]_{\rho} := (\rho\nu^a_{\rho}, \rho\nu^{a+1}_{\rho}, \dotsc, \rho\nu^b_{\rho} ),
    \]
    where $ a \leq b $ are integers and $\rho$ is an irreducible cuspidal representation of some $\G_h$.
    \item Two segments $[a, b]_{\rho}$ and $[a', b']_{\rho'}$ are equivalent if they have the same length $n$ and for each $i \in \{ 0, 1, \dotsc, n-1 \}$ the representations $\rho\nu_{\rho}^{a+i}$ and $\rho'\nu_{\rho'}^{a'+i}$ are isomorphic. 
    \item The degree of the segment $[a, b]_{\rho}$ is the number $(b-a+1)h$ and its support is the set $\{ \rho\nu^a_{\rho}, \dotsc, \rho\nu^b_{\rho} \}$.
\end{enumerate}    
\end{definition}
For each segment $ [a, b]_{\rho} = (\rho\nu^a_{\rho}, \rho\nu^{a+1}_{\rho}, \dotsc, \rho\nu^b_{\rho} )$, one denotes by $\rZ ([a, b]_{\rho})$ the sub-representation of $\Ind_{\rP}^{\G_{m(b-a+1)}} ( \rho\nu^a_{\rho} \times \rho\nu^{a+1}_{\rho} \times \dotsc \times \rho\nu^b_{\rho} ) $ defined in \cite[Definition 7.5]{MS14}, where $\rP$ is a parabolic subgroup of $\G_{m(b-a+1)}$ whose Levi factor is given by $ (\G_m)^{b-a+1} $.
\begin{definition}
\begin{enumerate}
    \item A multisegment is a multi-set of equivalent classes of segments, that is to say an element of $\N(\mathrm{Seg})$ where $\mathrm{Seg}$ denotes the set of equivalent classes of segments. In other words, a multisegment is a map $\mathfrak{m}$ from the set of equivalence classes of segments to $\N$, with finite support.  
    \item We can define the notions of support and degree of a multisegment $\mathfrak{m}$ as above. A multisegment is called supercuspidal if its support is formed by supercuspidal representations. We denote by $\Mult$, respectively by $\Mult_m$ the set of supercuspidal multisegments, respectively its subset of supercuspidal multisegments of degree $m$.
\end{enumerate}
\end{definition}
For each multi-segment $\mathfrak{m} = \Delta_1 + \Delta_2 + \dotsc + \Delta_r$ where the $\Delta_i$'s are segments (not necessarily non-equivalent), we consider the normalised parabolic induction $ \I(\mathfrak{m}) := \Ind_{\rP}^{\G}(\rZ(\Delta_1) \times \dotsc \times \rZ(\Delta_r))$ and we denote by $\rZ(\mathfrak{m})$ the unique sub-quotient of $\I(\mathfrak{m})$ satisfying some properties as in \cite[Definition 9.24]{MS14}. When $\Lambda = \overline{\Q}_{\ell}$, the representation $ \rZ(\mathfrak{m}) $ is a generalized Speh representation. If for all $ 1 \leq i \neq j \leq r $ the segments $\Delta_i$ and $\Delta_j$ are not linked in the sense of \cite[Definition 7.3]{MS14} then by \cite[Theorem 7.24]{MS14}, the representation $\I(\mathfrak{m})$ is irreducible and hence $ \rZ(\mathfrak{m}) = \I(\mathfrak{m})$. 

Denote by $\Irr(\G_m)$ the isomorphism classes of smooth irreducible representations of $\G_m$, we have the following classification result.
\begin{theorem} (\cite[Proposition 9.32]{MS14})
    The map
    \begin{align*}
        \rZ : \Mult_m &\longrightarrow \Irr(\G_m) \\
        \mathfrak{m} &\longmapsto \rZ(\mathfrak{m}), 
    \end{align*}
    is a bijection.
\end{theorem}
Next, we recall the decomposition of the category $\Rep_{\Lambda}(\G_m)$ into blocks. A supercuspidal pair of a reductive group $\G_m$ is a pair $(\M, \rho)$ consisting of a Levi subgroup $\M$ of $\G_m$ and a supercuspidal representation $\rho$ of $\M$. Let $\pi$ be an irreducible representation of $\G_m$, the supercuspidal support of $\pi$ is the set of supercuspidal pairs $(\M, \rho)$ of $\G_m$ such that $\pi$ is a subquotient of the induction $\Ind_{\rP}^{\G_m}(\rho)$ for some parabolic subgroup $\rP$ whose Levis factor is $\M$. By \cite[Theorem 8.16]{MS14}, the supercuspidal support of $\pi$ consists of a single $\G_m$-conjugacy class of supercuspidal pair of $\G_m$. We remark that when $\Lambda = \overline{\F}_{\ell}$, the analogue result on supercuspidal support is not known for general reductive group $\G$.
\begin{definition}
\begin{enumerate}
    \item The inertial class of a supercuspidal pair $(\M, \rho)$ of $\G_m$ is the set, denoted by $[M, \rho]_{\G_m}$, of all supercuspidal pairs $(\M', \rho')$ that are $\G_m$-conjugate to $(\M, \chi\rho)$ for some unramified character $\chi$ of $\M$.
    \item Given $\mathfrak{s}$ an inertial class of a supercuspidal pair of $\G_m$, we denote by $\Rep_{\Lambda}(\mathfrak{s})$ the full subcategory of representations whose irreducible subquotients have their supercuspidal support in $\mathfrak{s}$. 
\end{enumerate}
\end{definition}
\begin{theorem} (\cite[Theorems 10.4, 10.5]{SS16}) \phantomsection \label{itm : decomposition into blocks}
    The category $\Rep_{\Lambda}(\G_m)$ decomposes into the product of indecomposable subcategries $\Rep_{\Lambda}(\mathfrak{s})$ where $\mathfrak{s}$ ranges over all the possible inertial supercuspidal class of supercuspidal pairs of $\G_m$.
\end{theorem}

Remark that if $\Lambda = \ov \F_{\ell}$ where $\ell$ is a banal prime for $\G_m$ then a general decomposition of $\Rep_{\Lambda}(\G_m)$ into blocks is obtained in \cite[section 4]{DHKM24}.

\subsection{Fargues-Scholze correspondence}   

Let $\pi$ be an irreducible representation of $\G_m$ with coefficient in $\Lambda \in \{ \overline{\Q}_{\ell}, \overline{\F}_{\ell} \} $. By using the excursion operators, Fargues and Scholze constructed a unique semisimple continuous map 
\[
\phi^{\mathrm{FS}}_{\pi}: W_{\mathbb{Q}_{p}} \rightarrow \phantom{}^{L}\G(\Lambda), 
\]
which is the Fargues--Scholze parameter of $\pi$. They show that their correspondence has various good properties which we will invoke throughout.
\begin{theorem}{\cite[Theorem~I.9.6]{FS}}{\label{FSproperties}}
The mapping defined above 
\[ \pi \mapsto \phi^{\mathrm{FS}}_{\pi} \]
enjoys the following properties:
\begin{enumerate}
    \item The correspondance is compatible with reduction mod $\ell$. More precisely, let $\ov \pi$ be an irreducible $\overline{\F}_{\ell}$-representation of $\G(\Q_p)$ and let $\pi$ be an irreducible $\overline{\Q}_{\ell}$-representation of $\G(\Q_p)$ admitting a $\G(\Q_p)$-stable $\overline{\Z}_{\ell}$-lattice such that $\ov \pi$ occurs as a subquotient of $\pi$ mod $\ell$. Then the Fargues-Scholze parameter $\phi_{\pi}^{\rm FS}$ factors through $ \phantom{}^{L}\G(\overline{\Z}_{\ell})$ and its reduction modulo $\ell$ is equal to $\phi^{\rm FS}_{\ov \pi}$.
    \item The correspondence is compatible with character twists, passage to contragredients, and central characters.
    \item (Compatibility with products) Given two irreducible representations $\pi_{1}$ and $\pi_{2}$ of two connected reductive groups $\G_{1}(\Q_p)$ and $\G_{2}(\Q_p)$ over $\mathbb{Q}_{p}$, respectively, we have
    \[ \pi_{1} \boxtimes \pi_{2} \mapsto \phi^{\mathrm{FS}}_{\pi_{1}} \times \phi^{\mathrm{FS}}_{\pi_{2}}\]
    under the Fargues--Scholze local Langlands correspondence for $\G_{1}(\Q_p) \times \G_{2}(\Q_p)$. 
    \item (Compatibility with parabolic induction) Given a parabolic subgroup $P \subset \G$ with Levi factor $M$ and a representation $\pi_{M}$ of $M$, then the Weil parameter corresponding to any sub-quotient of the (normalized) parabolic induction $I_{P}^{\G}(\pi_{M})$ is the composition
    \[ W_{\mathbb{Q}_{p}}\xrightarrow{\phi^{\mathrm{FS}}_{\pi_{M}}} \\  ^{L}M(\overline{\mathbb{Q}}_{\ell}) \rightarrow ^{L}\G(\overline{\mathbb{Q}}_{\ell}) \]
    where the map $\phantom{}^{L}M(\overline{\mathbb{Q}}_{\ell}) \rightarrow \phantom{}^{L}\G(\overline{\mathbb{Q}}_{\ell})$ is the natural embedding. 
\end{enumerate}
If $\Lambda = \overline{\Q}_{\ell}$ then we have the following properties:
\begin{enumerate}
    \item[(5)] (Compatibility with Local Class Field Theory) If $\G = T$ is a torus, then $\pi \mapsto \phi_{\pi}$ is the usual local Langlands correspondence 
    \item[(6)] (Compatibility with Harris--Taylor/Henniart LLC)
    For $\G = \GL_{n}$ or an inner form of $\G$ the Weil parameter associated to $\pi$ is the (semi-simplified) parameter $\phi_{\pi}$ associated to $\pi$ by Harris--Taylor/Henniart \cite[Theorem 6.6.1]{HKW}.
\end{enumerate}
\end{theorem}
Remark that Vignéras and Dat showed that there exists a bijection $\ov \pi \longmapsto \sigma( \ov \pi)$ between the isomorphism classes of supercuspidal $\overline{\F}_{\ell}$-representations of $\GL_n(F)$ and the isomorphism classes of irreducible $\overline{\F}_{\ell}$-representations of $W_F$ of dimension $n$ which is compatible with the reduction modulo $\ell$ of the local Langlands correspondance. Hence, by the properties $(1)$ and $(6)$ of the above theorem, we deduce that $ \sigma(\ov \pi) = \phi_{\ov \pi}^{\rm FS} $ for supercuspidal irreducible $\overline{\F}_{\ell}$-representation $\pi$.

\begin{lemma} \phantomsection \label{itm : existence, irreducible}
    Let $\ov \phi$ be an $L$-parameter with coefficient in $\Lambda$ of $\G := \Res_{F/ \Q_p}(\GL_{n, F})$ whose corresponding $L$-parameter of $\GL_{n, F}$ is irreducible and let $\G_m$ be an inner form of $\G$. Then there exists a unique irreducible representation $\ov \pi^{\G_m}_{\ov \phi}$ in $\Rep_{\Lambda}(\G_m(\Q_p))$ whose Fargues-Scholze parameter is given by $\ov \phi$.
\end{lemma}
\begin{proof}
    The conclusion of the lemma is clear if $\Lambda = \ov \Q_{\ell}$ by the local Langlands correspondance for $\GL_n$ and its inner forms. We consider the case $\Lambda = \ov \F_{\ell}$.

    By the results of Vignéras and Dat \cite[Theorem 1]{Dat12}, there exists a unique representation $\ov\pi^{\G}_{\ov\phi}$ in $\Rep_{\ov \F_{\ell}}(\G(\Q_p))$ whose Fargues-Scholze parameter is given by $\ov\phi$. Then by \cite[Theorem 3.26]{MS}, we can lift this representation to an $\ell$-integral irreducible supercuspidal representation $\pi^{\G}$ in $\Rep_{\ov \Q_{\ell}}(\G(\Q_p))$. Thus if we denote by $\phi$ the Fargues-Scholze parameter of $\pi^{\G}$ then by the first point of the theorem \ref{FSproperties}, the parameter $\phi$ is $\ell$-integral and its reduction modulo $\ell$ is given by $\ov \phi$.

    Now let $\pi^{\G_m}$ be the irreducible supercuspidal representation in $\Rep_{\ov \Q_{\ell}}(\G_m(\Q_p))$ whose Fargues-Scholze parameter is given by $\phi$. We want to show that this representation is $\ell$-integral and its reduction modulo $\ell$ is then an irreducible representation in $\Rep_{\ov \F_{\ell}}(\G_m(\Q_p))$ whose Fargues-Scholze parameter is given by $\ov\phi$.

    Indeed, by \cite[Proposition 2.9]{Han}, some unramified twist of $\pi^{\G_m}$ is $\ell$-integral and then by theorem \ref{FSproperties} again, its $L$-parameter is given by $\phi \otimes \chi $ and is $\ell$-integral for some unramified character $\chi$. Since $\phi \otimes \chi$ and $\phi$ are both $\ell$-integral, we deduce that $\chi$ takes values in $\ov \Z^{\times}_{\ell}$. Hence $\pi^{\G_m}$ is $\ell$-integral. 

    Now, if there are $2$ non-isomorphic irreducible representations $\ov\pi_1$ and $\ov\pi_2$ in $\Rep_{\ov \F_{\ell}}(\G_m(\Q_p))$ whose Fargues-Scholze parameter is given by $\ov\phi$ then they are supercuspidal by the compatibility of Fargues-Scholze construction with sub-quotients of parabolic inductions. The connected component of the stack of $L$-parameter containing $\ov\phi$ is isomorphic to 
    \[
    [C_{\ov\phi}] \simeq [(\mathbb{G}_{m} \times \mu_{\Lambda}) / \mathbb{G}_m],
    \]
    and for each character $\chi \in \Irr(\Gm) \simeq \Z$, we have a vector bundle $C_{\chi}$ on $[C_{\ov\phi}]$ and $ C_{\chi} \star (-) $ is an auto-equivalence of $\Dlis^{[C_{\ov\phi}]}(\Bun_{\G}, \Lambda)^{\omega}$. Moreover, we have a decomposition 
    \[
    \Dlis^{[C_{\ov\phi}]}(\Bun_{\G}, \Lambda)^{\omega} = \bigoplus_{\chi \in \Irr(\Gm)} \Dlis^{[C_{\ov\phi}]}(\Bun^{b_{\chi}}_{\G}, \Lambda)^{\omega},
    \]
    where $\Bun^{b_{\chi}}_{\G}$ is the open sub-stack of $\Bun_{\G}$ corresponding to the basic element $b_{\chi}$ index by $\chi$. We can suppose that $\G_{b_{\chi}}(\Q_p) \simeq \G_m(\Q_p)$ for a fixed $\chi$.
    
    Since $\ov\pi_1$ and $\ov\pi_2$ are supercuspidal, we have $ \Hom_{\Dlis(\Bun_{\G}, \Lambda)}(i_{b_{\chi}!}\ov\pi_1[m_1], i_{b_{\chi}!}\ov\pi_2[m_2]) = 0 $ for every $m_1, m_2 \in \Z$. Thus $ \Hom_{\Dlis(\Bun_{\G}, \Lambda)}( C_{\chi^{-1}} \star i_{b_{\chi}!}\ov\pi_1[m_1], C_{\chi^{-1}} \star i_{b_{\chi}!}\ov\pi_2[m_2] ) = 0 $ for every $m_1, m_2 \in \Z$.

    However, $C_{\chi^{-1}} \star i_{b_{\chi}!}\ov\pi_i$ (for $i=1, 2$) are sheaves supported on the stratum corresponding to $1 \in B(\G)$. Moreover they are ULA and compact sheaves and the Fargues-Scholze parameter of a Schur-irreducible constituent is exactly $\ov\phi$. But it contradicts the fact that there exists a unique representation $\ov\pi^{\G}_{\ov\phi}$ in $\Rep_{\ov \F_{\ell}}(\G(\Q_p))$ whose Fargues-Scholze parameter is given by $\ov\phi$ \cite[Theorem 1]{Dat12}. Hence there is a unique irreducible representation in $\Rep_{\ov \F_{\ell}}(\G_m(\Q_p))$ whose Fargues-Scholze parameter is given by $\ov\phi$.
\end{proof}

\begin{lemma} \phantomsection \label{itm : parameter of divison algebra}
    Let $D$ be a division algebra over $F$ such that $D^{\times}$ is an inner form of $\GL_{n, F}$. Let $\ell$ be a banal prime and let $\ov \pi$ be a supercuspidal representation in $\Rep_{\ov \F_{\ell}}(D^{\times})$ then its Fargues-Scholze parameter $\ov \phi^{\ov\pi}$ is of the form $\ov\phi \oplus \ov\phi(1) \oplus \dotsc \oplus \ov\phi(r-1)$ for some integer $r$ and some irreducible parameter $\ov\phi$.
\end{lemma}
\begin{proof}
    Since $\ell$ is a banal prime, there exists an $\ell$-integral supercuspidal representation $\pi$ in $\Rep_{\ov \Q_{\ell}}(D^{\times})$ whose reduction modulo $\ell$ is isomorphic to $\ov \pi$ \cite[proposition 4.15]{DHKM24}. By the compatibility between Fargues-Scholze parameters and $L$-parameters for inner forms of $\GL_n$, we deduce that the Fargues-Scholze parameter $\phi^{\pi}$ of $\pi$ is of the form $\phi \oplus \phi(1) \oplus \dotsc \oplus \phi(r-1)$ for some $r$ and some irreducible parameter $\phi$. By the first point of theorem \ref{FSproperties}, we see that $\phi^{\pi}$ is $\ell$-integral and hence each $\phi(k)$ is $\ell$-integral. Moreover the Fargues-Scholze parameter of $\ov \pi$ is given by $\ov \phi \oplus \ov \phi(1) \oplus \dotsc \oplus \ov \phi(r-1)$ where $\ov\phi$ is the reduction modulo $\ell$ of $\phi$.

    Let $\pi^{\GL}$ be the irreducible supercuspidal representation of $\GL_{\tfrac{n}{r}}(F)$ corresponding to $\phi$ and let $\ov\pi^{\GL}$ be its reduction modulo $\ell$. Then the Fargues-Scholze parameter of $\ov\pi^{\GL}$ is given by $\ov\phi$. By \cite[Proposition 4.15]{DHKM24}, $\ov\pi^{\GL}$ is irreducible supercuspidal. Hence $\ov\phi$ is irreducible by a result of Dat and Vignéras that we mentioned before lemma \ref{itm : existence, irreducible}. This is what we want to show.
\end{proof}

\section{The supercuspidal case}

Let $\G$ denote the reductive group $\Res_{F/\Q_p}(\GL_{n, F})$ where $F$ is an unramified extension of $\Q_p$ of degree $d$ and let $\Lambda \in \{ \ov \F_{\ell}, \ov \Q_{\ell} \}$ for $\ell \neq 2$. The goal of this section is to prove the supercuspidal part of the categorical local Langlands conjecture for $\G(\Q_p)$.

\subsection{The categorical equivalence up to some shifts} \phantomsection \label{itm : up to some shift}

Let $\phi$ be an $L$-parameter of $\G$ (over $\Q_p$), by Shapiro's lemma it corresponds to a representation $\Tilde{\phi} : W_{F} \longrightarrow \widehat{\GL_n}(\Lambda)$. In this section, we will consider the $L$-parameters $\phi$ such that the representation $\Tilde{\phi}$ is irreducible. 

We know that the centralizer group $S_{\phi}$ is isomorphic to $\Gm$ and by proposition \ref{itm : simple connected components}, the connected component $ [C_{\phi}] $ of the stack of $L$-parameters $[Z^1(W_{\Q_p}, \widehat{\G})_{\Lambda}/\widehat{\G} ]$ containing $\phi$ is isomorphic to
\[
    [C_{\phi}] \simeq [(\mathbb{G}_{m} \times \mu_{ \phi, \Lambda}) / \mathbb{G}_m],
\]
where $\mu_{\phi, \Lambda} = \Spec \overline{\Q}_{\ell}$ if $\Lambda = \overline{\Q}_{\ell}$ and if $\Lambda = \ov \F_{\ell}$ then $\mu_{\Lambda} = \displaystyle  \Spec \overline{\F}_{\ell}[\Syl]$ where $\Syl$ is the $\ell$-Sylow group of $ \F^{\times}_{p^{f}} $ and $f$ is the number of irreducible factors of $\Tilde{\phi}_{| \I_F}$. Since the group of algebraic characters $\Irr(\Gm)$ is isomorphic to $\Z$, the (derived) category $\Coh([C_{\phi}])$ of coherent complexes supported on $[C_{\phi}]$ is equivalent to 
\[
\bigoplus_{\chi \in \Irr(\Gm) \simeq \Z} \mathrm{D}(\mathrm{Mod}_{\chi}(A_{\phi}))^{\omega},
\]
where $A_{\phi} := \mathcal{O}(\mathbb{G}_{m} \times \mu_{ \phi, \Lambda})$, the algebra of global sections of $\mathbb{G}_{m} \times \mu_{ \phi, \Lambda}$; $\mathrm{D}(\mathrm{Mod}_{\chi}(A_{\phi}))$ is the derived category of the category of $A_{\phi}$-modules together with the action of $\bb G_m$ acting by $\chi$ and $\mathrm{D}(\mathrm{Mod}_{\chi}(A_{\phi}))^{\omega}$ its full sub-category generated by compact objects. Thus $A_{\phi} \simeq \overline{\Q}_{\ell}[X, X^{-1}]$ if $\Lambda = \overline{\Q}_{\ell}$ and $A_{\phi} \simeq \overline{\F}_{\ell}[X, X^{-1}] \otimes_{\overline{\F}_{\ell}} \overline{\F}_{\ell}[\Syl] \simeq \overline{\F}_{\ell}[X, X^{-1}, Y]/(Y^{\ell^u}) $ if $\Lambda = \overline{\F}_{\ell}$ and $| \Syl | = \ell^u$.

The connected components of $\Bun_{\G}$ are indexed by $\Z \simeq \Irr(\Gm)$. Let $\chi \in \Irr(\Gm)$ and let $\Bun^{\chi}_{\G}$ be the connected component indexed by $\chi$ and let $b_{\chi} \in B(\G)$ be the basic element corresponding to the basic stratum in $\Bun^{\chi}_{\G}$. We also denote by $\G_{b_{\chi}}(\Q_p)$ the inner form of $\G(\Q_p)$ corresponding to the basic element $b_{\chi}$. In particular, $\G(\Q_p) \simeq \G_{b_{\Id}}(\Q_p) $. We have a direct decomposition:
\[
\Dlis(\Bun_{\G}, \Lambda)^{\omega} = \bigoplus_{\chi \in \Irr(\Gm)} \Dlis(\Bun^{\chi}_{\G}, \Lambda)^{\omega},
\]
and it induces a direct decomposition
\[
\Dlis^{[C_{\phi}]}(\Bun_{\G}, \Lambda)^{\omega} = \bigoplus_{\chi \in \Irr(\Gm)} \Dlis^{[C_{\phi}]}(\Bun^{\chi}_{\G}, \Lambda)^{\omega},
\]
where $\Dlis^{[C_{\phi}]}(\Bun_{\G}, \Lambda)^{\omega}$ is the full sub-category of $\Dlis(\Bun_{\G}, \Lambda)^{\omega}$ consisting of objects on which the excursion operator corresponding to the function that is $1$ on $[C_{\phi}]$ and $0$ elsewhere acts via the identity.

Let $\mathcal{F} \in \Dlis^{[C_{\phi}]}(\Bun_{\G}, \Lambda)^{\omega}$ be a Schur irreducible sheaf. Since $\phi$ is a supercuspidal $L$-parameter, one sees that $\mathcal{F}$ is supported on the semi-stable locus of $\Bun_{\G}$. In other words, $\mathcal{F}$ is supported on the strata $\Bun^{b_{\chi}}_{\G}$ for $(\chi \in \Irr(\Gm))$. Thus we have
\[
\Dlis^{[C_{\phi}]}(\Bun_{\G}, \Lambda)^{\omega} = \bigoplus_{\chi \in \Irr(\Gm)} \Dlis^{[C_{\phi}]}(\Bun^{b_{\chi}}_{\G}, \Lambda)^{\omega}.
\]

Let $\Rep_{\Lambda}(\G_{b_{\chi}}(\Q_p)) $ denote the category of $\Lambda$-smooth representations of $\G_{b_{\chi}}(\Q_p)$ and let $\pi_{\chi}$ be the supercuspidal representation corresponding to $\phi$ by lemma \ref{itm : existence, irreducible}. Denote by $\mathfrak{s}_{\phi, \chi}$ the inertial class of the supercuspidal pair $(\G_{b_{\chi}}, \pi_{\chi})$ and by $\Rep_{\Lambda}(\mathfrak{s}_{\phi, \chi})$ the block corresponding to $\mathfrak{s}_{\phi, \chi}$ in the category $ \Rep_{\Lambda}(\G_{b_{\chi}}(\Q_p)) $ by theorem \ref{itm : decomposition into blocks}. The irreducible representations in $\Rep_{\Lambda}(\mathfrak{s}_{\phi, \chi})$ are the unramified twists of $\pi_{\chi}$. By lemma \ref{itm : existence, irreducible} and the compatibility of Fargues-Scholze construction of parameters with twisting by a character, we deduce that $\Dlis^{[C_{\phi}]}(\Bun^{b_{\chi}}_{\G}, \Lambda)$ is equivalent to the derived category of the block $\Rep_{\Lambda}(\mathfrak{s}_{\phi, \chi})$. Recall that Dat constructed a compact pro-generator $\mathcal{W}_{\phi, \Id}$ of the block $\Rep_{\Lambda}(\mathfrak{s}_{\phi, \Id})$ and compute $\End_{\Rep_{\Lambda}(\mathfrak{s}_{\phi, \Id})}(\mathcal{W}_{\phi, \Id})^{\opp}$ and then show that this block is equivalence to the category of $A_{\phi}$-modules \cite[Proposition B.1.2]{Dat12}. Moreover, we will normalize this equivalence so that the supercuspidal representation $\pi_{\Id}$ corresponds to the module whose underlying vector space is isomorphic to $\Lambda$ and where $Y$ acts by multiplying by $0$ and $X$ acts by multiplying by $1$. 

We have an inclusion $\IndPerf([C_{\phi}]) \supset \Coh([C_{\phi}])$, thus we can define an action of $\Coh([C_{\phi}])$ on $\Dlis^{[C_{\phi}]}(\Bun_{\G}, \Lambda)$  as the colimit-preserving extension of the spectral action. By letting the category $\Coh([C_{\phi}])$ acts on $\mathcal{W}_{\phi, \Id}$ by the spectral action, we have a functor
\begin{align*}
   \mathrm{Act} :  \Coh([C_{\phi}]) &\longrightarrow \Dlis^{[C_{\phi}]}(\Bun_{\G}, \Lambda) \\
   \mathcal{C} &\longmapsto \mathcal{C} \star (i_{b_{\Id !}} \mathcal{W}_{\phi, \Id}),
\end{align*}
where we use $\star$ to denote the spectral action. As before, for each $\chi \in \Irr(S_{\phi}) \simeq \Irr(\Gm)$, we have a line bundle $C_{\chi}$ on $[C_{\phi}]$ and moreover $C_{\chi}$ acting on $\Dlis^{[C_{\phi}]}(\Bun_{\G}, \Lambda)^{\omega}$ by the spectral action gives rise to an auto-equivalence of $\Dlis^{[C_{\phi}]}(\Bun_{\G}, \Lambda)^{\omega}$.

Recall that the spectral Bernstein center of $\G$ is $ \mc Z^{\rm spec}(\G, \Lambda) := \mathcal{O}(Z^1(W_{\Q_p}, \widehat{\G})_{\Lambda})^{\widehat{\G}} $, the ring of global functions on the stack/the coarse moduli space of $L$-parameters. The geometric Bernstein center of $\G$ is $ \mc Z^{\rm geom}( \G, \Lambda ) := \pi_0 (\End (\id_{\Dlis(\Bun_{\G}, \Lambda)})) $ and there is a natural identification between $ \mc Z^{\rm spec}(\G, \Lambda) $ and the algebra of excursion operators (\cite[theorem VIII.5.1]{FS}). By \cite[theo. VIII.4.1]{FS}, there is a map $ \mc Z^{\rm spec}(\G, \Lambda) \longrightarrow \mc Z^{\rm geom}(\G, \Lambda)$ which induces a map
\[
\Psi_{\G} : \mc Z^{\rm spec}(\G, \Lambda) \longrightarrow \mc Z (\G, \Lambda)
\]
by the inclusion $\Dlis(\G(\Q_p), \Lambda) \hookrightarrow \Dlis(\Bun_{\G}, \Lambda)$ and where $\mc Z (\G, \Lambda)$ is the Bernstein center of $\G$. For the group $\G$ that we consider, the map $\Psi_{\G}$ is known to be an isomorphism\cite[Corollary 7.7]{HM} \cite[page 338]{FS}.

Let $f \in \mathcal{O}(Z^1(W_{\Q_p}, \widehat{\G})_{\Lambda})^{\widehat{\G}}$ be an element and denote by $\mathcal{O}$ the structural sheaf of $[Z^1(W_{\Q_p}, \widehat{\G})_{\Lambda} / \widehat{\G} ]$. In particular the multiplication by $ f $ gives rise to a map of $\mathcal{O}(Z^1(W_{\Q_p}, \widehat{\G}))^{\widehat{\G}}$-modules $ \mc O \xrightarrow{ \ \cdot f \ } \mc O  $. Note that $\mathcal{O}$ is the vector bundle corresponding to the Hecke operator of the trivial representation of $\G$. Hence via the spectral action, the map $ \mc O \xrightarrow{ \ \cdot f \ } \mc O  $ induce an element in $\End(\mc O \star A) \simeq \End(A)$ for every $A \in \Dlis(\Bun_{\G}, \Lambda)$ and therefore it induces an element $\widetilde{f}$ in $ \End(\id_{\Dlis(\Bun_{\G}, \Lambda)^{\omega}}) $. Now the excursion operator corresponding to $f$ also gives rise to an element $\overline{f}$ in $ \End(\id_{\Dlis(\Bun_{\G}, \Lambda)^{\omega}}) $, in other words, $\Psi_{\G}(f) = \overline{f}$. By the compatibility of the spectral action and the excursion operators \cite[Theorem 5.2.1]{Zou}, we have $ \widetilde{f} = \overline{f} $. More concretely, let $A$ be a Schur-irreducible element in $ \Dlis(\Bun_{\G}, \Lambda)^{\omega} $ whose Fargues-Scholze $L$-parameter is given by $\phi$. Then $ \mc O \xrightarrow{ \ \cdot f \ } \mc O  $ acting on $A$ gives us an endomorphism
\[
\{ \mc O \star A = A \longrightarrow \mc O \star A = A \} \in \End(A) = \Lambda,
\]
which will be precisely the scalar $\alpha$ given by evaluating $A$ on the excursion datum corresponding to $f$ (see also \cite[page 24] {Ham}). By \cite[Proposition VIII.3.8]{FS}, $\phi$ corresponds to a surjective map 
\[
\rm Ev_{\phi} : \mathcal{O}(Z^1(W_{\Q_p}, \widehat{\G})_{\Lambda})^{\widehat{\G}} \longrightarrow \Lambda
\]
and the scalar $\alpha$ above obtained by evaluating $\phi$ on the excursion datum corresponding to $f$ is exactly $ \mathrm{Ev}_{\phi} (f) $.

Recall that $\mathcal{O}([C_{\phi}])$ is a direct factor of $\mc Z^{\rm spec}(\G, \Lambda)$ and we now want to describe explicitly the restriction of $\Psi_{\G}$ to $\mathcal{O}([C_{\phi}])$. By the compatibility of Fargues-Scholze construction of $L$-parameters with the usual construction of $L$-parameters for $\G$ and the fact that $\Psi_{\G}$ is an isomorphism, we deduce that
\[
\Psi_{\G | \mathcal{O}([C_{\phi}])} : A_{\phi} \simeq \mathcal{O}([C_{\phi}]) \longrightarrow \End(\id_{\Rep_{\Lambda}(\mathfrak{s}_{\phi, \Id})}) \simeq A_{\phi}
\]
is again an isomorphism. Note that $ A_{\phi} \simeq \overline{\F}_{\ell}[X, X^{-1}, Y]/(Y^{\ell^u}) $ if $\Lambda = \overline{\F}_{\ell}$ and $| \Syl | = \ell^u$. 
\begin{lemma} \phantomsection \label{itm : global sections acting on sheaves}
     There is some change of variables such that the map
     \[
     \Psi_{\G | \mathcal{O}([C_{\phi}])} : A_{\phi} \simeq \mathcal{O}([C_{\phi}]) \longrightarrow \End(\id_{\Rep_{\Lambda}(\mathfrak{s}_{\phi, \Id})}) \simeq A_{\phi}
     \]
     becomes the identity map and moreover the maximal ideal corresponding to $\phi$ is $(X-1, Y) \subset \mathcal{O}([C_{\phi}]) \simeq A_{\phi} $ and the supercuspidal representation $\pi_0$ corresponds to the $A_{\phi}$-module whose underlying vector space is isomorphic to $\Lambda$ and where $Y$ acts by multiplying by $0$ and $X$ acts by multiplying by $1$.
\end{lemma}
\begin{proof}
   The nilradical of $A_{\phi}$ is generated by $Y$ and $\Psi_{\G | \mathcal{O}([C_{\phi}])}$ is an isomorphism, then $\Psi_{\G | \mathcal{O}([C_{\phi}])} (Y) = Y P$ for some invertible $P \in A_{\phi}$. Thus by some change of variables, we can suppose that $\Psi_{\G | \mathcal{O}([C_{\phi}])}$ sends $Y$ to $Y$. By combining again the fact that $\Psi_{\G | \mathcal{O}([C_{\phi}])}$ is an isomorphism and the arguments in \cite[lemma 9.1]{GLn}, we can show that $\Psi_{\G | \mathcal{O}([C_{\phi}])}$ sends $X$ to $X + Q $ for some $Q \in (Y) \subset A_{\phi}$. Thus by some change of variables, we can suppose that $ \Psi_{\G | \mathcal{O}([C_{\phi}])}(X) = X $.  
\end{proof}

\begin{theorem} \phantomsection \label{itm : supercuspidal - spectral action}
\begin{enumerate}
    \item  For each $\chi \in \Irr(\Gm) \simeq \Z$, there is an equivalence of categories between the block $\Rep_{\Lambda}(\mathfrak{s}_{\phi, \chi})$ in $\Rep_{\Lambda}(\G_{b_{\chi}}(\Q_p))$ and the category of $A_{\phi}$-modules. Moreover, under this equivalence, the supercuspidal representation $\pi_{\chi} \otimes (\xi_t \circ \det )$ corresponds to the $A_{\phi}$-module whose underlying vector space is isomorphic to $\Lambda$ and where $Y$ acts by multiplying by $0$ and $X$ acts by multiplying by $t^f$ where $\xi_t$ is the unramified character of $W_{F}$ corresponding to $t \in \Lambda^{\times}$ and $f$ is the torsion number of $\pi_{\chi}$. The number $f$ is also the length of the $\I_F$-representation $\Tilde{\phi}_{\I_F}$.
    \item  Let $\bL$ be a coherent complex of $A_{\phi}$-modules. For each $\chi \in \Irr(S_{\phi}) \simeq \Z$, let $\bL(\chi)$ be the coherent complex $\bL$ together with the action of $\bb G_m$ acting by $\chi$ so that it gives rise to a coherent complex on $[C_{\phi}]$. We denote by $\pi_{\chi}(\bL)$ the complex of smooth $\G_{b_{\chi}}(\Q_p)$-representations in the derived category of the Bernstein block $\Rep_{\Lambda}(\mathfrak{s}_{\phi, \chi})$ corresponding to the $A_{\phi}$-perfect complex $\bL$. Then
\[
\Act(\bL(\chi)) := \bL(\chi) \star (i_{b_{\Id} !} \mathcal{W}_{\phi, \Id}) \simeq i_{b_{\chi} !} (\pi_{\chi}(\bL)).
\]
\end{enumerate}
\end{theorem}
In the rest of this subsection, we will prove the first assertion of the theorem and prove that 
\[
\Act(\bL(\chi)) := \bL(\chi) \star (i_{b_{\Id} !} \mathcal{W}_{\phi, \Id}) \simeq i_{b_{\chi} !} (\pi_{\chi}(\bL)).
\]
up to some shifts. In the next subsection, we will use global method to determine the right shifts in the above formula.
\begin{proof}(up to some shifts) \\

\textbf{Step 1:} The case $\chi$ is trivial

Recall that 
\[
\Dlis^{[C_{\phi}]}(\Bun_{\G}, \Lambda)^{\omega} = \bigoplus_{\chi \in \Irr(\Gm)} \Dlis^{[C_{\phi}]}(\Bun^{b_{\chi}}_{\G}, \Lambda)^{\omega},
\]    
and $\Dlis^{[C_{\phi}]}(\Bun^{b_{\Id}}_{\G}, \Lambda)$ is the derived category of the block $\Rep_{\Lambda}(\mathfrak{s}_{\phi, \Id})$ in $\Rep_{\Lambda}(\G_{b_{\Id}}(\Q_p))$. Recall that $C_{\Id}$ acts on $\Dlis^{[C_{\phi}]}(\Bun^{b_{\Id}}_{\G}, \Lambda)$ by the identity functor. By using the same arguments in \cite[Theorem 9.5]{GLn} together with lemma \ref{itm : global sections acting on sheaves}, we can show that the conclusion of the theorem is true for $\chi$ the trivial representation of $\Gm$ and $\bL$ a perfect complex of $A_{\phi}$-module. Then by colimit-preserving extension, the conclusion is also true for coherent complexes (with the correct shifts). \\

\textbf{Step 2:} The computation of $C_{\chi} \star (i_{b_{\Id} !} \pi_{\Id})$ up to some twist.

Let $\chi \in \Irr(\Gm)$ be a non-trivial character corresponding to $d \in \Z^{*}$. By direct computation, we see that $C_{\chi}$ is a direct factor of the restriction of $V_{\mu}$ to $[C_{\phi}]$ where $\mu$ is the cocharacter $ \big( (d, 0^{(n-1)}), \underbrace{(0^{(n)}), \dotsc, (0^{(n)})}_{d-1 \ \mathrm{times}}  \big) $ of $\G(\Q_p)$ and $V_{\mu}$ is the vector bundle associated to the highest weight representation of $\widehat{\G}(\Lambda)$ of highest weight $\mu$. By the same arguments as in page $347$ of \cite{FS} we see that if $\mathcal{F}$ is a sheaf in $\Dlis^{[C_{\phi}]}(\Bun^{b_{\chi'}}_{\G}, \Lambda)^{\omega}$ then $ C_{\chi} \star \mathcal{F} $ is in $\Dlis^{[C_{\phi}]}(\Bun^{b_{\chi' \otimes \chi}}_{\G}, \Lambda)^{\omega}$. Since $C_{\chi}$ acting on $\Dlis^{[C_{\phi}]}(\Bun_{\G}, \Lambda)^{\omega}$ is an auto-equivalence of this category. We deduce that $C_{\chi}$ induces an equivalence between $\Dlis^{[C_{\phi}]}(\Bun^{b_{\Id}}_{\G}, \Lambda)^{\omega}$ and $\Dlis^{[C_{\phi}]}(\Bun^{b_{\chi}}_{\G}, \Lambda)^{\omega}$.

We will first show that $C_{\chi} \star (i_{b_{\Id} !} \pi_{\Id}) \simeq i_{b_{\chi} !} \pi_{\chi}$ up to some shift. Since the Hecke operators preserve compact objects and universally locally acyclic objects, we deduce that $C_{\chi} \star (i_{b_{\Id} !} \pi_{\Id})$ is the $i_{b_{\chi} !}$ pushforward of a bounded complex whose cohomology groups are admissible representations of finite length in the derived category of $\Rep_{\Lambda}(\mathfrak{s}_{\phi, \chi})$. Moreover, the action of the excursion algebra commute with the Hecke operators, thus the irreducible constituents of the cohomology groups of $C_{\chi} \star (i_{b_{\Id} !} \pi_{\Id})$ all have Fargues-Scholze parameter given by $\phi$.

Suppose that the cohomology groups of $C_{\chi} \star (i_{b_{\Id} !} \pi_{\Id})$ is strictly concentrated in the interval $[a, b]$ where $a < b$ are integers and both $ \mathrm{H}^a( C_{\chi} \star (i_{b_{\Id} !} \pi_{\Id}) )$, $ \mathrm{H}^b(C_{\chi} \star (i_{b_{\Id} !} \pi_{\Id}))$ are non trivial.  Let $\mathcal{F}$ be the $i_{b_{\Id} !}$ pushforward of a complex whose cohomology groups are strictly concentrated in the interval $[a', b']$ and the Fargues-Scholze parameter of the irreducible constituents of its cohomology groups are given by $\phi$. Then by using the long exact sequences coming from exact triangles, we deduce that $C_{\chi} \star  \mathcal{F}$ is a complex concentrated in the interval $[a+a', b+b']$. In particular, $C_{\chi} \star (\mathcal{F})$ is not concentrated in one degree. Thus there exist no $\mathcal{F}$ compact and ULA such that $ C_{\chi} \star \mathcal{F} \simeq i_{b_{\chi} !}(\pi_{\chi}) $, a contradiction with the fact that $C_{\chi}$ induces an equivalence between $\Dlis^{[C_{\phi}]}(\Bun^{b_{\Id}}_{\G}, \Lambda)^{\omega}$ and $\Dlis^{[C_{\phi}]}(\Bun^{b_{\chi}}_{\G}, \Lambda)^{\omega}$. We conclude that $C_{\chi} \star (i_{b_{\Id} !} \pi_{\Id})$ is concentrated in one degree. Now by using the fact that
\[
\mathrm{H}^0(i_{b_{\Id} !}\pi_{\Id}, i_{b_{\Id !}\pi_{\Id}}) = \mathrm{H}^0(C_{\chi} \star (i_{b_{\Id} !}\pi_{\Id}), C_{\chi} \star (i_{b_{\Id} !}\pi_{\Id})) = \Lambda,
\]
we deduce that 
\[
C_{\chi} \star (i_{b_{\Id} !}\pi_{\Id}) \simeq i_{b_{\chi} !}\pi_{\chi},
\]
up to some shift. And hence, up to some shift, we also have
\[
C_{\chi^{-1}} \star (i_{b_{\chi} !}\pi_{\chi}) \simeq i_{b_{\Id} !}\pi_{\Id}.
\]
\textbf{}\\
\phantom{}\hspace{2ex} \textbf{Step 3:} The computation of $C_{\chi} \star (i_{b_{\Id} !} \mathcal{W}_{\phi, \Id})$, up to some shift.

Recall that we can also constructed a compact pro-generator of the block $\Rep_{\Lambda}(\mathfrak{s}_{\phi, \chi})$ in $\Rep_{\Lambda}(\G_{b_{\chi}}(\Q_p))$ (by types theory) but in general we do not know how to compute its endomorphisms algebra. Let $\mathcal{W}_{\phi, \chi}$ be such a compact pro-generator, we will compute $C_{\chi^{-1}} \star (i_{b_{\chi} !} \mathcal{W}_{\phi, \chi}) $. Since the Hecke operators preserve compact objects, $\mathcal{F} := C_{\chi^{-1}} \star (i_{b_{\chi} !} \mathcal{W}_{\phi, \chi}) $ is a compact object in $\Dlis^{[C_{\phi}]}(\Bun^{b_{\Id}}_{\G}, \Lambda)^{\omega}$. Thus we can represent $\mathcal{F}$ as the $i_{b_{\Id}!}$-pushforward of a bounded complex of smooth representations in the block $\Rep_{\Lambda}(\mathfrak{s}_{\phi, \Id})$ whose cohomology groups are finitely generated smooth representations. Moreover, the block $\Rep_{\Lambda}(\mathfrak{s}_{\phi, \Id})$ is equivalent to the category of $A_{\phi}$-modules and hence the cohomology groups of $\mathcal{F}$ corresponds to finitely generated $A_{\phi}$-modules.

Let $M$ be a finitely generated $A_{\phi}$-modules then $M$ is also finitely generated as $R := \Lambda[X, X^{-1}]$-module where the structure of $R$-module is induced from that of $A_{\phi}$-module. Thus as $R$-modules, we have a decomposition $M \simeq R^a \bigoplus M_{\Tors} $ where $a \in \N$ and $M_{\Tors}$ is the torsion-submodule of $M$. Since the action of $X$ and $Y$ commutes, we see that $Y$ also preserves $M_{\Tors}$. Hence it is also a sub-$A_{\phi}$-module of $M$ and is of finite length. Moreover since $Y^{\ell^u}$ acts by $0$, if $a > 0$ then we can find a free sub $R$-module $N$ of $M$ that is annihilated by $Y$. Thus $N$ is also a sub $A_{\phi}$-module of $M$ and more precisely, $N$ can be written as successive extension of $A_{\phi}$-module $A^{\redu}_{\phi}$ whose underlying $R$-module is isomorphic to $R$ and where $Y$ acts by $0$. We deduce that $M$ also has a quotient $A_{\phi}$-module isomorphic to $A^{\redu}_{\phi}$.

Since $\mathcal{W}_{\phi, \chi}$ is a compact pro-generator then for any supercuspidal representation $\pi_{\chi} \otimes (\xi_t \circ \det )$ the group $\Hom \big(\mathcal{W}_{\phi, \chi}, \pi_{\chi} \otimes (\xi_t \circ \det )[k] \big)$ is non-trivial if and only if $k = 0$. We showed that $C_{\chi^{-1}} \star (i_{b_{\chi} !} \pi_{\chi} \otimes (\xi_t \circ \det ) ) \simeq i_{b_{\Id} !} \pi_{\Id} \otimes (\xi_t \circ \det ) $ up to some shift and moreover $\pi_{\Id} \otimes (\xi_t \circ \det )$ corresponds to the $A_{\phi}$-module whose underlying vector space is isomorphic to $\Lambda$ and where $Y$ acts by multiplying by $0$ and $X$ acts by multiplying by $t^f$. In particular, there exists a unique $k \in \Z$ such that $\Hom \big(\mathcal{F}, i_{b_{\Id} !} \pi_{\Id} \otimes (\xi_t \circ \det )[k] \big)$ is non-trivial.

Suppose that all the cohomology groups of $\mathcal{F}$ correspond to torsion $R$-modules (of finite length). Then there exists $t$ such that $\Hom_{A_{\phi}} (M_{\Tors}, \lambda_t[k]) = 0$ for all $k \in \Z$ and any cohomology group $M_{\Tors}$ of the complex corresponding to $\mathcal{F}$ and where $\lambda_t$ is the $A_{\phi}$-module corresponding to $\pi_{\chi} \otimes (\xi_t \circ \det )$. Then $\Hom \big(\mathcal{F}, i_{b_{\chi} !} \pi_{\Id} \otimes (\xi_t \circ \det )[k] \big) = 0 $ for all $k \in \Z$, a contradiction. Hence, a cohomology group of $\mathcal{F}$ has a sub and a quotient $A_{\phi}$-module isomorphic to $A^{\redu}_{\phi}$. Recall that $\Hom_{A_{\phi}}(A^{\redu}_{\phi}, \lambda_t) \neq 0 $ for $t \neq 0$.

Suppose that the cohomology groups of $\mathcal{F}$ are concentrated in the interval $[a, b]$ where $a < b$ are integers and both $\mathrm{H}^a(\mathcal{F})$ and $\mathrm{H}^b(\mathcal{F})$ are non-trivial. If $\mathrm{H}^b(\mathcal{F})$ has a quotient $A_{\phi}$-module isomorphic to $A^{\redu}_{\phi}$ then we deduce that $\Hom( \mathcal{F}, \mathcal{F}[a-b]) \neq 0 $ since $\mathrm{H}^a(\mathcal{F})$ has a sub-$A_{\phi}$-module isomorphic to $A^{\redu}_{\phi}$ or to $\lambda_t$ for some $t$. It is a contradiction, then $\mathrm{H}^b(\mathcal{F})$ corresponds to a torsion $A_{\phi}$-module. We can suppose that $\Hom^b(\mathcal{F}, \lambda_t) \neq 0$. Note that we have a projective resolution of the $A_{\phi}$-module $A_{\phi}/(X - t^f)$
\[
A_{\phi} \xrightarrow{ \ \cdot (X - t^f) \ } A_{\phi} \xrightarrow{\quad \quad} A_{\phi}/(X - t^f).
\]

We choose $c$ the largest integer $i$ such that $\mathrm{H}^i(\mathcal{F})$ has a quotient isomorphic to the $A_{\phi}$-module $A^{\redu}_{\phi}$. By using the above projective resolution and using the long exact sequence coming from exact triangles we can show that $ \Hom( \mathcal{F}, A_{\phi}/(X - t^f)[k]) \neq 0 $ for $ k = b, c $. It contradicts the fact that $\mathcal{W}_{\phi, \chi}$ is projective. Thus $\mathcal{F}$ concentrates in one degree $m$, for some $m \in \Z$. By using the fact that the group $\Hom \big(\mathcal{W}_{\phi, \chi}, \pi_{\chi} \otimes (\xi_t \circ \det )[k] \big)$ is non-trivial if and only if $k = 0$, we deduce that for all $t \in \Lambda^{*}$
\[
C_{\chi^{-1}} \star (i_{b_{\chi} !}\pi_{\chi} \otimes (\xi_t \circ \det ) ) \simeq i_{b_{\Id} !}\pi_{\Id}\otimes (\xi_t \circ \det )[-m].
\]

Now let $M_{\mathcal{F}}$ denote the $A_{\phi}$-module corresponding to $\mathcal{F}$. We can choose $t \in \Lambda^{*}$ such that $\Hom_{A_{\phi}}((M_{\mathcal{F}})_{\Tors}, \lambda_t[k]) = 0$ for all $k \in \Z$. By using again the fact that $\Hom \big(\mathcal{W}_{\phi, \chi}, \pi_{\chi} \otimes (\xi_t \circ \det )[k] \big)$ is non-trivial if and only if $k = 0$, we deduce that $ M_{\mathcal{F}}/(M_{\mathcal{F}})_{\Tors} \simeq (A_{\phi})^{a} $ for some $a \in \N^*$ and then it implies that $(M_{\mathcal{F}})_{\Tors}$ is a direct factor of $M_{\mathcal{F}}$. Finally it implies that $(M_{\mathcal{F}})_{\Tors} \simeq 0 $ and $M_{\mathcal{F}} \simeq A^a_{\phi}$. By applying $C_{\chi}$ we see that
\[
i_{b_{\chi}!} \mathcal{W}_{\phi, \chi} \simeq \big(C_{\chi} \star (i_{b_{\Id} !} \mathcal{W}_{\phi, \Id})\big)^a[-m]
\]

Since $ \mathcal{W}_{\phi, \chi}$ is a projective generator of the block $\Rep_{\Lambda}(\mathfrak{s}_{\phi, \chi})$, we deduce that 
$i^*_{b_{\chi}} \big(C_{\chi} \star (i_{b_{\Id} !} \mathcal{W}_{\phi, \Id})\big)$ is also a projective generator of $\Rep_{\Lambda}(\mathfrak{s}_{\phi, \chi})$. Thus we can suppose $a = 1$ and 
\begin{equation} \phantomsection \label{itm : progenerator to progenerator}
 i_{b_{\chi}!} \mathcal{W}_{\phi, \chi} \simeq C_{\chi} \star (i_{b_{\Id} !} \mathcal{W}_{\phi, \Id})[-m].   
\end{equation}

\textbf{}\\
\phantom{}\hspace{2ex} \textbf{Step 4:} Conclusion.

We deduce from equation $(\ref{itm : progenerator to progenerator})$ that $\End_{\Rep_{\Lambda}(\mathfrak{s}_{\phi, \chi})}(\mathcal{W}_{\phi, \chi})^{\opp} \simeq A_{\phi}$. Thus the block $\Rep_{\Lambda}(\mathfrak{s}_{\phi, \chi})$ is equivalent to the category of $A_{\phi}$-modules and we deduce the first conclusion of the theorem.

We can embed the block $\Rep_{\Lambda}(\mathfrak{s}_{\phi, \chi})$ into $\Dlis^{[C_{\phi}]}(\Bun_{\G}, \Lambda)$ by the functor $i_{b_{\chi} !}$. Then it gives us a map $\Psi_{\chi} : \mathcal{O}([C_{\phi}]) \longrightarrow \End(\id_{\Rep_{\Lambda}(\mathfrak{s}_{\phi, \chi})})$. By using again the isomorphism $(\ref{itm : progenerator to progenerator})$, we can compute the action of $\Psi_{\chi}(x)$ on $\mathcal{W}_{\phi, \chi}$ for any $x \in \mathcal{O}([C_{\phi}])$ and deduce that $\Psi_{\chi}$ is a ring isomorphism. Thus by using the same arguments as in the first step, we deduce the second conclusion of the theorem, up to a shift $[-m]$ as above.

\end{proof}
\subsection{The supercuspidal part of the cohomology of local Shimura varieties} \textbf{} \label{itm : Shimura}

In this paragraph, we compute the supercuspidal part of the mod $\ell$ cohomology of Shimura varieties and Rapoport-Zink spaces, generalizing \cite{Shin15}. We proceed by using Mantovan's product formula and some results on vanishing of the cohomology of Shimura varieties. 

\subsubsection{Shimura varieties and Mantovan's formula} \textbf{}

We consider some Shimura varieties of type $A$ as in \cite{Kot92}. Assume that $E = \mathcal{K} E_{0} $ is a CM field where $\mathcal{K}$ is an imaginary quadratic field, $p$ is split in $\mathcal{K}$ and $(E_{0})_p = F$ is the unramified extension of degree $d$ of $\Q_p$. Then we have an unramified integral PEL datum of the form $(\rm B, \mc O_{\rm B} , *, V, \Lambda_0, \langle, \rangle, h )$, where
\begin{enumerate}
    \item[$\bullet$] $\rm B$ is a semi-simple algebra over $E$,
    \item[$\bullet$] $*$ is an involution of the second kind,
    \item[$\bullet$] $\rm B, \mc O_{\rm B}$ is a $\Z_{(p)}$ maximal order in $\rm B$ that is preserved by $*$ such that $ \rm B, \mc O_{\rm B} \otimes_{\Z} \Z_p $ is a maximal order in $\rm B_{\Q_p}$, 
    \item[$\bullet$] $V$ is a simple $ \rm B$-module,
    \item[$\bullet$] $\langle, \rangle : V \times V \longrightarrow \Q $ is a $*$-Hermitian pairing with respect to the $\rm B$-action,
    \item[$\bullet$] $\Lambda_0$ is a $\Z_p$-lattice in $V_{\Q_p}$ that is preserved by $\mc O_{\rm B}$ and self-dual for $\langle, \rangle$.
    \item[$\bullet$] $h : \C \longrightarrow \End_{\R}(V_{\R})$ is an $\R$-algebra homomorphism satisfying the equality $\langle h(z), w \rangle = \langle v , h(z^c)w \rangle, \forall v,w \in V_{\R} $ and $z \in \C$ and such that the bilinear pairing $(v, w) \longmapsto \langle v , h(\sqrt{-1})w \rangle $ is symmetric and positive definite.
\end{enumerate}

We can define a similitude unitary group $ \rm \textbf{G} $ over $\Q$ associated to this PEL datum and denote by $\rm \textbf{G}'$ the unitary group that is the kernel of the similitude factor, thus $\rm \textbf{G}'(\R)$ $\simeq \displaystyle \prod_{\sigma \in \Gal(F / \Q_p) } \U (q_{\sigma}, n - q_{\sigma})$ for some $q_{\sigma} \in \N$. We assume further that the $p$-adic group $\rm \textbf{G}'_{\Q_p}$ is isomorphic to $\G \simeq \Res_{F/ \Q_p}\GL_{n, F}$ and hence $\mathrm{\textbf{G}}_{\Q_p} \simeq \G \times \bb G_{m, \Q_p} $.

From now on to the end of this subsection, we suppose that $\rm B$ is a division algebra with center $F$. Let $\mu : \Gm \longrightarrow \rm \textbf{G} $ be the group homomorphism over $\C$ corresponding to $h$. Associated to the above PEL datum is a system of \textit{projective} Shimura varieties $\{\Sh_K\}$ defined over the field $E$ and where $K$ runs over the set of sufficiently small open compact subgroups of $\rm \textbf{G}(\A^{\infty})$. Moreover, the PEL datum also gives rise to a system of integral models $\mc S_p = \{ \mc S_{K^p} \}$ defined over $ \mc O_{E, (p)} $ such that the generic fiber of each $\mc S_{K^p}$ is naturally identified with $\Sh_{K^pK_p^{\rm hs}} \otimes_{E} E_{p}$ where $K^p$ runs over the set of sufficiently small open compact subgroups of $\rm \textbf{G}(\A^{\infty, p})$ and $K_p^{\rm hs} \subset \GL_n(\Q_p)$ is a hyperspecial subgroup.

On the other hand, for each $b \in B(\rm \textbf{G}'_{\Q_p}, -\mu)$, we can define an Igusa variety $\Ig_b$ which is closely related to the special fiber $\overline{\mc S}_p$ of $\mc S_p$. The Igusa variety is a projective system of smooth varieties $ \Ig_b = \{ \Ig_{b, m, K^p} \}$ where $m \in \N$ and $K^p$ are sufficiently small open compact subgroups of $\rm \textbf{G}(\A^{\infty, p})$ as before. 
Define
\[
R\Gamma_c(\Sh_{K^p}, \Lambda) := \mathop{\mathrm{colim}}_{\overrightarrow{K_p \rightarrow \{1 \}}} R\Gamma_c(\Sh_{K_pK^p, \ov \Q}, \Lambda),
\]
\[
R\Gamma_c(\Ig_{b, K^p}, \Lambda) := \mathop{\mathrm{colim}}_{\overrightarrow{m \rightarrow \{ \infty \} } } R\Gamma_c(\Ig_{b, m, K^p}, \Lambda).
\]

These complexes are endowed with an action of $\rm \textbf{G}(\A^{\infty}) \times \Gal(\overline{F}/F)$ and of $\mathrm{\textbf{G}}(\A^{\infty, p}) \times \big( \G_b(\Q_p) \times \Q_p^{\times} \big)$ respectively.

For each $K^p$, the special fiber $\overline{\mc S}_{K^p}$ has a stratification by the Kottwitz set $B(\GL_n, -\mu)$. Thus we can compute the cohomology of Shimura varieties by the cohomology of its strata. We recall a formula of Mantovan that expresses a cohomological relation between $R\Gamma_c(\Sh_{K^p}, \Lambda)$, $R\Gamma_c(\Ig_{b, K^p}, \Lambda)$'s and the Hecke operator $\T_{-\mu}$. 

\begin{proposition} (\cite[Theorem 22]{Man2}, \cite[Theorems. 6.26, 6.32]{LS2018}, \cite[Theorem 1.12]{HL}), \cite[Theorem 8.5.7]{DHKZ} \phantomsection \label{itm : Mantovan formula}
    We set $h := \dim \Sh_{K^p} = \langle 2\rho_{\G}, \mu \rangle$ and for $b \in B(\G, -\mu)$, we set $d_b := \dim \Ig_{b, K^p} = \langle 2\rho_{\G} , \nu_b \rangle$. Then the complex $R\Gamma_c(\Sh_{K^p}, \Lambda)$ has a filtration as a complex of $ \mathrm{\textbf{G}}(\Q_p) \times W_{F} $ representations with graded pieces isomorphic to $ i_1^* \T_{-\mu}(i_{b!} \delta^{-1}_b \otimes R\Gamma_c(\Ig_{b, K^p}, \Lambda) )[-h](-\tfrac{h}{2})$ where $b$ runs in the Kottwitz set $B(\G, -\mu)$. More precisely, the graded pieces are isomorphic to 
    \[
    R\Gamma_c(\G, b, \mu) \otimes^{\bL}_{\mathcal{H}(\G_b)}R\Gamma_c(\Ig_{b, K^p}, \Lambda)[2d_b-h](-\tfrac{h}{2}), 
    \]
    as complexes of $ \textbf{G}(\Q_p) \times W_{F} $-modules.
\end{proposition}

\subsubsection{Supercuspidal part of the mod-$\ell$ cohomology of some Shimura varieties} \textbf{}

In this paragraph we analyse the complex $R\Gamma_c(\Sh_{K^p}, \Lambda)$ of $ \textbf{G}(\Q_p) \times W_{F} $ representations. By \cite[Pproposition 8.21]{Zha23}, the cohomology groups of this complex are admissible. Thus we can consider the projection of $R\Gamma_c(\Sh_{K^p}, \Lambda)$ to the derived category of the block $\Rep_{\Lambda}(\mathfrak{s}_{\phi, \Id})$. We use superscript $(-)^{[\phi]}$ to denote this projection.

\begin{proposition}
    The complex $R\Gamma_c(\Sh_{K^p}, \Lambda)^{[\phi]}$ is concentrated in degree $h$.
\end{proposition}
\begin{proof}
We know from proposition \ref{itm : Mantovan formula} that the complex $R\Gamma_c(\Sh_{K^p}, \Lambda)$ has a filtration as a complex of $ \textbf{G}(\Q_p) \times W_{F} $ representations with graded pieces isomorphic to $ (R\Gamma_c(\G, b, \mu) \otimes^{\bL}_{\mathcal{H}(\G_b)}R\Gamma_c(\Ig_{b, K^p}, \Lambda))[2d_b-h](-\tfrac{h}{2})$ and where $b$ runs in the Kottwitz set $B(\G, -\mu)$. Since $\phi$ is a supercuspidal Fargues-Scholze parameter and the functors $\T_{\mu}$, $i^*_1$ preserve Fargues-Scholze parameters, we deduce that the contributions of non-basic strata to $R\Gamma_c(\Sh_{K^p}, \Lambda)^{[\phi]}$ are trivial. Thus
\[
R\Gamma_c(\Sh_{K^p}, \Lambda)^{[\phi]} \simeq (R\Gamma_c(\G, b_0, \mu) \otimes^{\bL}_{\mathcal{H}(\G_{b_0})}R\Gamma_c(\Ig_{b_0, K^p}, \Lambda))^{[\phi]}[-h](-\tfrac{h}{2}),
\]
where $b_0$ is the unique stratum in $B(\G, \mu)$.

Our Shimura variety is of type $A$ and compact then the Igusa variety $\Ig_{b_0, K^p}$ corresponding to the basic stratum is affine of dimension $0$. Thus $R\Gamma_c(\Ig_{b_0, K^p}, \Lambda))$ is concentrated in degree $0$. Moreover, since theorem \ref{itm : supercuspidal - spectral action} is true up to some shift, we deduce that $R\Gamma_c(\Sh_{K^p}, \Lambda)^{[\phi]} \simeq (R\Gamma_c(\G, b_0, \mu) \otimes^{\bL}_{\mathcal{H}(\G_{b_0})}R\Gamma_c(\Ig_{b_0, K^p}, \Lambda))^{[\phi]}[-h](-\tfrac{h}{2})$ is concentrated in one degree. Thus we deduce that $R\Gamma_c(\Sh_{K^p}, \overline{\Z}_{\ell})^{[\phi]}$ is concentrated in one degree and torsion free where $R\Gamma_c(\Sh_{K^p}, \overline{\Z}_{\ell}):= R\Gamma_c(\Sh_{K^p}, \Z_{\ell}) \otimes_{\Z_{\ell}} \ov \Z_{\ell}$ and $R\Gamma_c(\Sh_{K^p}, \overline{\Z}_{\ell})^{[\phi]}$ is the projection of $R\Gamma_c(\Sh_{K^p}, \overline{\Z}_{\ell})$ to the derived category of the block corresponding to $\pi_{\Id}$ in $\Rep_{\overline{\Z}_{\ell}}(\G(\Q_p))$ (see \cite[Appendix B1]{Dat12}). Therefore $R\Gamma_c(\Sh_{K^p}, \overline{\F}_{\ell})^{[\phi]}$ is the reduction modulo $\ell$ of $R\Gamma_c(\Sh_{K^p}, \overline{\Z}_{\ell})^{[\phi]}$ and we conclude that $R\Gamma_c(\Sh_{K^p}, \overline{\Q}_{\ell})^{[\phi]}$ and $R\Gamma_c(\Sh_{K^p}, \overline{\F}_{\ell})^{[\phi]}$ are concentrated in degree $h = \dim \Sh_{K^p} $. Moreover by choosing the level $K_p$ sufficiently small, the cohomology group $H^d_c(\Sh_{K^p}, \overline{\Q}_{\ell})^{[\phi]}$ is non trivial since there exists automorphic representation $ \Pi $ of $\mathrm{\textbf{G}}(\A)$ such that $\Pi^{K^p}$ is not trivial and $\Pi_{p}$ is a supercuspidal representation in $\Rep_{\overline{\Q}_{\ell}}(\G(\Q_p))$ lifting some supercuspidal representation in the block $\Rep_{\Lambda}(\mathfrak{s}_{\phi, \Id})$.
\end{proof}

This result allows us to deduce the correct shift in the computation of the Hecke operator $\T_{\mu}$ acting on $i_{b_{\Id}!}(\Rep_{\Lambda}(\mathfrak{s}_{\phi, \Id}))$ and thus it implies theorem \ref{itm : supercuspidal - spectral action}.
\section{Spectral actions}
\subsection{Combinatoric description of the spectral actions} \textbf{} \phantomsection \label{itm : combinatoric description of Hecke operators}

We fix a maximal split torus $\mathrm{A}$ inside the maximal torus $\T$ and a Borel subgroup $\rB$ of $\G:= \Res_{F/\Q_p}(\GL_{n, F})$ where $F$ is the unramified extension of degree $d$ of $\Q_p$. Let $\overline{C}$ denote the closed Weyl chamber in $X_*(\mathrm{A})_{\R}$ associated to $\rB$ and let $\overline{C}_{\Q}$ denote its intersection with $X_*(\mathrm{A})_{\Q}$. For any standard parabolic subgroup $\rP$ with Levi decomposition $\rP = \M\rN$ such that $ \T \subset \M $ (i.e., $\M$ is a standard Levi subgroup), we denote by $\mathrm{A}_{\M}$ the split maximal torus in the center of $\M$. We put
\[
X_*(\rP)^+ := \{ \mu \in X_*(\mathrm{A}_{\M}) \ | \ \langle \mu, \alpha \rangle > 0 \ \text{for any root of $\T$ in $\rN$}  \}.
\]

Then we have a decomposition 
\[
\overline{C} = \coprod_{\rP}  X_*(\rP)^+,
\]
where the index is the set of standard parabolic subgroups of $\G$. We define the subset $B(\G)_{\rP}$ of $B(\G)$ to be the pre-image of $X_*(\rP)^+$ under the Newton map. This gives a decomposition
\begin{equation} \phantomsection \label{decomposition of Kottwitz' set}
 B(\G) = \coprod_{\rP} B(\G)_{\rP}   
\end{equation}
where the index is the set of standard parabolic subgroups of $\G$. For a general standard parabolic $\rP = \M\rN$, $B(\G)_{\rP}$ has the following description. By noting that the image of the Newton map $\nu_{\M}$ for $\M$ lies in $X_*(\M)$, we define $B(\M)^+_{\text{bas}}$ by 
\[
B(\M)^+_{\text{bas}} := \{ b \in B(\M)_{\text{bas}} \ | \ \nu_{\M}(b) \in X_*(\rP)^+ \}. 
\]
Then the canonical map $B(\M) \longrightarrow B(\G)$ induces a bijection $ B(\M)^+_{\text{bas}} \simeq B(\G)_{\rP} $ (see \cite[\S 5.1]{KottwitzIsocrystals2}).

Let $b$ be an element in $B(\G)$ then there exists a unique standard parabolic subgroup $\rP$ such that $ \nu_b \in X_*(\rP)^+$, moreover its Levi factor $\M$ is isomorphic to the quasi-split inner form of $\G_{b}$. We recall that if $ \rP = \M \rN $ where $\M$ is the Levi subgroup and $\rN$ is the unipotent radical then the modulus character is defined by $\delta_{\rP, \C}(mn) := |\det(\ad(m);\Lie \rN)| $ where $|\cdot|$ denotes the normalized absolute value of the field $\ov \Q_{p}$. Thus, by using an isomorphism $\iota : \C \simeq \overline{\Q}_{\ell}$ we get a character $\delta_{\rP} : \rP(\Q_p) \longrightarrow \overline{\Q}_{\ell}^{\times} $. Moreover, since $p$ is invertible in $\overline{\Z}_{\ell}$, we can also regard $\delta_{\rP}$ as a character valued in $\overline{\F}_{\ell}^{\times}$.

Denote by $\delta_{\rP} : \rP(\Q_p) \longrightarrow \Lambda^{\times} $ the modulus character of $\rP$. This character factors through the group $\M^{ab}(\Q_p)$. Since $\G_b$ is an inner form of $\M$, we see that $\G_b^{ab} \simeq \M^{ab}$. Thus $\delta_{\rP}$ can be seen as a character of $\G_b(\Q_p)$ and we denote this character by $\delta_b$. Our characters $\delta_{b}$ is the same as the character $\delta_b$ defined by Hamann and Imai in \cite{HI}.

Recall that by Shapiro lemma, there is a bijection between the set (of conjugacy classes) of $L$-parameter of $\G$ and that of $\GL_{n, F}$. In the rest of this section, we will consider an $L$-parameter $\phi$ of $\G$ such that the corresponding $L$-parameter $\tilde{\phi}$ of $\GL_{n, F}$ satisfies the following condition:
\begin{enumerate} \label{itm : condition A1}
    \item[(A1)] We have a decomposition $ \tilde{\phi} = \tilde{\phi}_1 \oplus \dotsc \oplus \tilde{\phi}_r $ where the $\tilde{\phi}_i$'s are pairwise disjoint irreducible representations of dimension $n_i$ and if $i \neq j$ then there does not exist unramified character $\chi$ such that $ \tilde{\phi}_i \simeq \tilde{\phi}_j \otimes \chi $. 
\end{enumerate}

Let $\phi_i$ be the $L$-parameter of $\G$ corresponding to $\tilde{\phi}_i$ by Shapiro's lemma (for $1 \leq i \leq r$), thus $\phi = \phi_1 \oplus \dotsc \oplus \phi_r$. Then by the computation in \cite[page 83]{KMSW}, we can see that $ \displaystyle S_{\phi} := \Cent (\phi) = \prod_{i = 1}^r \Gm $ and the set $\Irr(S_{\phi})$ of irreducible representations of $S_{\phi}$ is isomorphic to the abelian group $ \displaystyle \prod_{i = 1}^r \Z$. For each character $\chi = (d_1, \dotsc, d_r) \in \displaystyle \prod_{i=1}^r \Z $, we define an element $b_{\chi} \in B(\G) = B(\GL_{n, F})$, an irreducible representation $\pi_{\chi}$ of $ \G_{\chi}(\Q_p) := \G_{b_{\chi}}(\Q_p)$ and a sheaf $\mathcal{F}_{\chi}$ as follows:
\begin{enumerate}
    \item[$\bullet$] For each $1 \leq i \leq r$, we denote $ \lambda_i := d_i / n_i $ and we consider the sequence $ (-\tfrac{\lambda_1}{d}^{(n_1)}, \dotsc, - \tfrac{\lambda_r}{d}^{(n_r)}) \in \Q^n $. After rearranging the terms, we denote by $\nu$ the image of this sequence in the set $ \{ (x_1, \dotsc, x_n) \ | \ x_i \in \Q, \ \ x_i \geq x_{i + 1}  \}$. Then $b_{\chi}$ is the unique element in $B(\G)$ such that $ \nu_{b_{\chi}} $ is equal to $\nu$. We can also define $b_{\chi}$ as the unique element in $B(\G) = B(\GL_{n, F})$ such that $\E_{F, b_{\chi}}$, the corresponding rank $n$ vector bundle over $X_F$, is isomorphic to $\OO_F(\lambda_1)^{m_1} \oplus \dotsc \oplus \OO_F(\lambda_r)^{m_r} $ where $\OO_F(\lambda_i)$ is the stable vector bundle of slope $ \lambda_i = d_i / n_i $ and $m_i = (d_i, n_i)$.  
    \item[$\bullet$] Consider the group $\G_{b_{\chi}}$, it is an inner form of a standard Levi subgroups of $\G$. Let $\G^*_{b_{\chi}}$ be the split inner form of $\G_{b_{\chi}}$. For each $i$, denote $\rH_i := \GL_{m_i}(D_{-\lambda_i})$ where $D_{-\lambda_i}$ is the division algebra whose invariant is $-\lambda_i$, thus $\rH_i^* = \Res_{F/\Q_p}(\GL_{n_i, F})$. Therefore we have a $1$-cocycle $\phi_i : W_{\Q_p} \longrightarrow \widehat{\rH^*_i}(\Lambda)$ and the direct sum $ \phi_1 \oplus \dotsc \oplus \phi_r $ gives us a $1$-cocycle $W_{\Q_p} \longrightarrow \displaystyle \prod_{i=1}^r \widehat{\rH^*_i}(\Lambda)$ whose post-composition with the natural embeddings $ \displaystyle \prod_{i=1}^r \widehat{\rH^*_i}(\Lambda) \hookrightarrow \widehat{\G^*_{\chi}}(\Lambda) $ defines an $L$-parameter $\phi_{\chi}$ of $\G_{b_{\chi}}$. Moreover, the post-composition of $\phi_{\chi}$ with $ \widehat{\G^*_{b_{\chi}}}(\Lambda) \hookrightarrow \widehat{\GL}_n(\Lambda) $ is the ($\widehat{\GL}_n$-conjugacy class of) $\phi$. If $\Lambda = \ov \Q_{\ell} $ then we define $\pi_{\chi}$ to be the representation of $\G_{b_{\chi}}(\Q_p)$ whose $L$-parameter is given by $\phi_{\chi}$ via the local Langlands correspondence for general linear groups and its inner forms. If $\Lambda = \ov \F_{\ell}$ then the construction of $\pi_{\chi}$ is more complicated. For each $1 \leq i \leq r$, we denote by $\pi_i$ the unique irreducible supercuspidal representation in $\Rep_{\ov \F_{\ell}}(\rH_i(\Q_p))$ whose Fargues-Scholze parameter is given by $\phi_i$ (lemma \ref{itm : existence, irreducible}). Thus $ \pi := \displaystyle \bigotimes_{i=1}^r \pi_i $ is a supercuspidal representation of $ \rH (\Q_p) := \displaystyle \prod_{i=1}^r \rH_i(\Q_p)$ and the latter is a Levi subgroup of $\G_{b_{\chi}}(\Q_p)$. Thus we define $\pi_{\chi}$ to be the normalised parabolic induction $\Ind_{\rH (\Q_p)}^{\G_{b_{\chi}}(\Q_p)}(\pi)$. Remark that $\pi_{\chi}$ is irreducible by \cite[Theorem 7.23]{MS14}.
    \item[$\bullet$] We consider the embedding $ i_{b_{\chi}} : \Bun^{b_{\chi}}_{\G} \longrightarrow \Bun_{\G} $ and define $\mathcal{F}_{\chi} \displaystyle := i_{b_{\chi} !}(\delta^{-1/2}_{b_{\chi}} \otimes \pi_{\chi}) [- d_{\chi}] $ where $d_{\chi} = \langle 2\rho, \nu_{b_{\chi}} \rangle$. For the sake of notation, we also denote the renormalized push-forward functor $ i_{b_{\chi}!}(\delta^{-1/2}_{b_{\chi}} \otimes - )[- d_{\chi}] $ by $i^{\rm ren}_{b_{\chi} !}( - )$.

    We recall also that for $b \in B(\G)$, there exists a unique standard parabolic subgroup $\rP$ such that $\nu_b \in X_*(\rP)^+$. Since we have $\nu_{b} = (- \nu_{\E_b})_{\text{dom}}$, then the standard parabolic subgroup corresponding to the Harder-Narasimhan reduction of $\E_{b_{\chi}}$ (\cite[Theorem 1.7]{CFS}) is conjugated to the opposite of $\rP$ above.

\end{enumerate}
\subsection{Vanishing results}
Let $\G$ denote the group $\Res_{F/\Q_p}(\GL_{n, F})$. For each Fargues-Scholze parameter $\phi$ with coefficient in $\Lambda$, we denote by $R(\phi)$ the set consisting of pairs $ (b, \pi) $ where $b$ is in $B(\G)$ and $\pi$ is an irreducible representation in $\Rep_{\Lambda}(\G_b(\Q_p))$ whose Fargues-Scholze parameter is $\phi$.

\begin{lemma} \phantomsection \label{itm : relevance}
    Let $\phi = \phi_1 \oplus \dotsc \oplus \phi_r$ be a Fargues-Scholze parameter satisfying condition (A1) and let $\G_m := \GL_m(D)$ where $D$ is a division algebra over $F$ of degree $g^2$. Then there exists at most one irreducible representation in $\Rep_{\Lambda}(\G_m)$ such that its Fargues-Scholze parameter is given by $\phi$. Moreover, there exists one such irreducible representation if and only if $g$ divides $\dim(\phi_i)$ for all $1 \leq i \leq r$.
\end{lemma}
\begin{proof}
when $\Lambda = \ov \Q_{\ell}$, the lemma is clear by the local Langlands correspondence for $\GL_n$ and its inner forms. Thus we suppose $\Lambda = \ov \F_{\ell}$.

Suppose that $\pi$ is an irreducible representations in $\Rep_{\ov \F_{\ell}}(\G_m)$ such that its Fargues-Scholze parameter is given by $\phi$. Thus by Minguez-Sécherre's classification of irreducible representations in $\Rep_{\ov \F_{\ell}}(\G_m)$ and by the compatibility of Fargues-Scholze's construction with parabolic induction, we deduce there is an integer $k$ such that $\phi = \phi'_1 \oplus \dotsc \oplus \phi'_k $ and for $1 \leq i \leq k$, $\phi'_i$ is the Fargues-Scholze parameter of some representation of the form $\rZ(\Delta_i)$ where $\Delta_i = [a_i, b_i]_{\rho_i}$ is a segment.  

Since $\phi$ satisfies condition (A1), we deduce that $a_i = b_i$ and therefore $ \rZ(\Delta_i) = \nu^{a_i}\rho_i$ is some supercuspidal representation of $\G_{n_i} := \GL_{n_i}(D)$. Thus, by \cite[Theorem 3.26]{MS}, we can lift $\rZ(\Delta_i)$ to a supercuspidal $\ell$-integral representation $\widetilde{\pi}_i$ of $\G_{n_i}$. Once again, by using the condition (A1) and the compatibility of Fargues-Scholze's construction with reduction modulo $\ell$, we see that the Fargues-Scholze parameter $\widetilde{\phi'}_i$ of $\pi_i$ is irreducible and $\ell$-integral. Moreover, the reduction modulo $\ell$ of $\widetilde{\phi'}_i$ is given by $\phi'_i$. However we know that the reduction modulo $\ell$ of an $\ell$-integral irreducible parameter is of the form $ V + \nu V + \dotsc + \nu^{b-1}V $ for some $\overline{\F}_{\ell}$-representation of $W_{F}$. Thus by using condition (A1), we deduce that $\phi'_i$ is irreducible and in particular $k = r$. We can then suppose that $ \phi'_i = \phi_i $ for all $1 \leq i \leq r$. Hence for all $1 \leq i \leq r$, $g$ divides $\dim(\phi_i)$.

Let $\pi_i$ be the unique irreducible supercuspidal representation in $\Rep_{\ov \F_{\ell}}(\G_{n_i})$ whose Fargues-Scholze parameter is given by $\phi_i$. Then by construction, $\pi$ is a sub-quotient of the normalised parabolic induction $ \Ind^{\G_m}_{\rP} (\pi_1 \times \dotsc \times \pi_r)$. However this representation is irreducible by \cite[Theorem 7.23]{MS14}, we see that $\pi$ is given by the above parabolic induction. This shows the uniqueness claim of the lemma. It is also clear that when $g$ divides $\dim(\phi_i)$ for all $1 \leq i \leq r $, by using parabolic induction one can construct an irreducible representation whose Fargues-Scholze parameter is given by $\phi$.   
\end{proof}
\begin{proposition} \phantomsection \label{itm : vanishing}
    Suppose that $\phi = \phi_1 \oplus \dotsc \oplus \phi_r$ is a Fargues-Scholze parameter with coefficient in $\Lambda$ satisfying condition (A1). Then the map
    \begin{align*}
        \Irr(S_{\phi}) &\longrightarrow R(\phi) \\
        \chi &\longmapsto (b_{\chi}, \pi_{\chi})
    \end{align*}
    is a bijection.

    Moreover, if $\mathcal{F}$ is a non-zero Schur-irreducible sheaf on $\Bun_{\G}$ supported on a stratum corresponding to $b \in B(\G)$ and the corresponding parameter is given by $\phi$ then there exists $\chi \in \Irr(S_{\phi})$ such that $b = b_{\chi}$.
\end{proposition} 
\begin{proof}
By using lemma (\ref{itm : relevance}), the same arguments as in \cite[\S 4.1.2]{GLn} still work.    
\end{proof}
\begin{remark}
    This is a more general version of some result on vanishing of the cohomology of local Shimura varieties as in \cite[Proposition 1.5]{Ko1}
\end{remark}
\begin{remark}
    See \cite{BMO} for a general statement when $\G$ is a connected reductive group and $\Lambda = \ov \Q_{\ell}$
\end{remark}
\subsection{Spectral action}
 In this subsection, we consider an $L$-parameter $\phi$ satisfying the condition (A1). In particular we see that $ \displaystyle S_{\phi} := \Cent (\phi) = \prod_{i = 1}^r \Gm $. Denote by $ [C_{\phi}] $ the connected component of $[Z^1(W_{\Q_p}, \widehat{\G})_{\Lambda}/\widehat{\G}]$ containing the closed point defined by $\phi$. By proposition \ref{itm : simple connected components}, we know that $[C_{\phi}] \simeq [(\mathbb{G}^r_{m} \times \mu_{\Lambda}) / \mathbb{G}^r_m]$ where the quotient is taken with respect to the trivial action of $\bb G_m^r$. Let $C \in \Perf^{\mathrm{qc}}([Z^1(W_{\Q_p}, \widehat{\G})_{\Lambda}/\widehat{\G}])$ be a sheaf such that its support does not intersect with $[C_{\phi}]$. Then by \cite[lemma 3.8]{Ham}, we know that $ C \star \mathcal{F} = 0 $ where $\mathcal{F}$ is any Schur irreducible sheaf on $\Bun_{\G}$  whose $L$-parameter belongs to $[C_{\phi}]$.

Denote by $\Perf([C_{\phi}])$ the category of perfect complex supported on $[C_{\phi}]$. We have a monoidal embedding of categories
\[
\Rep_{\Lambda}(S_{\phi}) \longrightarrow \Perf([C_{\phi}]) \longrightarrow \Perf([Z^1(W_{\Q_p}, \widehat{\G})/\widehat{\G}])
\]
where the image of an irreducible character $\chi$ is the vector bundle $C_{\chi}$ on $[C_{\phi}]$ corresponding to the structural sheaf on $\bb G_m^r \times \mu_{\Lambda}$ together with the $\bb G_m^r$-action defined by $\chi$.

Remark that $S_{\phi}$ is commutative then the set $\Irr( S_{\phi} )$ of its algebraic characters forms a group under the tensor product operator. In this specific case that group is isomorphic to $ \displaystyle \prod_{i=1}^r \Z $. Let $ \displaystyle \chi = (d_1, \dotsc, d_r) \in \prod_{i=1}^r \Z$ be the character of $S_{\phi}$ such that $ \chi (t_1, \dotsc, t_r) \displaystyle = \prod_{i = 1}^r t_i^{d_i} $. We can define a triple $( b_{\chi}, \pi_{\chi}, \mathcal{F}_{\chi} )$ as in the previous section. Recall that $ \mathcal{F}_{\chi} := i_{b_{\chi} !} ( \delta^{-1/2}_{b_{\chi}} \otimes \pi_{\chi}) [-d_{\chi}]$ where $d_{\chi} = \langle 2\rho, \nu_{b_{\chi}} \rangle$.

\begin{theorem} \phantomsection \label{itm : main theorem I}
Let $\phi$ be an $L$-parameter satisfying the condition (A1) in the beginning of the subsection. Let $\chi = (d_1, \dotsc, d_r)$ be an element in $\displaystyle \prod_{i=1}^r \Z $, then we have
\[
C_{\chi} \star \mathcal{F}_{\Id} = \mathcal{F}_{\chi}.
\]
where $\Id$ is the identity of $ \displaystyle \prod_{i=1}^r \Z$.
\end{theorem}
\begin{proof}
Let $\G'$ denote the group $ \GL_{n, F}$, thus $\G = \Res_{F/\Q_p}(\G')$. We will need the Fargues-Fontaine curves constructed over $\Q_p$ and over $F$. We denote by $\Bun_{\G', F}$ the stack of $\G'$-bundles on the Fargues-Fontaine curve over $F$ and similarly for $\Bun_{\G, \Q_p}$. By \cite[Proposition IX.6.3]{FS} we have $\Bun_{\G', F} \simeq \Bun_{\G, \Q_p}$ and in particular, we have $ \Dlis(\Bun_{\G', F}, \Lambda) \simeq \Dlis(\Bun_{\G, \Q_p}, \Lambda)$. Moreover $[C_{\phi}]$  is isomorphic to the connected component $[C_{\Tilde{\phi}}]$ containing $\Tilde{\phi}$ of $ [Z^1(W_{F}, \widehat{\GL_{n, F}})_{\Lambda}/\widehat{\GL_{n, F}} ] $ where $\Tilde{\phi}$ is the $L$-parameter of $^{L}\widehat{\G'}(\Lambda)$ corresponding to $\phi$ by Shapiro's lemma.  

Thanks to the identification $[C_{\phi}] \simeq [C_{\Tilde{\phi}}]$, we can compute the action of $C_{\chi}$ on $\Dlis(\Bun_{\G', F}, \Lambda)$ by the same strategy as in the proof of \cite[Theorem 4.5]{GLn}. Now we need to transfer this computation to $\Dlis(\Bun_{\G, \Q_p}, \Lambda)$. Let $V$ be the standard representation of $\widehat{\G'}(\Lambda)$. Remark that $ \widehat{\G}(\Lambda) = \displaystyle \prod_{F \hookrightarrow \overline{\Q}_p} \widehat{\G'}(\Lambda) $. Thus by fixing an embedding $F \hookrightarrow \ov \Q_p$, we can inflate $V$ to a representation of $\G(\Lambda) \rtimes W_F$ and then induce it to a representation $V'$ of $\G(\Lambda) \rtimes W_{\Q_p}$. By \cite[Proposition IX.6.3]{FS}, the Hecke action $\T_{V'}$ on $\Dlis(\Bun_{\G, \Q_p}, \Lambda)$ acts as $\Ind_{W_{F}}^{W_{\Q_p}} \circ T_{V}$ under the identification $ \Dlis(\Bun_{\G', F}, \Lambda) \simeq \Dlis(\Bun_{\G, \Q_p}, \Lambda)$ and restriction the Galois action to $W_F$. Thus we know the Hecke action $\T_{V'}$ on $\Dlis(\Bun_{\G, \Q_p}, \Lambda)$ (via the action of $\Ind_{W_{F}}^{W_{\Q_p}} \circ T_{V}$ on $\Dlis(\Bun_{\G', F}, \Lambda)$) and deduce the action of $C_{\chi}$ on $\Dlis(\Bun_{\G, \Q_p}, \Lambda)$.
\end{proof}
\subsection{Hecke eigensheaves and cohomology of local Shimura varieties for $\GL_n$}
\subsubsection{Hecke eigensheaves} \textbf{}

The main goals of this paragraph is to use theorem \ref{itm : main theorem I} to give an explicit description of the Hecke eigensheaves associated to some $L$-parameters $\phi$ of $\G := \Res_{F/\Q_p} \GL_{n, F}$ such that the corresponding $L$-parameter $\tilde{\phi} = \tilde{\phi}_1 \oplus \dotsc \oplus \tilde{\phi}_r $ of $\GL_{n, F}$ satisfies the condition ($\mathrm{A}1$). 
The connected component $[C_{\phi}]$ of the stack of $L$-parameters of $\G$ is isomorphic to $[C_{\phi}] \simeq [(\mathbb{G}^r_{m} \times \mu_{\Lambda}) / \mathbb{G}^r_m]$ as before. 
The category $\Rep_{\Lambda} (S_{\phi}) $ is semi-simple, we see that the regular representation of $S_{\phi}$ has the following description
\[
V_{\text{reg}} = \bigoplus_{\chi \in \Irr(S_{\phi})} \chi.
\] 
We deduce an explicit description of a Hecke eigensheaf associated to $\phi$.
\begin{theorem} \label{itm : Hecke eigensheaf}
The sheaf
\[
\displaystyle \mathcal{G}_{\phi} := \bigoplus_{\chi \in \Irr(S_{\phi})} \mathcal{F}_{\chi}
\]
is a non trivial Hecke eigensheaf corresponding to the $L$-parameter $\phi$.
\end{theorem}
\subsubsection{Cohomology of local Shimura varieties and Harris-Viehmann conjecture} \textbf{}

The main goals of this subsection is to compute some parts of the cohomology of the moduli spaces $\Sht(\G, b, b', \mu)$ with 
$\Lambda$-coefficient where $b, b' \in B(\G)$ and $\mu$ is a dominant cocharacter of $\G$. Then we deduce new cases of the Harris-Viehmann conjecture for $\G$.

We denote by $r_{-\mu}$ the highest weight representation of $\widehat{\G}(\Lambda)$ of highest weight $(\mu^{-1})_{\dom}$. Let $\lambda \in \Irr(S_{\phi})$ be a character and let $(b, \pi_{b})$ be the corresponding pair by proposition \ref{itm : vanishing}. Then by lemma \ref{shimhecke}, we have
\[
 R\Gamma_{c}(\G, b',b,\mu)[\delta^{1/2}_b \otimes \pi_{b}][d_{b}] \simeq  i_{b'}^{*}\T_{-\mu}\mathcal{F_{\lambda}}, 
\]
where $d_b = \langle 2\rho, \nu_b \rangle $ and $b' \in B(\G)$.

We have an identification of $S_{\phi} \times W_{F}$-representations:
\[
r_{-\mu} \circ \phi_{| \widehat{\G}(\Lambda)} = \bigoplus_{\chi \in \Irr(S_{\phi})} \chi \boxtimes \sigma_{\chi}, 
\]
where $\sigma_{\chi}$ is the $W_{F}$-representations $\Hom_{S_{\phi}}(\chi, r_{-\mu} \circ \phi_{| \widehat{\G}(\Lambda)})$, then we have 
\[
\T_{-\mu} \mathcal{F}_{\lambda} = \bigoplus_{\chi \in \Irr(S_{\phi})} (C_{\chi} \star \mathcal{F}_{\lambda}) \boxtimes \sigma_{\chi}. 
\]

Hence we can use the explicit description of the action of $C_{\chi}$ to compute $R\Gamma_{c}(\GL_n,b',b,\mu)) [\delta^{1/2}_b \otimes \pi_{b}]$.

\begin{theorem}(some cases of the generalized Harris-Viehmann's conjecture) \label{itm : Harris-Viehmann conjecture}
    Let $b$ be an element in $B(\GL_n)$ and let $\M$ be the standard Levi subgroup of $\GL_n$ that is the split inner form of $\G_b$ and $\rP$ is the standard parabolic subgroup of $\GL_n$ whose Levi factor is $\M$. Let $\mu$ be an arbitrary cocharacter of $\G$ and $ \pi_b $ be an irreducible representation of $\G_b(\Q_p)$ such that its corresponding $L$-parameter $\phi^b$ post-composed with the natural embedding $ ^{L}\G_b(\Lambda) \longrightarrow \ ^{L}\G(\Lambda) $ is $\phi$. Thus we have 
    \begin{equation} \phantomsection \label{itm : cohomology of RZ spaces}
      R\Gamma_{c}(\G,b,\mu) [\delta^{1/2}_b \otimes \pi_b] = \pi_1 \boxtimes \Hom_{S_{\phi}} (\chi^{-1}_b ,r_{-\mu} \circ \phi_{| \widehat{\G}(\Lambda)})[-h]  
    \end{equation}
    where $\pi_1$ is the irreducibe representation of $\G(\Q_p)$ whose Fargues-Scholze $L$-parameter is $\phi$ and $\chi_b$ is the unique character of $S_{\phi}$ corresponding to the couple $(b, \pi_b)$ and where $h = \langle 2\rho, \nu_b \rangle$. In particular, the generalized Harris-Viehmann's conjecture is true in this case. More precisely, suppose that as $ ^{L}\M(\Lambda)$-representation, we have 
    \begin{equation} 
   (r_{-\mu})^{\rm ss} \simeq \bigoplus_{\mu_{\M} \in X_{*}(\M) } \big((r_{-\mu_{\M}})^{\rm ss}\big)^{m_{-\mu_{\M}}} 
\end{equation}
for some multiplicity $ m_{-\mu_{\M}} = \dim \Hom ( r_{-\mu_{\M}} ,r_{-\mu | \M} ) $ and where $(-)^{\rm ss}$ denotes the semisimplification. Then we have

\begin{align*}
R\Gamma_{c}(\G, b,\mu) [ \delta^{1/2}_b \otimes \pi_{b}] [h - d_{\M}]  &\simeq \bigoplus_{\mu_{\M} \in X_{*}(\M) } \big( \Ind_{\rP}^{\GL_n} R\Gamma_{c}(\M, b_{\M},\mu_{\M})  [\delta^{1/2}_{b_{\M}} \otimes \pi_{b}]\big)^{m_{-\mu_{\M}}}.
\end{align*}
where $d_{\M} = \langle 2\rho_{\M}, \nu_{b_{\M}} \rangle$ and $b_{\M}$ is the reduction of $b$ to $\M$.
\end{theorem}
\begin{proof}
    Note that the category $\Rep_{\Lambda}(\widehat{\G}(\Lambda))$ is semisimple if $\Lambda = \overline{\Q}_{\ell}$ but it is no longer semisimple if $\Lambda = \overline{\F}_{\ell}$. However we can still apply the arguments in \cite[Theorem 8.1]{GLn} to this situation. The point is that since $S_{\phi}$ is a torus, the category $\Rep_{\Lambda}(S_{\phi})$ is semi-simple and the Hecke operator $\T_{-\mu}$ only depends on $r_{-\mu} \circ \phi_{| \widehat{\G}(\Lambda)}$ as $S_{\phi} \times W_F$-representation.
\end{proof}
\section{On categorical local Langlands conjecture for $\GL_n$}

\subsection{Spectral action acting on generators of $\Rep_{\Lambda}(\G_b(\Q_p))$} \textbf{}

In this section, we will consider an $L$-parameter $\phi$ of $\G$ such that the corresponding $L$-parameter $\tilde{\phi} = \tilde{\phi}_1 \oplus \dotsc \oplus \tilde{\phi}_r $ of $\GL_{n, F}$ satisfies the condition ($\mathrm{A}1$). For $1 \leq i \leq r$, we denote by $f_i$ the number of irreducible factors of $\tilde{\phi}_{i| \I_F}$. The connected component $[C_{\phi}]$ of the stack of $L$-parameters of $\G$ is isomorphic to $[C_{\phi}] \simeq [(\mathbb{G}^r_{m} \times \mu_{\Lambda}) / \mathbb{G}^r_m]$. As before, for each $ 1 \leq i \leq r $, denote by $\chi_i$ the character of $S_{\phi} \simeq \mathbb{G}^r_m $ corresponding to $(0, \dotsc, 0, 1, 0, \dotsc, 0)$ where $1$ is in the $i^{\mathrm{th}}$-position.

Let $\G_m := \Res_{F/\Q_p} (\GL_{m, D})$ where $D$ is a division algebra over $F$ of degree $g^2$ and of invariant $-\tfrac{a}{g}$. We suppose that $g$ divides $ n_i := \dim(\tilde{\phi}_i)$ for each $1 \leq i \leq r$. Thus by lemma \ref{itm : relevance}, there exists a unique irreducible representation $\pi^{\G_m}_{\phi}$ in $\Rep_{\Lambda}(\G_m)$ whose Fargues-Scholze parameter is given by $\phi$. Let $\pi_{\phi_i}$ be the unique irreducible supercuspidal representation of $\GL_{m_i}(D)$ with Fargues-Scholze parameter $\phi_i$; denote by $\M$ the Levi subgroup $\displaystyle \prod_{i = 1}^r \Res_{F/\Q_p} (\GL_{m_i, D})$ of $\G_m$ where $m_i := \tfrac{n_i}{g}$ and denote by $\rho$ the irreducible representation $ \displaystyle \prod_{i=1}^r \pi_{\phi_i}$. Thus $(\M, \rho)$ is a supercuspidal pair of $\G_m$. Denote by $\mathfrak{s}_{\phi, \G_m}$ the inertial class of $(\M, \rho)$ and denote by $\Rep_{\Lambda}(\mathfrak{s}_{\phi, \G_m})$ the block of $\Rep_{\Lambda}(\G_m)$ corresponding to $\mathfrak{s}_{\phi, \G_m}$ (\cite[Theorems 10.4, 10.5]{SS16}).

Let $A_{\phi} := \displaystyle \Lambda[\Z^r \times \prod_{i = 1}^r \Syl_i] = \Lambda [X_1^{\pm 1}, \dotsc, X^{\pm 1}_r, Y_1, \dotsc, Y_r]/(Y^{\ell^{u_1}}_1, \dotsc Y^{\ell^{u_r}}_r) $ be the ring of global sections of $[C_{\phi}]$ where $|\Syl_i| = \ell^{u_i}$ for $1 \leq i \leq r$. For each Fargues-Scholze parameter $\phi'$ in $[C_{\phi}]$, we denote by $\mathfrak{m}_{\phi'}$ the associated maximal ideal of $A_{\phi}$.

\begin{proposition} \phantomsection \label{itm : description of blocks}
We suppose $\Lambda = \overline{\F}_{\ell}$. Then there is an equivalence between the block $\Rep_{\Lambda}(\mathfrak{s}_{\phi, \G_m})$ and the category of $A_{\phi}$-modules such that the unique irreducible representation in $\Rep_{\Lambda}(\mathfrak{s}_{\phi, \G_m})$ whose Fargues-Scholze parameter given by $\phi' \in [C_{\phi}]$ corresponds to the $A_{\phi}$-module $A_{\phi} / \mathfrak{m}_{\phi'}$. 
\end{proposition}
\begin{remark}
    If $\ell$ is a banal prime then the results can be deduced from \cite{DHKM24}.
\end{remark}
\begin{proof}
We prove the proposition by induction on $r$. The base case follows from theorem $\ref{itm : supercuspidal - spectral action}$. Hence suppose that $r > 1$ and the proposition is true for $1, \dotsc ,r-1 $.  

We choose $\chi = (m_1 \cdot a, \dotsc, m_r \cdot a) \in \Irr(S_{\phi}) \simeq \Z^r$ then $b_{\chi}$ is the basic element in $B(\G)$ such that $\kappa(b_{\chi}) = m \cdot a $. Moreover $\G_{b_{\chi}}(\Q_p) \simeq \G_m$ and $\pi_{\chi} \simeq \pi^{\G_m}_{\phi}$. Without loss of generality, we suppose $n_1 = \mathrm{max} \{ n_i \ | \ 1 \leq i \leq r \}$. 

Let $\chi' = \chi \otimes \chi_1$ then $\G_{b_{\chi'}}(\Q_p) \simeq D^{\times}_{-\lambda} \times \Res_{F/\Q_p} (\GL_{m-m_1, D}) $ where $\lambda = \tfrac{m_1 \cdot a + 1}{n_1}$ and $D_{-\lambda}$ is the inner form of $\Res_{F, \Q_p}(\GL_{m_1, F})$ of invariant $ -\lambda$ over $F$. Moreover $ \M' := \displaystyle D^{\times}_{-\lambda} (\Q_p) \times \prod_{i=2}^r \Res_{F/\Q_p} (\GL_{m_i, D}) (\Q_p) $ is a Levi subgroup of $ \G_{b_{\chi'}}$. Let $\pi'_{\phi_1}$ be the unique irreducible supercuspidal representation in $\Rep_{\Lambda}( D^{\times}_{-\lambda}(\Q_p) )$ whose Fargues-Scholze parameter is given by $\phi_1$ and denote by $\rho'$ the representation $ \pi'_{\phi_1} \times \displaystyle \prod_{i=2}^r \pi_{\phi_i} $ of $\M'$. Thus $(\M', \rho')$ is a supercuspidal pair of $\G_{b_{\chi'}}$. Denote by $\mathfrak{s}_{\phi, \chi'}$ the inertial class of $(\M', \rho')$ and denote by $\Rep_{\Lambda}(\mathfrak{s}_{\phi, \chi'})$ the block of $\Rep_{\Lambda}(\G_{b_{\chi'}})$ corresponding to $\mathfrak{s}_{\phi, \chi'}$ (\cite[Theorems 10.4, 10.5]{SS16}). The categories $\Rep_{\Lambda}(\mathfrak{s}_{\phi, \G_m})^{\omega}$ and $\Rep_{\Lambda}(\mathfrak{s}_{\phi, \chi'})^{\omega}$ are full sub-categories of $\Dlis^{[C_{\phi}]}(\Bun_{\G}, \Lambda)^{\omega}$ via the push-forward functors $i^{\rm ren}_{b_{\chi}!}$ and $i^{\rm ren}_{b_{\chi'!}}$ respectively. We will show that the functor $C_{\chi_1^{-1}} \star (-) $ gives an equivalence of the above two categories. 

Let $\pi$ be a finitely generated smooth representation in $\Rep_{\Lambda}(\mathfrak{s}_{\phi, \chi'})^{\omega}$. By the same arguments as in \cite[section 6.1]{GLn}, we can compute $i^*_{b_{\chi}}\big( C_{\chi^{-1}_1} \star (i^{\rm ren}_{b_{\chi'} !} \pi) \big)$ and then show that $ i_{b_{\chi}!}\big(i^*_{b_{\chi}}\big( C_{\chi^{-1}_1} \star (i^{\rm ren}_{b_{\chi'} !} \pi) \big)\big) $ belongs to the sub-category $ i^{\rm ren}_{b_{\chi}!} (\Rep_{\Lambda}(\mathfrak{s}_{\phi, \G_m})^{\omega}) $. Thus, to show that $C_{\chi^{-1}_1} \star (i^{\rm ren}_{b_{\chi'} !} \pi)$ belongs to $ i^{\rm ren}_{b_{\chi}!} (\Rep_{\Lambda}(\mathfrak{s}_{\phi, \G_m})^{\omega}) $, it is enough to show that the restriction $i^*_{b}\big( C_{\chi^{-1}_1} \star (i^{\rm ren}_{b_{\chi'} !} \pi) \big)$ vanishes whenever the stratum corresponding to $b \in B(\G)$ is different from that of $b_{\chi}$. Now let $b \in B(\G)$ be such an element. 

Let $\mu$ be the cocharacter $ \big( (-1, 0^{(n-1)}), \underbrace{(0^{(n)}), \dotsc, (0^{(n)})}_{d-1 \ \mathrm{times}}  \big) $ of $\G(\Q_p)$. Let $V_{\mu}$ be the highest weight representation of $\widehat{\G}$ of highest weight $\mu$ and $C_{\mu}$ be the vector bundle on $[C_{\phi}]$ corresponding to  $V_{\mu}$. We have an equality of $S_{\phi} \times \I_F$-representations
\[
V_{\mu} \simeq \bigoplus_{i = 1}^r \chi^{-1}_i \boxtimes \tilde{\phi}^{\vee}_{i | \I_F},
\]
and thus we have a decomposition
\[
C_{\mu} \star ( i^{\rm ren}_{b_{\chi'}!} \pi) = \bigoplus_{i = 1}^r (C_{\chi^{-1}_i} \star ( i^{\rm ren}_{b_{\chi'}!} \pi)) \boxtimes \tilde{\phi}^{\vee}_{i | \I_F}.
\]

However, $C_{\mu} \star ( i^{\rm ren}_{b_{\chi'}!} \pi) = \T_{V_{\mu}}(i^{\rm ren}_{b_{\chi'}!} \pi)$ where $\T_{V_{\mu}}$ denotes the Hecke operator associated to $V_{\mu}$. In particular, by using lemma \ref{shimhecke}, we deduce that there is a modification of type $\mu$ from $\E_{b_{\chi'}}$ to $\E_b$. By combining the proposition \ref{itm : vanishing} and the arguments as in \cite[Example 2.3]{GLn}, we deduce that $b$ is of the form $b_{\chi' \otimes \chi^{-1}_i}$ where $i$ is an integer and $ 2 \leq i \leq r $. We can then suppose that $\nu_{E_b}$ is maximal with respect to the Bruhat order among the set of $b$ in $B(\G)$ such that the restriction $i^*_{b}\big( C_{\chi^{-1}_1} \star (i^{\rm ren}_{b_{\chi'} !} \pi) \big)$ is non trivial. Since the Hecke operators preserve compact objects, we deduce that $C_{\chi^{-1}_1} \star ( i^{\rm ren}_{b_{\chi'}!} \pi)$ is a compact object. In particular, the cohomology groups of $i^*_{b}(C_{\chi^{-1}_1} \star ( i^{\rm ren}_{b_{\chi'}!} \pi))$ has an irreducible quotient. Denote one of this quotient by $\rho$ then its Fargues-Scholze parameter belongs to $[C_{\phi}]$. Since $b$ is maximal, we deduce that there exists some $k \in \Z$ such that
\[
\Hom_{\Dlis^{[C_{\phi}]}(\Bun_{\G}, \Lambda)^{\omega}}( C_{\chi^{-1}_1} \star \big( i^{\rm ren}_{b_{\chi'}!} \pi), i^{\rm ren}_{b !} \rho [k] \big) \neq 0.
\]

Now by using theorem \ref{itm : main theorem I} and the fact that $C_{\chi_1} \star (-)$ is an auto-equivalence of $\Dlis^{[C_{\phi}]}(\Bun_{\G}, \Lambda)^{\omega}$, we deduce that 
\[
\Hom_{\Dlis^{[C_{\phi}]}(\Bun_{\G}, \Lambda)^{\omega}} \big( i^{\rm ren}_{b_{\chi'}!} \pi, i^{\rm ren}_{b_{\chi' \otimes \chi_1 \otimes \chi^{-1}_i } !} \rho [k] \big) \neq 0,
\]
and it is a contradiction by \cite[lemma 2.6]{GLn}. We deduce that $C_{\chi^{-1}_1} \star (i^{\rm ren}_{b_{\chi'} !} \pi)$ belongs to $ i^{\rm ren}_{b_{\chi}!} (\Rep_{\Lambda}(\mathfrak{s}_{\phi, \G_m})^{\omega}) $.

Let $\pi'$ be a finitely generated smooth representation in $\Rep_{\Lambda}(\mathfrak{s}_{\phi, \G_m})^{\omega}$, we want to show that $C_{\chi_1} \star (i^{\rm ren}_{b_{\chi} !} \pi')$ belongs to $ i^{\rm ren}_{b_{\chi'}!} (\Rep_{\Lambda}(\mathfrak{s}_{\phi, \chi'})^{\omega}) $. As above, by using the decomposition of $S_{\phi} \times \I_F$ representations
\[
V_{-\mu} \simeq \bigoplus_{i = 1}^r \chi_i \boxtimes \tilde{\phi}_{i | \I_F},
\]
and by relating the spectral action of the vector bundle $C_{-\mu}$ with the Hecke operator $\T_{V_{-\mu}}$ and local Shimura varieties (\ref{shimhecke}), we can show that $i^*_b \big(C_{\chi_1} \star (i^{\rm ren}_{b_{\chi} !} \pi')\big)$ vanishes if the strata corresponding to $b$ and $b_{\chi'}$ are not the same. Moreover, as in \cite[section 6.1]{GLn}, we can also compute $i^*_{b_{\chi'}} \big(C_{\chi_1} \star (i^{\rm ren}_{b_{\chi} !} \pi')\big)$ by using the known cases of the Harris-Viehmann conjecture (or some analogue of Boyer's trick) for $b_{\chi}$ and $b_{\chi'}$. 

We have then $C_{\chi_1} \star(i^{\rm ren}_{b_{\chi}!} \Rep_{\Lambda}(\mathfrak{s}_{\phi, \G_m})^{\omega} ) \subset i^{\rm ren}_{b_{\chi'}!} \Rep_{\Lambda}(\mathfrak{s}_{\phi, \chi'})^{\omega} $ and $C_{\chi^{-1}_1} \star(i^{\rm ren}_{b_{\chi'}!} \Rep_{\Lambda}(\mathfrak{s}_{\phi, \chi'})^{\omega} ) \subset i^{\rm ren}_{b_{\chi}!} \Rep_{\Lambda}(\mathfrak{s}_{\phi, \G_m})^{\omega} $ and $C_{\chi_1} \star(-)$ is an auto-equivalence of $\Dlis^{[C_{\phi}]}(\Bun_{\G}, \Lambda)^{\omega}$. Therefore $C_{\chi_1} \star (-)$ induces an equivalence between the two categories $\Rep_{\Lambda}(\mathfrak{s}_{\phi, \G_m})^{\omega}$ and 
$\Rep_{\Lambda}(\mathfrak{s}_{\phi, \chi'})^{\omega}$. Thus the category $\Rep_{\Lambda}(\mathfrak{s}_{\phi, \G_m})^{\omega}$ is equivalence to the category of finitely generated $A_{\phi}$-modules and the block $\Rep_{\Lambda}(\mathfrak{s}_{\phi, \G_m})$ is equivalence to the category of $A_{\phi}$-modules. 

\end{proof}

For a character $\chi \in \Irr(S_{\phi})$, we can consider the irreducible representation $\pi_{\chi}$ in $\Rep_{\Lambda}(\G_{b_{\chi}}(\Q_p))$ whose Fargues-Scholze parameter is given by $\phi$.  By the construction of the pair $(b_{\chi}, \pi_{\chi})$, there exists a parabolic subgroup $\rP_{\chi}(\Q_p) = \M_{\chi}(\Q_p) \ltimes \U $ of $\G_{b_{\chi}}(\Q_p)$ where $\M_{\chi}(\Q_p)$ is the Levi factor and a supercuspidal representation $\tau_{\chi}$ of $\M_{\chi}(\Q_p)$ such that $\pi_{\chi} = \Ind_{\rP(\Q_p)}^{\G_{b_{\chi}}(\Q_p)}(\tau_{\chi})$. Denote by $\mathfrak{s}_{\phi, \chi}$ the inertial class of $(\tau_{\chi}, \M_{\chi}(\Q_p) )$ of $\G_{b_{\chi}}(\Q_p)$ and $\Rep_{\Lambda}(\mathfrak{s}_{\phi, \chi})$ the block corresponding to $\mathfrak{s}_{\phi, \chi}$ in $\Rep_{\Lambda}(\G_{b_{\chi}}(\Q_p))$. Similarly, denote by $\mathfrak{t}_{\phi, \chi}$ the inertial class of $(\tau_{\chi}, \M_{\chi}(\Q_p) )$ of $\M_{\chi}(\Q_p)$ and $\Rep_{\Lambda}(\mathfrak{t}_{\phi, \chi})$ the block corresponding to $\mathfrak{t}_{\phi, \chi}$ in $\Rep_{\Lambda}(\M_{\chi}(\Q_p))$.

\begin{corollary} \phantomsection \label{itm : structure of a block}
We suppose $\Lambda = \overline{\F}_{\ell}$. Then there is an equivalence between the block $\Rep_{\Lambda}(\mathfrak{s}_{\phi, \chi})$ and the category of $A_{\phi}$-modules such that the unique irreducible representation in $\Rep_{\Lambda}(\mathfrak{s}_{\phi, \chi})$ whose Fargues-Scholze parameter given by $\phi' \in [C_{\phi}]$ corresponds to the $A_{\phi}$-module $A_{\phi} / \mathfrak{m}_{\phi'}$.    
\end{corollary}

There is also an equivalence between the block $\Rep_{\Lambda}(\mathfrak{t}_{\phi, \chi})$ and the category of $A_{\phi}$-modules such that the irreducible supercuspidal representation $\tau$ in $\Rep_{\Lambda}(\mathfrak{t}_{\phi, \chi})$ corresponds to the $A_{\phi}$-module $A_{\phi} / \mathfrak{m}_{\phi'}$ where $\phi'$ is the Fargues-Scholze parameter of $\Ind^{\G_{b_{\chi}}(\Q_p)}_{\rP_{\chi}(\Q_p)}(\tau)$ and $\mathfrak{m}_{\phi'}$ is the maximal ideal of $A_{\phi}$ attached to $\phi'$. Let $ \mathcal{W}_{\mathfrak{t}_{\phi, \chi}} $ be the compact pro-generator of $\Rep_{\Lambda}(\mathfrak{t}_{\phi, \chi})$ corresponding to the module $A_{\phi}$ and denote by $\mathcal{W}_{\mathfrak{s}_{\phi, \chi}}$ the normalised parabolic induction $\Ind^{\G_{b_{\chi}}(\Q_p)}_{\rP_{\chi}(\Q_p)}(\mathcal{W}_{\mathfrak{t}_{\phi, \chi}})$. By using the fact that parabolic induction is an exact functor and preserves projective objects, we deduce that $\mathcal{W}_{\mathfrak{s}_{\phi, \chi}}$ is the compact-pro-generator of $\Rep_{\Lambda}(\mathfrak{s}_{\phi, \chi})$ whose corresponding $A_{\phi}$-module is given by $A_{\phi}$. Denote by $\mathcal{F}_{\mathcal{W}(\chi)}$ the sheaf $i^{\rm ren}_{b_{\chi}!}(\mathcal{W}_{\mathfrak{s}_{\phi, \chi}})$ in $\Dlis^{[C_{\phi}]}(\Bun_{\G}, \Lambda)^{\omega}$. We can also embed the category $\Rep_{\Lambda}(\mathfrak{s}_{\phi, \chi})$ in $\Dlis^{[C_{\phi}]}(\Bun_{\G}, \Lambda)$ by the functor $i^{\rm ren}_{b_{\chi}!}(-)$. It induces a map 
\[
\Psi_{\G}: \mc Z^{\rm spec}(\G, \Lambda) \longrightarrow \End(\id_{\Rep_{\Lambda}(\mathfrak{s}_{\phi, \chi})}) \simeq A_{\phi}.
\]
Denote the restriction of this map to $\mathcal{O}([C_{\phi}])$ by $\Psi_{\G}(\phi, \chi)$. We have the following results.

\begin{proposition} \phantomsection \label{itm : spectral action - progenerator}
 Let $\phi = \phi_1 \oplus \dotsc \oplus \phi_r$ be an $L$-parameter satisfying the conditions of theorem \ref{itm : main theorem I}. With the above notations, for each $\chi \in \Irr(S_{\phi})$ we have
    \[
    C_{\chi} \star \mathcal{F}_{\mathcal{W}(\Id)} \simeq \mathcal{F}_{\Wchi}. 
    \]      
\end{proposition}
\begin{proof}
    Essentially the same proof as in \cite[Theorem 9.3]{GLn} still works here. More precisely, step $1$ in loc.cit. is true by theorem \ref{itm : supercuspidal - spectral action}. Then with theorem \ref{itm : main theorem I} and proposition \ref{itm : description of blocks} in hand, we can follow the same steps as in  \cite[Theorem 9.3]{GLn}.
\end{proof}

\begin{corollary} \phantomsection \label{itm : identification between spectral actions and Bernstein center actions}
    There is some change of variables such that the map
     \[
     \Psi_{\G}(\phi, \chi)([C_{\phi}]) : A_{\phi} \simeq \mathcal{O}([C_{\phi}]) \longrightarrow \End(\id_{\Rep_{\Lambda}(\mathfrak{s}_{\phi, \chi})}) \simeq A_{\phi}
     \]
     becomes the identity map and moreover the maximal ideal corresponding to $\phi$ is $(X_1-1, \dotsc, X_r-1, Y_1, \dotsc, Y_r) \subset \mathcal{O}([C_{\phi}]) \simeq A_{\phi} $ and the irreducible representation $\pi_{\chi}$ corresponds to the $A_{\phi}$-module whose underlying vector space is isomorphic to $\Lambda$ and where $Y_i$ acts by multiplying by $0$ and $X_i$ acts by multiplying by $1$.
\end{corollary}

\begin{theorem} \phantomsection \label{itm : spectral action - general}
\begin{enumerate}
    \item Let $\bL$ be a coherent complex of $A_{\phi}$-modules. For each $\chi \in \Irr(S_{\phi})$, let $\bL(\chi)$ be the coherent complex $\bL$ together with the action of $\bb G_m^r$ acting by $\chi$ so that it gives rise to a coherent complex on $[C_{\phi}]$. We denote by $\pi_{\bL}(\chi)$ the complex of smooth $\G_{b_{\chi}}(\Q_p)$-representations in the derived category of the Bernstein block $\Rep_{\Lambda}(\mathfrak{s}_{\phi, \chi})$ corresponding to the coherent complex $\bL$. Then
\[
\bL(\chi) \star \Fs \simeq i_{b_{\chi} !} ( \delta^{-1/2}_{b_{\chi}} \otimes \pi_{\bL}(\chi))[-d_{\chi}].
\]
\item We have an orthogonal decomposition
    \[
    \Dlis^{[C_{\phi}]}(\Bun_{\G}, \Lambda)^{\omega} \simeq \bigoplus_{\chi \in \Irr(S_{\chi})}\mathrm{D}(i^{\rm ren}_{b_{\chi}!}\Rep_{\Lambda}(\mathfrak{s}_{\phi, \chi}))^{\omega}.
    \]
\end{enumerate}
    
\end{theorem}
\begin{proof}
    With propositions \ref{itm : description of blocks}, \ref{itm : spectral action - progenerator} and corollary \ref{itm : identification between spectral actions and Bernstein center actions}, we can use the same arguments as in \cite[Theorems 9.5, 9.7]{GLn} to prove these results.
\end{proof}
\begin{remark}
    While finishing this paper, the author is informed that Konrad Zou also uses the work \cite{GLn} as the main idea to prove a version of this result \cite{Zou25}.
\end{remark}
\section{Some Applications to the cohomology of Shimura varieties}

\subsection{Torsion classes in the cohomology of some Shimura varieties: the generic part} \textbf{} \label{itm : generic part}

In this paragraph we analyze the cohomology with torsion coefficient of some Shimura varieties of type $A$ as classified in \cite{Kot92} in the style of \cite{CS, Ko1, HL}. We use the same notations as in section \ref{itm : Shimura} and suppose that $ \textbf{G}(\Q_p) \simeq \G(\Q_p) \times \Q_p^{\times} \simeq \Res_{F/\Q_p} \GL_{n, F}(\Q_p) \times \Q_p^{\times} $ with $F / \Q_p$ unramified. For each $K = K_p K^p$ where $K^p$ is a sufficiently small open compact subgroup of $ \textbf{G}(\A^{\infty, p})$ and $K_p \subset G(\Q_p)$. As before, we have a Shimura variety $\Sh_{K}$ defined over its reflex field. We also have a system $\Sh_{K^p} :=  (\Sh_{K^pK_p})_{K_p}$, and we have a complex
\[
R\Gamma_c(\Sh_{K^p}, \Lambda) := \mathop{\mathrm{colim}}_{\overrightarrow{K_p \rightarrow \{1 \}}} R\Gamma_c(\Sh_{K_pK^p}, \Lambda)
\]
of admissible $\textbf{G}(\Q_p) \times W_{F}$-representations\footnote{actually we do not need the $W_{F}$-equivariant property} (the cohomology here is the étale cohomology after base change to an algebraically closed field). 

For each $b \in B(\G, -\mu)$ where $\mu$ is the cocharacter defining the Shimura variety then we also have an Igusa variety $\Ig_{b, K^p}$ and a complex 
\[
R\Gamma_c(\Ig_{b, K^p}, \Lambda) := \mathop{\mathrm{colim}}_{\overrightarrow{m \rightarrow \{ \infty \} } } R\Gamma_c(\Ig_{b, m, K^p}, \Lambda).
\]
of admissible $\G_b(\Q_p)$-representations. 

We let $\phi$ be an $L$-parameter of $^{L}\widehat{\G}(\Lambda)$ satisfying the condition (A1). Thus as in section \ref{itm : Shimura}, we use the superscription $(-)^{[\phi]}$ to denote the projection to the derived category of the block $\Rep_{\Lambda}(\mathfrak{s}_{\phi, \Id})$ \cite[Theorems 10.4, 10.5]{SS16}. 

\begin{theorem} \phantomsection \label{itm : torsion vanishing}
    Let $\ell \neq p$ be an odd prime number then the complex $R\Gamma_c(\Sh_{K^p}, \ov \F_{\ell})^{[\phi]}$ is concentrated in degrees $[0, d]$ where $d = \langle 2\rho_{\G}, \mu \rangle $ is the dimension of the Shimura variety $\Sh_{K^p}$. If $\Sh_{K^p}$ is compact or $\phi$ is irreducible then the complex $R\Gamma_c(\Sh_{K^p}, \ov \F_{\ell})^{[\phi]}$ is concentrated in degree $d$ (see also \cite[Example 6.18]{KS}). 
\end{theorem}
\begin{proof}
The same analysis as in \cite[section 5]{HL} still works, we just sketch the proof. First of all we can use Mantovan's formula to relate the complex $R\Gamma_c(\Sh_{K^p}, \ov \F_{\ell})$ with the Hecke operators $\T_{-\mu}$ and the complexes $R\Gamma_c(\Ig_{b, K^p}, \ov \F_{\ell})$ where $b \in B(\G, -\mu)$. More precisely when the Shimura variety is non-compact, we need to use the complexes $R\Gamma_{c-\partial}(\Ig_{b, K^p}, \ov \F_{\ell})$ of the minimal compactification of Igusa varieties \cite[Theorem 1.12]{HL}, \cite{Zha23}. Since the minimal compactification of Igusa varieties are affine \cite{CS, CS1}, we can deduce that the complex $R\Gamma_{c-\partial}(\Ig_{b, K^p}, \ov \F_{\ell})$ is concentrated below degree $d_b = \langle \nu_b, 2\rho_{\G} \rangle$ for $b \in B(\G, -\mu)$ \cite[Proposition 3.7]{HL}. On the other hand, we can completely compute the Hecke operator $\T_{-\mu}$ and then we can check that this operator is $t$-exact when restrict to the sub-category $\Dlis^{[\phi]}(\Bun_{\G}, \ov \F_{\ell})$ and the first conclusion of the theorem follows. 

If $\Sh_{K^p}$ is compact then by Poincare duality, it is clear that $R\Gamma_c(\Sh_{K^p}, \ov \F_{\ell})^{[\phi]}$ is concentrated in degree $d$. If $\phi$ is irreducible then in fact, the contribution of non-basic strata to $R\Gamma_c(\Sh_{K^p}, \ov \F_{\ell})^{[\phi]}$ vanishes since $\T_{-\mu}$ preserves Fargues-Scholze parameters and if $\pi_b$ is an irreducible representation in $\Rep_{\ov \F_{\ell}}(\G_b(\Q_p))$ then its Fargues-Scholze parameter is not irreducible. Moreover, the Igusa variety corresponding to the basic stratum is zero dimension and thus the complex $R\Gamma_{c-\partial}(\Ig_{b, K^p}, \ov \F_{\ell})$ concentrated in degree $0$. In this case, one can also deduce that $R\Gamma_c(\Sh_{K^p}, \ov \F_{\ell})^{[\phi]}$ is concentrated in degree $d$.
\end{proof}
\begin{remark}
The conclusion of the above theorem is also true if we consider the localization $R\Gamma_c(\Sh_{K^p}, \ov \F_{\ell})_{\phi}$ with respect to the parameter $\phi$ as in \cite{HL}.
\end{remark}
\begin{remark}
    See the paper \cite{KS} for more conjectures on vanishing of cohomology of Shimura varieties.
\end{remark}
We can relate $R\Gamma_c(\Sh_{K^p}, \ov \F_{\ell})$ and $R\Gamma_c(\Sh_{K^p}, \ov \Z_{\ell})$ when the Shimura variety is compact. If we suppose further that the prime number $\ell$ is banal for $\G(\Q_p)$ then we can consider $R\Gamma_c(\Sh_{K^p}, \ov \Z_{\ell})$ as a complex of admissible $\ov \Z_{\ell}[\G(\Q_p)]$-modules (here $R\Gamma_c(\Sh_{K^p}, \ov \Z_{\ell}) := R\Gamma_c(\Sh_{K^p}, \Z_{\ell}) \otimes_{\Z_{\ell}} \ov \Z_{\ell}$ ). Let $\ov \pi$ be the corresponding supercuspidal representation in $\Rep_{\ov \F_{\ell}}(\G(\Q_p))$ and by \cite[Proposition 4.16(3)]{DHKM24}, we can lift $\ov \pi$ to an $\mathcal{O}_{K}$-integral absolute irreducible cuspidal $K[\G]$-module $\pi$ where $K$ is some number field. Thus we can consider the block of the $\ov \Z_{\ell}[\G(\Q_p)]$-modules constructed from $\pi$ \cite[Theorem 4.22]{DHKM24} (see \cite[Section 4]{DHKM24} for more details on the definition of the block corresponding to $\phi$). We use the superscript $(-)^{[\phi]}$ to denote the projection to the derived category of the block corresponding to $\phi$. Since $R\Gamma_c(\Sh_{K^p}, \ov \F_{\ell})^{[\phi]}$ is concentrated in degree $d$, the complex $R\Gamma_c(\Sh_{K^p}, \ov \Z_{\ell})^{[\phi]}$ is concentrated in degree $d$ and torsion free. Hence $R\Gamma_c(\Sh_{K^p}, \ov \F_{\ell})^{[\phi]}$ is the reduction modulo $\ell$ of $R\Gamma_c(\Sh_{K^p}, \ov \Z_{\ell})^{[\phi]}$. By taking the colimit when $K_p$ tends to $\{ 1 \}$, we see that the complex $R\Gamma_c(\Sh, \ov \Z_{\ell})^{[\phi]}$ is concentrated in degree $d$ and torsion free and $R\Gamma_c(\Sh, \ov \F_{\ell})^{[\phi]}$ is the reduction modulo $\ell$ of $R\Gamma_c(\Sh, \ov \Z_{\ell})^{[\phi]}$.

\subsection{Torsion classes in the cohomology of some Shimura varieties: the non-generic part} \textbf{} \label{itm : non-generic part}

In \cite{Boy19, CT}, the authors proved some results on torsion vanishing of cohomology of some Shimura varieties beyond the generic case. In this paragraph we suppose that $\ell$ is a banal prime and we prove some similar results.

Let $D_{1/n}$ be the division algebra of invariant $\tfrac{1}{n}$ over the unramified extension $F$ of $\Q_p$ such that $D_{1/n}^{\times}$ is an inner form of $\G:= \Res_{F/\Q_p} \GL_{n, F}$. Let $\rho$ be a supercuspidal representation in $\Rep_{\ov \F_{\ell}}(D_{1/n}^{\times})$, then by lemma \ref{itm : parameter of divison algebra}, the Fargues-Scholze parameter $\phi^{\rho}$ of $\rho$ is of the form $\phi \oplus \phi(1) \oplus \dotsc \oplus \phi(r-1)$. Let $\mu_{\Std}$ be a cocharacter of $ D^{\times}_{1/n} \otimes_F \ov \Q_p $ the form $(1, 0^{(n-1)})$ for one $\sigma_0 \in \Gal(F / \Q_p)$ and $(0^{(n)})$ for all other elements in $\Gal(F / \Q_p)$. Let $b_0 \in B(\G, \mu_{\Std})$ be the element corresponding to the vector bundle $\mathcal{O}(-\tfrac{1}{n})$ then $\G_{b_0}(\Q_p) \simeq D_{1/n}^{\times}$ and the local Shimura variety $\Sht(\G, b_0, -\mu_{\Std})$ is the Lubin-Tate space. 

\begin{proposition} \phantomsection \label{itm : bound on cohomological degree of Lubin-Tate space}
    The complex $R\Gamma_c(\Sht(\G, b_0, -\mu_{\Std}))[\rho]$ of $\G(\Q_p)$-representations is concentrated in degree $ [1-n, r-1] $. The same conclusion is also true for basic local Shimura variety associated to $\GL_n$ and $\mu_{\Std}$.
\end{proposition}
\begin{remark}
    We can deduce stronger bound by Boyer's works \cite{Boy09, Boy-libre}. However, one can apply the arguments in our proof to obtain a similar bound for basic local Shimura varieties/Rapoport-Zink spaces associated to a fully Hodge-Newton reducible pair $(\G, b, \mu)$, if one has enough understanding about the representations with $\ov \F_{\ell}$-coefficient of $\G(\Q_p)$.  
\end{remark}
\begin{proof}
    We prove the result by induction on the number of irreducible factors of $\phi^{\rho}$. We can deduce the case where $ \phi^{\rho}$ is irreducible by theorem \ref{itm : spectral action - general}. Suppose that the conclusion of the proposition is true (for all $n$) up to $r-1$.

    We consider the Lubin-Tate case, the arguments work similarly for the other case (remark that the pairs $(\GL_n, -\mu_{\Std})$, $(\GL_n, \mu_{\Std})$ are fully Hodge-Newton reducible). 

    We consider the Shimura variety of type $A$ as in section \ref{itm : Shimura}. Thus we have a CM field $E = \mathcal{K} E_0$ where $\mathcal{K}$ is a quadratic imaginary field. We also have a general unitary group $\textbf{G}$ defined over $\Q$ satisfying $\textbf{G}(\Q_p) \simeq \G(\Q_p)$. 
    We suppose that the associated cocharacter is given by $\mu_{\Std}$ and the Shimura variety that we consider is simple in the terminology of Kottwitz-Harris-Taylor and in particular it is \textit{compact}.

    We choose a prime $p'$ that splits in $\mathcal{K}$ and that $\textbf{G}(\Q_{p'}) \simeq \GL_n(F) \times \Q^{\times}_{p'}$ where $F'$ is an unramified extension of $\Q_{p'}$ such that $[F' : \Q_{p'}] = [F : \Q_p]$. Suppose that $\ell$ is also banal for $\GL_n(F')$. For each $K \subset \textbf{G}(\A^{\infty})$ sufficiently small open subgroup, we have a Shimura variety $\Sh_{K}$ defined over $E$. As before, for each $K^{p, p'}$ sufficiently small open compact subgroups of $ \textbf{G}(\A^{\infty, p, p'})$, we can consider the system
\[
\Sh_{K^{p, p'}} := (\Sh_{K^{p, p'}K_{p, p'}})_{K_{p,p'}},
\]
where $K_{p, p'} \subset \textbf{G}(\Q_p) \times \textbf{G}(\Q_{p'})$ is an open compact subgroup. Similarly we have a complex $R\Gamma_c(\Sh_{K^{p, p'}}, \Lambda)$ of admissible $\textbf{G}(\Q_p) \times \textbf{G}(\Q_{p'}) \times W_{F}$-representations.

Now we choose a supercuspidal representation $\ov \pi_{p'}$ in $\Rep_{\ov \F_{\ell}}(\GL_n(F'))$ whose Fargues-Scholze parameter is given by $\phi'$. Therefore, by theorem \ref{itm : torsion vanishing}, we see that the complex $R\Gamma_c(\Sh_{K^{p, p'}}, \ov \F_{\ell})^{[\phi^{\rho}, \phi']}$ is concentrated in degree $d = \langle 2\rho_{\G}, \mu \rangle$ where we use the superscript $(-)^{[\phi^{\rho}, \phi']}$ to denote the projection to the direct subcategory of $\Rep_{\ov \F_{\ell}} ( \G(\Q_p) \times \GL_n(F')) $ forming by the blocks containing the representations $\pi \times \pi'$ whose Fargues-Scholze parameters are given by $\phi^{\rho} \times \phi'$. Denote by $\Rep_{\ov \F_{\ell}}(\phi^{\rho})$ the block of $\Rep_{\ov \F_{\ell}}(\G(\Q_p))$ containing the irreducible representation of $\G(\Q_p)$ whose Fargues-Scholze parameter is given by $\phi^{\rho}$. Now let $\pi$ be an irreducible representation in the block $\Rep_{\ov \F_{\ell}}(\phi^{\rho})$ then its Fargues-Scholze parameter is of the form $(\phi \otimes \chi_1) \oplus \dotsc \oplus (\phi \otimes \chi_{r}) $ for unramified characters $\chi_1, \dotsc, \chi_r$. By taking $K_{p'}$-invariant with respect to an appropriate compact subgroup $K_{p'} \subset \textbf{G}(\Q_{p'})$, we see that the complex $R\Gamma_c(\Sh_{K^{p}}, \ov \Z_{\ell})^{[\ov\phi, \ov\phi']}$ also concentrated in middle degree where $K^p = K_{p'}K^{p, p'} \subset \textbf{G}(\A^{\infty, p})$.

Now by \cite[Theorem 1.12]{HL}, the complex $R\Gamma_c(\Sh_{K^{p}}, \ov \F_{\ell})$ has a filtration as a complex of $ \mathrm{\G}(\Q_p) \times W_{\Q_p} $ representations with graded pieces isomorphic to $ (R\Gamma_c(\G, b, -\mu_{\Std}) \otimes^{\bL}_{\mathcal{H}(\G_b)}R\Gamma_c(\Ig_{b, K^{p}}, \ov \F_{\ell}))[2d_b-d](-\tfrac{d}{2})$ where $b$ runs in the Kottwitz set $B(\G, \mu_{\Std})$ and where $d = \dim \Sh_{K^{p}}$. More precisely, this filtration comes from the interpretation of the cohomology of the Shimura variety by some Igusa sheaves on $\Bun_{\G}$ \cite[Theorem 8.4.10]{DHKM24} and the stratification of $\Bun_{\G}$ by $B(\G)$. 

The set $B(\G, \mu_{\Std})$ contains exactly $n$ elements $b_0, b_1, \dotsc, b_{n-1}$ where $b_i$ corresponds to the vector bundle $\mathcal{O}^i \oplus \mathcal{O}(-\tfrac{1}{n-i})$ and $\G_{b_i}(\Q_p) \simeq \GL_i(F) \times D^{\times}_{1/(n-i)}$, where $D_{1/(n-i)}$ is the division algebra over $F$ of invariant $1/(n-i)$. Denote by $t$ the integer $\tfrac{n}{r} = \dim \phi$ then by lemma \ref{itm : parameter of divison algebra},if $i \notin \{0, t, 2t, \dotsc, (r-1)t \}$ then there is no irreducible representation in the category $\Rep_{\ov \F_{\ell}}(\G_{b_i}(\Q_p))$ whose corresponding parameter is of the form $(\phi \otimes \chi_1) \oplus \dotsc \oplus (\phi \otimes \chi_{r})$. Thus if $i \notin \{0, t, 2t, \dotsc, (r-1)t \}$ then the contribution of the term $ (R\Gamma_c(\G, b_i, -\mu_{\Std}) \otimes^{\bL}_{\mathcal{H}(\G_{b_i})}R\Gamma_c(\Ig_{b_i, K^{p}}, \ov \F_{\ell}))[2d_{b_i}-d](-\tfrac{d}{2})$ to $R\Gamma_c(\Sh_{K^{p}}, \ov \F_{\ell})^{[\phi^{\rho}, \phi']}$ is trivial.

For $i \in \{0, t, 2t, \dotsc, (r-1)t \}$ we want to estimate the degrees where the cohomology groups of the terms $(R\Gamma_c(\G, b_i, -\mu_{\Std}) \otimes^{\bL}_{\mathcal{H}(\G_{b_i})}R\Gamma_c(\Ig_{b_i, K^{p}}, \ov \F_{\ell}))^{[\phi^{\rho}, \phi']}[2d_{b_i}-d](-\tfrac{d}{2})$ vanish. First, by affineness of the Igusa varieties, one deduce that $R\Gamma_c(\Ig_{b_i, K^{p}}, \ov \F_{\ell})$ concentrates in degrees smaller than $d_{b_i}$. Let $\rho'$ be an irreducible representation in $\Rep_{\ov \F_{\ell}}(\G_{b_i}(\Q_p))$ whose parameter is given by $(\phi \otimes \chi_1) \oplus \dotsc \oplus (\phi \otimes \chi_{r})$, then by Boyer's trick \cite[Proposition 5.1]{GLn} and induction hypothesis, the complex $R\Gamma_c(\G, b_{at}, -\mu_{\Std})[\rho']$ concentrates in degrees smaller than $ d_{b_{at}} + (r-1-a) $ if $a \neq 0$. More precisely, the term $(r-1-a)$ comes from induction hypothesis for $R\Gamma_c(\Sht(\ov \G, \ov b, - \ov \mu_{\Std}))$ where $\ov \G := \Res_{F / \Q_p} \GL_{n-at, F}$, $\ov b$ corresponds to the vector bundle $\mathcal{O}(-\tfrac{1}{n-at})$, $\ov \mu_{\Std}$ is the cocharacter $(1, 0^{(n-at-1)})$. The term $d_{b_{at}}$ is the dimension of the diamond $\mathcal{J}^U$ in \cite[Proposition 5.1]{GLn}. All in all, if $a \neq 0$ then the contribution of the term $(R\Gamma_c(\G, b_{at}, -\mu_{\Std}) \otimes^{\bL}_{\mathcal{H}(\G_{b_{at}})}R\Gamma_c(\Ig_{b_{at}, K^{p}}, \ov \F_{\ell}))[2d_{b_{at}}-d](-\tfrac{d}{2})$ to $R\Gamma_c(\Sh_{K^{p}}, \ov \F_{\ell})^{[\phi^{\rho}, \phi']}$ concentrates in the degrees smaller than $d + (r-1-a)$.

However, we know that $R\Gamma_c(\Sh_{K^{p}}, \ov \F_{\ell})^{[\phi^{\rho}, \phi']}$ is concentrated in degree $d$. Therefore, the complex $(R\Gamma_c(\G, b_0, -\mu_{\Std}) \otimes^{\bL}_{\mathcal{H}(\G_{b_0})}R\Gamma_c(\Ig_{b_0, K^{p}}, \ov \F_{\ell}))^{[\phi^{\rho}, \phi']}[2d_{b_0}-d](-\tfrac{d}{2})$ concentrates in degrees smaller than $d + r - 1$. Note that $d_{b_0} = 0$, we deduce that $ (R\Gamma_c(\G, b_0, -\mu_{\Std}) \otimes^{\bL}_{\mathcal{H}(\G_{b_0})}R\Gamma_c(\Ig_{b_0, K^{p}}, \ov \F_{\ell}))^{[\phi^{\rho}, \phi']}$ is concentrated in degree smaller than $r - 1$.

To conclude, we remark that the Igusa variety $\Ig_{b_0}$ is zero dimension and its cohomology groups concentrated in degree $0$. Let $\Tilde{\rho}$ and $\pi_{p'}$ be characteristic $0$ lift of $\rho$ and $\ov \pi_{p'}$, respectively. Since the Shimura variety that we are considering is a simple Shimura variety of Kottwitz-Harris-Taylor, we can compute the $\ell$-adic cohomology $R\Gamma_c(\Ig_{b_0, K^{p}}, \ov \Q_{\ell}))$ by \cite[Theorem 6.7]{SWS}. Thus we can globalize $\Tilde{\rho} \boxtimes \pi_{p'} $ (up to some unramified twist) to an automorphic representation that appears in $R\Gamma_c(\Ig_{b_0, K^{p}}, \ov \Q_{\ell}))$. Hence $R\Gamma_c(\Sht(\G, b_0, -\mu_{\Std}))[\rho]$ is a direct summand of $ (R\Gamma_c(\G, b_0, -\mu_{\Std}) \otimes^{\bL}_{\mathcal{H}(\G_{b_0})}R\Gamma_c(\Ig_{b_0, K^{p}}, \ov \F_{\ell}))^{[\phi^{\rho}, \phi']}$ and we are done.

\end{proof}

The local vanishing result implies global vanishing result as in \cite[Section 6.28]{KS}. We consider Shimura variety of type $A$ (not necessarily compact) as in the beginning of section \ref{itm : Shimura}. We then have the CM field $E$ and the general unitary group $\textbf{G}$ defined over $\Q$ satisfying $\textbf{G}(\Q_p) \simeq \G(\Q_p)$. We suppose that the associated cocharacter is given by $\mu_{\Std}$ or $-\mu_{\Std}$. Let $\pi$ be an irreducible representation in $\Rep_{\ov \F_{\ell}}\G(\Q_p)$ whose Fargues-Scholze parameter $\phi$ is of the form $\phi_1 \oplus \phi_2 \oplus \dotsc \oplus \phi_k$ where the $\phi_i$'s are irreducible. Let $r$ be the maximum length of the sequences of the form $\psi, \psi(1), \dotsc, \psi(t)$ whose terms are in the set $\{ \phi_i \ | \ 1 \leq i \leq k \}$.
We use the subscript $(-)_{\phi}$ to denote the localization with respect to $\phi$ as in \cite[Appendix A]{HL}


\begin{theorem} \phantomsection \label{itm : beyond generic}
    The complex $R\Gamma_c(\Sh_{K^p}, \ov \F_{\ell})_{\phi}$ of $\G(\Q_p) \times W_{F}$-representations concentrates in degree smaller than $d + r - 1$ where $d = \dim \Sh_{K^p}$. If $\Sh_{K^p}$ is compact then $R\Gamma_c(\Sh_{K^p}, \ov \F_{\ell})_{\phi}$ concentrates in degrees $[d - r +1 , d + r - 1]$. 
\end{theorem}
\begin{proof}
 We consider the case where the cocharacter defining the Shimura variety is given by $\mu_{\Std}$. The argument for the other case is similar. 
 
 Note that the cohomology of Shimura and Igusa varieties are admissible. Thus, the complex $R\Gamma_c(\Sh_{K^p}, \ov \F_{\ell})_{\phi}$ has a filtration as a complex of $\mathrm{\G}(\Q_p) \times W_{F}$-representations with graded pieces isomorphic to $ (R\Gamma_c(\G, b, -\mu_{\Std}) \otimes^{\bL}_{\mathcal{H}(\G_b)}R\Gamma_c(\Ig_{b, K^p}, \ov \F_{\ell})_{\phi})[2d_b-d](-\tfrac{d}{2})$ and where $b$ runs in the Kottwitz set $B(\G, \mu_{\Std})$ \cite[Theorem 1.12, lemma 4.2]{HL}.

The complex $R\Gamma_c(\Ig_{b, K^p}, \ov \F_{\ell})_{\phi}$ concentrates in degrees smaller than $d_b$ by the affineness of Igusa varieties. Now if $\rho$ is an irreducible representation whose Fargues-Scholze parameter is given by $\phi$ then the complex $R\Gamma_c(\G, b, -\mu_{\Std}) \otimes^{\bL}_{\mathcal{H}(\G_b)} \rho$ concentrates in degrees smaller than $d_b + r - 1$. In fact, as before the term $(r-1)$ comes from the bound on the cohomology of Lubin-Tate spaces, by proposition \ref{itm : bound on cohomological degree of Lubin-Tate space}. The term $d_b$ is the dimension of the diamond $\mathcal{J}^U$ in \cite[Proposition 5.1]{GLn}. Therefore the complex $(R\Gamma_c(\G, b, -\mu_{\Std}) \otimes^{\bL}_{\mathcal{H}(\G_b)}R\Gamma_c(\Ig_{b, K^p}, \ov \F_{\ell})_{\phi})[2d_b-d](-\tfrac{d}{2})$ concentrates in degrees smaller than $d + r - 1$ for all $b \in B(\G, \mu_{\Std})$. Then the complex $R\Gamma_c(\Sh_{K^p}, \ov \F_{\ell})_{\phi}$ of $\G(\Q_p) \times W_{F}$-representations concentrates in degree smaller than $d + r - 1$.

The assertion for compact Shimura varieties follows from Poincare duality and the dual parameter $\phi^{\vee}$ satisfies the same conditions as $\phi$.
\end{proof}
\begin{remark}
    One should be able to prove an analogue of theorem \ref{itm : beyond generic} for Shimura varieties globalizing fully Hodge-Newton reducible local pairs $(\G, \mu)$ when one has enough understanding about the representations with $\ov \F_{\ell}$-coefficient of $\G(\Q_p)$.
\end{remark}
\subsection{Cohomology of local Shimura varieties for unitary groups with torsion coefficient}\textbf{} \label{itm : local Shimura variety of unitary group}

Let $\G$ be the general unitary group in $2n+1$ variables associated to the unramified \textit{quadratic} extension $F/\Q_p$. As before, denote by $\Sht(\G, b, \mu)$ the local Shimura variety associated to the triple $( \G, b, \mu )$ and denote by $R\Gamma_c(\Sht(\G, b_{\mu}, \mu), \Lambda)$ the cohomology with $\Lambda$-coefficients with $\Lambda \in \{ \ov \F_{\ell}, \ov \Z_{\ell}, \ov \Q_{\ell} \}$ . In this paragraph, we will analyse the supercuspidal part of the complex $R\Gamma_c(\Sht(\G, b_{\mu}, \mu), \Lambda)$ where $\mu$ is a minuscule cocharacter and $b_{\mu}$ is the basic element in $B(\G, -\mu)$. In this subsection, we slightly change the notations, namely we will use $\pi, \phi$ to denote representations and parameters with $\ov\Z_{\ell}$ (or $\ov \Q_{\ell}$)-coefficient and we use $\ov\pi, \ov\phi$ to indicate $\ov \F_{\ell}$-coefficient.

We suppose that $\ell$ is banal for $\G(\Q_p)$. Let $\phi$ be an integral supercuspidal L-parameter of $^L\widehat{\G}(\ov \Q_{\ell})$ and denote by $\ov \phi$ its reduction modulo $\ell$. We suppose further that $\ov \phi$ is also supercuspidal, in particular it does not factor through any proper parabolic subgroup of $^{L}\widehat{\G}(\ov \F_{\ell})$. We suppose further that the group $\tilde{S}_{\phi} := S_{\phi}/S^0_{\phi} \simeq (\Z / 2\Z)^{r-1}$ where $S_{\phi}$ is the centralizer of the $L$-parameter $\phi$ and $r$ is a positive integer \cite[page 16]{BHN}. By the classification of irreducible representation of $\G(\Q_p)$, the corresponding $L$-packet $\Pi_{\ov \Q_{\ell}}(\phi)$ consists of $2^{r-1}$ supercuspidal representations, indexed by $\xi \in \Irr(\tilde{S}_{\phi})$. Let $\pi_{\xi}$ be in $\Pi_{\ov \Q_{\ell}}(\phi)$, by \cite[Proposition 7.14]{DHKM24} we see that $\pi_{\xi}$ is integral and we denote by $\ov \pi_{\xi}$ its reduction modulo $\ell$. It is also an irreducible supercuspidal representation of $\G(\Q_p)$ with $\ov \F_{\ell}$-coefficient \cite[Theorem 4.14, proposition 4.15]{DHKM24}. We want to compute the complex $R\Gamma_c(\Sht(\G, b_{\mu}, \mu), \ov \F_{\ell})[\ov \pi_{\xi}]$ of $\G(\Q_p) \times W_{F}$-representations.

As in the $\ell$-adic case, we globalize the local representations and then use Mantovan's product formula. One difficulty is that a priori, we do not know how to compute the mod $\ell$ cohomology of Shimura varieties. Therefore, we will globalize in such a way that the resulting part of the mod $\ell$ cohomology is concentrated in one degree (by putting conditions on some auxiliary prime $p'$) and then the $\overline{\Z}_{\ell}$-cohomology is torsion free. In particular, the mod $\ell$ part is the reduction modulo $\ell$ of the corresponding part of the $\ell$-adic cohomology of the global Shimura variety. Then the same is true for local Shimura variety and we can deduce a formula for the mod $\ell$ cohomology of the local Shimura variety. 

We consider some simple Shimura variety of type $A$ in the terminology of Kottwitz-Harris-Taylor as in previous paragraph. We suppose that the general unitary group $\textbf{G}$ defined over $\Q$ satisfies $\textbf{G}(\Q_p) \simeq \G(\Q_p)$.

We choose a prime $p'$ that splits in $E$ so that $\textbf{G}(\Q_{p'}) \simeq \GL_n(\Q_{p'}) \times \Q^{\times}_{p'}$. Suppose that $\ell$ is also banal for $\GL_n(\Q_{p'})$. As before, for each $K^{p, p'}$ sufficiently small open compact subgroups of $ \textbf{G}(\A^{\infty, p, p'})$, we can consider the system $\Sh_{K^{p, p'}}$. Similarly we have a complex $R\Gamma_c(\Sh_{K^{p, p'}}, \Lambda)$ of admissible $\textbf{G}(\Q_p) \times \textbf{G}(\Q_{p'}) \times W_{F}$-representations. The role of this auxiliary prime $p'$ is to help us extract a direct part concentrating in the middle degree of the cohomology of the Shimura variety.

We choose a supercuspidal representation $\ov \pi_{p'}$ in $\Rep_{\ov \F_{\ell}}(\GL_n(\Q_{p'}))$ whose Fargues-Scholze parameter is given by $\ov \phi'$. Therefore, by theorem \ref{itm : torsion vanishing}, we see that the complex $R\Gamma_c(\Sh_{K^{p, p'}}, \ov \F_{\ell})^{[\ov\phi, \ov\phi']}$ is concentrated in degree $d = \langle 2\rho_{\G}, \mu \rangle$ where we use the superscript $(-)^{[\ov\phi, \ov\phi']}$ to denote the projection to the direct subcategory of $\Rep_{\ov \F_{\ell}} ( \G(\Q_p) \times \GL_n(\Q_{p'})) $ forming by the blocks containing the representations $\pi \times \pi'$ whose Fargues-Scholze parameters are given by $\phi \times \phi'$ \cite[Theorem 4.22]{DHKZ}. Thus $R\Gamma_c(\Sh_{K^{p, p'}}, \ov \Z_{\ell})^{[\ov\phi, \ov\phi']}$ is torsion free, concentrated in degree $d$ and its reduction modulo $\ell$ is $R\Gamma_c(\Sh_{K^{p, p'}}, \ov \F_{\ell})^{[\ov\phi, \ov\phi']}$. By taking $K_{p'}$-invariant with respect to an appropriate compact subgroup $K_{p'} \subset \textbf{G}(\Q_{p'})$, the same conclusion holds true for $R\Gamma_c(\Sh_{K^{p}}, \ov \Z_{\ell})^{[\ov\phi, \ov\phi']}$ and $R\Gamma_c(\Sh_{K^{p}}, \ov \F_{\ell})^{[\ov\phi, \ov\phi']}$ where $K^p = K_{p'}K^{p, p'} \subset \textbf{G}(\A^{\infty, p})$.

Now by applying \cite[Theorem 1.12]{HL}, the complex $R\Gamma_c(\Sh_{K^{p}}, \ov \F_{\ell})$ has a filtration as a complex of $ \mathrm{\G}(\Q_p) \times W_{F} $ representations with graded pieces isomorphic to $ (R\Gamma_c(\G, b, \mu) \otimes^{\bL}_{\mathcal{H}(\G_b)}R\Gamma_c(\Ig_{b, K^{p}}, \ov \F_{\ell}))[2d_b-d](-\tfrac{d}{2})$ and where $b$ runs in the Kottwitz set $B(\G, -\mu)$. Since the Fargues-Scholze parameter $\phi$ is supercuspidal and the Hecke operators preserve Fargues-Scholze parameters, we deduce that the non-basic strata do not contribute to $R\Gamma_c(\Sh_{K^{p}}, \ov \F_{\ell})^{[\ov\phi, \ov\phi']}$. Thus we have
\begin{equation} \phantomsection \label{itm : equation F}
R\Gamma_c(\Sh_{K^{p}}, \ov \F_{\ell})^{[\ov\phi, \ov\phi']}\simeq i^*_1\T_{ - \mu} (i^{\rm ren}_{b_{\mu}!}R\Gamma_{c}(\Ig_{b_{\mu}, K^{p}}, \ov \F_{\ell}))^{[\ov\phi, \ov\phi']}[-d](-\tfrac{d}{2})).    
\end{equation}

Moreover, the basic Igusa variety is of zero dimension, the complex $R\Gamma_{c}(\Ig_{b_{\mu}, K^{p}}, \ov \F_{\ell})) $ is concentrated in degree $0$ and then $R\Gamma_{c}(\Ig_{b_{\mu}, K^{p}}, \ov \Z_{\ell}))$ is torsion free and its reduction modulo $\ell$ is given by $R\Gamma_{c}(\Ig_{b_{\mu}, K^{p}}, \ov \F_{\ell}))$. Similarly $R\Gamma_c(\Sh_{K^{p}}, \ov \Z_{\ell})^{[\ov\phi, \ov\phi']}$ is also concentrated in degree $d$ and torsion free. 

More generally, the above argument still works when we replace $\ov \F_{\ell}$ by $\Lambda = \Z^{\textrm{unr}}_{\ell} / \ell^{a} \Z^{\textrm{unr}}_{\ell}$ since Mantovan's formula and also the Bernstein decomposition are still valid in this context \cite[Theorems 8.5.7, 8.6.3]{DHKZ}, \cite[Theorem 4.22]{DHKM24}. The Hecke operator $\T_{-\mu}$ on $\Dlis(\Bun_{\G}, 
\ov \Z_{\ell})$ preserves limits, colimits and is $\ov \Z_{\ell}$-linear. Thus we can deduce that
\begin{equation}  \phantomsection \label{itm : equation Z}
 H^d_c(\Sh_{K^{p}}, \ov \Z_{\ell})^{[\ov\phi, \ov\phi']} \simeq i^*_1\T_{ - \mu} (i^{\rm ren}_{b_{\mu}!}H^0_c(\Ig_{b_{\mu}, K^{p}}, \ov \Z_{\ell}))^{[\ov\phi, \ov\phi']}[-d](-\tfrac{d}{2})),  
\end{equation}
in particular, equation (\ref{itm : equation F}) is the reduction modulo $\ell$ of equation (\ref{itm : equation Z}).
\begin{proposition} \phantomsection \label{itm : cohomology of GU with torsion coefficient}
  Let $\ell > 2$ be a banal prime for $\G$. Then we have the following identity
 \[
i^*_1\T_{-\mu}(i^{\rm ren}_{b_{\mu}!} \ov \pi_{\xi} ) \simeq \bigoplus_{\xi' \in \Irr(\tilde{S}_{\phi})} \ov \pi_{\xi'} \boxtimes \Hom_{\tilde{S}_{\phi}}(\xi \otimes (\xi')^{\vee}, r_{-\mu} \circ \ov\phi).
\]   
\end{proposition}
\begin{proof}
 Denote by $\textbf{I}$ the general unitary group defined over $\Q$ such that $\textbf{I}(\R)$ is compact modulo the center and $\textbf{I}(\Q_{p}) \simeq \textbf{G}(\Q_p)$ for every prime number $p$. We have a description of the point set Igusa variety $\Ig_{b_{\mu}}$ by \cite[Lemma 7.3]{Shi4} then up to some multiplicity we have an equality
 \[
 \mathop{\mathrm{colim}}_{\overrightarrow{K^{p} \rightarrow \{1 \}}} H^0_c(\Ig_{b_{\mu}, K^{p}}, \ov \Q_{\ell}) = C_c^{\infty}(\mathbf{I}(\Q) \backslash \mathbf{I} (\A^{\infty}), \ov \Q_{\ell}).
 \]
 
 The latter space is the space of automorphic forms of type $1$ at infinity of $\mathbf{I}(\A)$. One could also see \cite[equation (5.14)]{ABM} for a formula of the cohomology of Igusa varieties in this case. 

 By the Kottwitz conjecture for basic Rapoport-Zink spaces of unitary type \cite[Theorem 6.1]{BMN}, we have
\begin{equation} \phantomsection \label{itm : Q equation}
 i^*_1\T_{-\mu}(i^{\rm ren}_{b_{\mu}!} \pi_{\xi} ) \simeq \bigoplus_{\xi' \in \Irr(\tilde{S}_{\phi})} \pi_{\xi'} \boxtimes \Hom_{\tilde{S}_{\phi}}(\xi \otimes (\xi')^{\vee}, r_{-\mu} \circ \phi).   
\end{equation}

The representations $\pi_{\xi}$ for $\xi \in \Irr(\tilde{S}_{\phi})$ are $\ell$-integral, that means each $\pi_{\xi}$ has a $\G(\Q_p)$-stable $\ov \Z_{\ell}$-lattice. We want to show that in equation $(\ref{itm : Q equation})$, the Hecke operator $\T_{-\mu}$ preserves these stable $\ov \Z_{\ell}$-lattices. For that, we use equation $( \ref{itm : equation Z} )$. In fact, with this identity, we can follow the arguments as in \cite[Appendice]{KH1} to globalize the representations $\pi_{\xi}$ (up to some unramified twist $\chi$) into automorphic representations $\Pi$ of $\mathbf{I}(\A)$ and compute $i^*_1\T_{-\mu}(i^{\rm ren}_{b_{\mu}!} \displaystyle \bigoplus_{\xi' \in \Irr(\tilde{S}_{\phi})} \pi_{\xi'} \otimes \chi )$ by taking the $[\Pi^{\infty, p}]$-isotypical part of the  equation $(\ref{itm : equation Z})$ after inverting $\ell$. Note that $\pi_{\xi'}$ and $\pi_{\xi'} \otimes \chi$ are both $\ell$-integral, we deduce that $\chi$ takes value in $\ov \Z^{\times}_{\ell}$. Thus, by \cite[Lemma 5.13]{BMN} we can deduce the formula for $i^*_1\T_{-\mu}(i^{\rm ren}_{b_{\mu}!} \displaystyle \bigoplus_{\xi' \in \Irr(\tilde{S}_{\phi})} \pi_{\xi'} )$ by twisting by $\chi$. Hence we get the following formula 
\begin{equation} \phantomsection \label{itm : equation Z sum}
 i^*_1\T_{-\mu}(i^{\rm ren}_{b_{\mu}!} \displaystyle \bigoplus_{\xi' \in \Irr(\tilde{S}_{\phi})} \pi_{\xi'} ) \simeq \bigoplus_{\xi' \in \Irr(\tilde{S}_{\phi})} \pi_{\xi'} \boxtimes r_{-\mu} \circ \phi.   
\end{equation}

Note that in the equality $(\ref{itm : equation Z sum})$, the operator $\T_{-\mu}$ preserves the stable $\G(\Q_p)$-lattice since this equality comes from taking some direct factor of equality $(\ref{itm : equation Z})$ after inverting $\ell$ and all the terms in $(\ref{itm : equation Z})$ are torsion free. Now by taking reduction modulo $\ell$ of the stable lattices, we conclude by the Brauer-Nesbitt principle of Vignéras \cite[page 11]{DHKM24} that
\[
i^*_1\T_{-\mu}(i^{\rm ren}_{b_{\mu}!} \ov \pi_{\xi} ) \simeq \bigoplus_{\xi' \in \Irr(\tilde{S}_{\phi})} \ov \pi_{\xi'} \boxtimes \Hom_{\tilde{S}_{\phi}}(\xi \otimes (\xi')^{\vee}, r_{-\mu} \circ \ov\phi),
\]
and this is what we want to prove.
\end{proof}
\begin{remark}
   This result is an analogue of the Kottwitz conjecture for supercuspidal part of the cohomology of basic Rapoport-Zink space with $ \ov \F_{\ell}$-coefficients. However, it is not clear to us that the representations $\ov \pi_{\xi}$'s, $\xi \in \Irr(\tilde{S}_{\phi})$ are pairwise non-isomorphic.
\end{remark}

\begin{corollary} \phantomsection \label{itm : supercuspidal torsion vanishing unitary}
Let $\ell \neq p$ be an odd prime number and $\ov \phi$ be as above and let $\Sh_{K_p}$ be a Shimura variety of type A as defined earlier in this paragraph, but non necessarily compact. Denote by $R\Gamma_c(\Sh_{K^p}, \ov \F_{\ell})^{[\phi]}$ the projection of the complex of admissible $\G(\Q_p)$-representations $R\Gamma_c(\Sh_{K^p}, \ov \F_{\ell})$ to the direct sub-category of the derived category of $\Rep_{\ov \F_{\ell}}(\G(\Q_p))$ forming by the blocks containing $\ov\pi_{\chi}$, $\chi \in \Irr(S_{\phi})$. Then the complex $R\Gamma_c(\Sh_{K^p}, \ov \F_{\ell})^{[\phi]}$ is concentrated in degree $d = \dim \Sh_{K^p}$.  
\end{corollary}
\begin{proof}
  The same argument as in the proof of theorem \ref{itm : torsion vanishing} still works here. First of all we can use Mantovan's formula to relate the complex $R\Gamma_c(\Sh_{K^p}, \ov \F_{\ell})$ with the Hecke operators $\T_{-\mu}$ and the complexes $R\Gamma_{c- \partial}(\Ig_{b, K^p}, \ov \F_{\ell})$ where $b \in B(\G, -\mu)$. Since the minimal compactification of Igusa varieties are affine \cite{CS, CS1}, we can deduce that the complex $R\Gamma_{c-\partial}(\Ig_{b, K^p}, \ov \F_{\ell})$ is concentrated below degree $d_b = \langle \nu_b, 2\rho_{\G} \rangle$ for $b \in B(\G, -\mu)$ \cite[Proposition 3.7]{HL}. However, the contribution of non-basic strata vanishes since $\T_{-\mu}$ preserves Fargues-Scholze parameters and if $\pi_b$ is an irreducible representation in $\Rep_{\ov \F_{\ell}}(\G_b(\Q_p))$ then its Fargues-Scholze parameter is not supercuspidal. Moreover, the Igusa variety corresponding to the basic stratum is zero dimension and thus the complex $R\Gamma_{c-\partial}(\Ig_{b, K^p}, \ov \F_{\ell})$ concentrated in degree $0$. In this case, one can also deduce that $R\Gamma_c(\Sh_{K^p}, \ov \F_{\ell})^{[\phi]}$ is concentrated in degree $d$ by proposition \ref{itm : cohomology of GU with torsion coefficient}.  
\end{proof}

\begin{remark}
     In a forthcoming work, Patrick Daniels, Pol van Hoften, Dongryul Kim and Mingjia Zhang will prove a version of Mantovan's formula for Shimura varieties of abelian type. One can combine it with the Kottwitz conjecture in the $\ell$-adic setting \cite{Ham, Pen} to prove an analogue of the above results for $\GSp_4$ and special orthogonal and unitary groups. 
\end{remark}


\bibliographystyle{amsalpha}
\bibliography{publications}
\end{document}